\documentclass[10pt]{amsart}
\usepackage{srcltx,graphicx,amsmath,amssymb,amsthm,mathrsfs,bm}
\usepackage[unicode,colorlinks,linkcolor=blue,hyperindex]{hyperref}
\usepackage{indentfirst} 
\usepackage{enumerate} 
\usepackage{enumitem}
\usepackage{graphicx}
\usepackage{amsmath} 
\usepackage{cases} 
\usepackage[mathscr]{euscript}
\usepackage{url, breakurl}
\usepackage{multirow}

\usepackage{subfigure}
\makeatletter
\renewcommand{\p@subfigure}{\thefigure}
\makeatother

\numberwithin{equation}{section} \numberwithin{figure}{section}

\newcommand\rd{{\mathrm{d}}} 

\newcommand\rmm{{\mathrm{m}}}
\newcommand\rs{{\mathrm{s}}}

\def\bA{\mathbf{A}}

\def\bC{\mathbf{C}}
\def\bD{\mathbf{D}}
\def\bF{\mathbf{F}} 

\def\bG{\mathbf{G}}
\def\bI{\mathbf{I}}
\def\bK{\mathbf{K}}
\def\bR{\mathbf{R}}
\def\bS{\mathbf{S}}

\def\bU{\mathbf{U}} 
\def\bM{\mathbf{M}}
\def\bV{\mathbf{V}} 
\def\bI{\mathbf{I}}
\def\bW{\mathbf{W}}
\def\bu{\mathbf{u}} 

\def\bzero{\mathbf{0}}
\def\cL{\mathcal{L}}

\def\cO{\mathcal{O}}

\def\bx{\mathbf{x}} 
\def\be{\mathbf{e}}
\def\bg{\mathbf{g}}
\def\bc{\mathbf{c}}
\def\bX{\mathbf{X}}

\def\br{\mathbf{r}}

\def\bLambda{\mathbf{\Lambda}}
\def\bvarphi{\vec{\varphi}}
\def\blambda{\vec{\lambda}}
\def\bphi{\vec{\phi}}
\def\balpha{\vec{\alpha}}
\def\bbeta{\vec{\beta}}

\def\diag{{\text{diag}}}

\def\R{\mathbb R}

\def\RmOne{{\bf{\rm{\uppercase\expandafter{\romannumeral1}}}}}
\def\RmTwo{{\bf{\rm{\uppercase\expandafter{\romannumeral2}}}}}

\renewcommand{\vec}[1]{\mbox{\boldmath $#1$}}

{\theoremstyle{remark} \newtheorem{remark}{Remark}}

\newtheorem{theorem}{Theorem}[section]
\newtheorem{lemma}{Lemma}[section]

\makeatletter
\newcommand{\Rmnum}[1]{\expandafter\@slowromancap\romannumeral #1@}
\makeatother

\begin{document} 

\title[Sediment Transport]{A Second Order Time Homogenized Model for
  Sediment Transport}

\author[Y.-C. Jiang]{Yuchen Jiang} 
\address{School of Mathematical Sciences, Peking University, Beijing,
  P. R. China.} 
\email{jiangyuchen@pku.edu.cn}

\author[R. Li]{Ruo Li}
\address{HEDPS \& CAPT, LMAM \& School of Mathematical Sciences,
  Peking University, Beijing, P. R. China.}
\email{rli@math.pku.edu.cn}

\author[S.-N. Wu]{Shuonan Wu}
\address{Department of Mathematics, The Pennsylvania State University,
  University Park, PA, 16802, USA}
\email{wsn1987@gmail.com}

\date{}
\maketitle

\begin{abstract}
  A multi-scale method for the hyperbolic systems governing sediment
  transport in subcritical case is developed. The scale separation of
  this problem is due to the fact that the sediment transport is much
  slower than flow velocity. We first derive a zeroth order
  homogenized model, and then propose a first order correction. It is
  revealed that the first order correction for hyperbolic systems has
  to be applied on the characteristic speed of slow variables in one
  dimensional case. In two dimensional case, besides the
  characteristic speed, the source term is also corrected.  We develop
  a second order numerical scheme following the framework of
  heterogeneous multi-scale method. The numerical results in both one
  and two dimensional cases demonstrate the effectiveness and
  efficiency of our method.
\end{abstract}

\smallskip
%\noindent 


%%%%%%%%%%%%%%%%%%%%%%%%%%%%%%%%%%%%%%%%%%%%%%%%%%
%% Introduction 
%%%%%%%%%%%%%%%%%%%%%%%%%%%%%%%%%%%%%%%%%%%%%%%%%%
The sediment transport in flow is often modelled by a scalar
convective equation coupled with the shallow water equations.
Typically, the morphodynamic process by sediment transport is an
extremely slow process \cite{wilcock2009sediment,
cunge1980practical} in term of the flow velocity in the model, while
the changes in the topography are usually of practical interests.
Since interesting changes in the topography can only be produced by
the continual erosion of the flow for a long period of time, the model
is provided a scale separation in time.  In the coupled model of the
sediment transport equation and the shallow water equations, the
system of shallow water equations is of standard formation, with the
contribution from the riverbed elevation. In the past decades, the
spatial variation of the riverbed elevation has been extensively
investigated \cite{simons1992sediment}. The expressions of the
sediment tranport flow are usually proposed for granular non-cohesive
sediments and quantified empirically \cite{grass1981sediment,
meyer1948formulas, fernandez1976erosion, van1984sediment,
van1993principles, nielsen1992coastal, camenen2005general}. 

%%The time derivative of the
%%non-dimensional riverbed elevation is often formulated as the product
%%of an erosive rate coefficient and a term with flow velocity. The
%%erosive rate coefficient is significantly less than the magnitude of
%%the typical flow velocity to depict the slow evolving of the riverbed
%%elevation. 

In this paper, we consider the following widely used system to model
the sediment transport in one dimensional case
\begin{equation} \label{equ:1d-sediment}
\left\{
\begin{aligned}
&\partial_t h + \partial_x (hu) = 0,  \\
&\partial_t (hu) + \partial_x (hu^2+\frac{1}{2}gh^2) = -gh\partial_x
B, \\
&\partial_t B + \xi \partial_x q_b = 0,
\end{aligned}
\right.
\end{equation}
where the first two equations are the shallow water equations, and the
last equation is the Exner equation \cite{exner1920zur,
exner1925uber} involving a sediment transport flux. In this system,
$h$ is the water depth, $u$ is the vertically averaged flow velocity
along the $x$ direction, and $B$ is the riverbed elevation.  
$q_b$ denotes the volumetric bedload sediment transport discharge.
Different from the standard shallow water equations, the riverbed
elevation $B$ depends on time $t$ as well. The parameters involved are
the gravity constant $g$ and $\xi = (1-\gamma)^{-1}$ with $\gamma$ the
porosity of sediment layer. Grass \cite{grass1981sediment} proposed
one of the simplest formulation of $q_b$ as 
\begin{equation} \label{equ:Grass}
q_b = A_g u |u|^{m-1} \qquad 1 \leq m \leq 4, A_g\in (0,1].
\end{equation}

In practice, estimates of the bedload transport rate are mainly based
on the modeling of bottom shear stress $\tau_b$ and a non-dimensional
parameters $q_b^*$ as 
$$ 
q_b = q_b^* \sqrt{(s - 1)gd_s^3}, 
$$ 
where $s$ is the density ratio between sediment $\rho_s$ and water
$\rho$, and $d_s$ is the median diameter of sediment. The
non-dimensional form of bottom shear stress, which is also called
Shields parameter, is defined as  
$$
\tau_b^* = \frac{\tau_b}{(s - 1)\rho g d_s}.
$$
A variety of models \cite{grass1981sediment, meyer1948formulas,
fernandez1976erosion, van1984sediment, van1993principles,
nielsen1992coastal, camenen2005general} have been often applied to
build the relationship between $q_b^*$ and Shields parameter. In
this paper, one of the most commonly-used models proposed by Meyer,
Peter and M\"{u}ler \cite{meyer1948formulas} is considered: 
\begin{equation} \label{equ:MPM}
q_b^* = 8(\tau_b^* - \tau_{cr}^*)_+^{3/2} := 
\begin{cases}
8(\tau_b^* - \tau_{cr}^*)^{3/2} & \text{if}~\tau_b^* > \tau_{cr}, \\ 
0 & \text{otherwise}.
\end{cases}
\end{equation}
Consider the Darcy-Weisbach formula for the bottom shear stress in the
context of laminar flows, we have 
$$ 
\tau_b = \rho gh S_f \qquad S_f = \frac{fu|u|}{8gh}, 
$$ 
where $f$ is the Darcy-Weisbach's coefficient. Consequently, the
original Meyer-Peter-M\"{u}ler model \eqref{equ:MPM} can be reduced to
the following expression, 
\begin{equation} \label{equ:MPM-Darcy-Weisbach}
q_b = 8\sqrt{(s-1)g d_s^3} \frac{u}{|u|} \left( \frac{f}{8(s-1)gd_s}|u|^2 -
    \tau_{cr}^*\right)^{3/2}.
\end{equation}

The Grass model \eqref{equ:Grass} and Meyer-Peter-M\"{u}ler model
\eqref{equ:MPM-Darcy-Weisbach} can be recast into a unified fashion as  
\begin{equation} \label{equ:unified-sed}
\xi q_b = \varepsilon u \tilde{q}_b(|u|),
\end{equation}
where 
\begin{itemize}
\item For Grass model \eqref{equ:Grass}:
\begin{equation} \label{equ:sed-Grass}
\varepsilon = \xi A_g \qquad \text{and} \qquad \tilde{q}_b(|u|) =
|u|^{m-1} \qquad 1\leq m \leq 4. 
\end{equation}
\item For Meyer-Peter-M\"uler model \eqref{equ:MPM-Darcy-Weisbach}:
\begin{equation} \label{equ:sed-MPM}
\varepsilon = \frac{\xi}{(s-1)g}\sqrt{\frac{f^3}{8}} \qquad \text{and}
\qquad \tilde{q}_b(|u|) = \frac{1}{|u|}(|u|^2 - u_{cr}^2)_+^{3/2}. 
\end{equation}
Here the critical velocity $u_{cr} = \sqrt{8(s-1)gd_s \tau_{cr}^*/f}$. We
note that this model is one of the commonly-used for rivers and
channels with slope lower than $2\%$, see \cite{soulsby1997dynamics}
for details.
\end{itemize}
For the case that the flow-sediment interaction is low, $\varepsilon$
is far less than the magnitude of typical flow velocity to depict the
scale separation in time. In this sense, we call the parameter
$\varepsilon$ in \eqref{equ:unified-sed} the {\em time scaling
parameter}. 

%% 2D formulation %% 
Similarly, the governing system in two dimensional case is formulated as
\begin{equation} \label{equ:2d-sediment}
  \dfrac{\partial}{\partial t}\begin{pmatrix}
    h \\ hu \\ hv \\ B
  \end{pmatrix} +
  \dfrac{\partial}{\partial x}\begin{pmatrix}
    hu \\ hu^2+\frac{1}{2}gh^2 \\ huv \\ \varepsilon u\tilde{q}_b(|\bu|)
  \end{pmatrix} + 
  \dfrac{\partial}{\partial y}\begin{pmatrix}
    hv \\ huv \\ hv^2+\frac{1}{2}gh^2 \\ \varepsilon v\tilde{q}_b(|\bu|)
  \end{pmatrix} = 
  \begin{pmatrix}
    0 \\ -ghB_x \\ -ghB_y \\ 0
  \end{pmatrix},
\end{equation}
where $\bu=(u,v)^T$ is the vertically averaged flow velocity along
the $x$ and $y$ direction.

%% Part III: numerical method: decoupled method --> practical but have
%% drawbacks --> coupled method (stable)

The scale separation in time brings us serious difficulty in carrying
out numerical simulation for \eqref{equ:1d-sediment} or
\eqref{equ:2d-sediment}. Currently, there are two classifications of
the numerical methods for this problem: {\it coupled method} and {\it
decoupled method}. The decoupled method is suitable for the case that
the topography changes much slower than the flow, which results in a
quasi-steady water motions with respect to the topography. It was
pioneered by Cunge et al. \cite{cunge1980practical}, and has been
widely used in industry
\cite{simons1992sediment,wu2008computational,juez20142d}
on account of its high computational efficiency. There are, however,
two drawbacks associated with decoupled method, including the
instability when updating the riverbed with traditional scheme (e.g.
Lax-Wendroff scheme) \cite{hudson2001numerical, hudson2005numerical,
cordier2011bedload} and the low accuracy in terms of $\varepsilon$. 
%%the limited capacity to the problems with strong interaction.

With the purpose of overcoming the above drawbacks of the decoupled
method, much effort has been devoted to develop numerical schemes by
coupling the hydrodynamics and morphodynamics. 
The numerical techniques developed in this fold include the Roe-type
scheme \cite{hudson2001numerical, hudson2003formulations,
hudson2005high, di2009two, murillo2010exner, de2011duality,
serrano2012finite}, the second order LHLL method
\cite{fracarollo2003godunov},
the balanced finite volume WENO scheme \cite{crnjaric2004balanced}, 
the state reconstructions for non-conservation hyperbolic systems
\cite{castro2008sediment}, the relaxation approximation
\cite{delis2008relaxation}, the second order SRNHS scheme
\cite{benkhaldoun2009solution}, the WAF method
\cite{postacchini2012multi}, the McCormack scheme
\cite{briganti2012efficient}, etc. One of the main difficulties for
coupled method comes from the fact that the interaction between fluid
and sediment is usually too weak \cite{castro2008sediment}, and the
simulations have to carry out for a long time. Therefore, it is
necessary to use high order numerical schemes to depress the numerical
dissipation. At the same time, one has to suffer by the small time
step size restricted by the hydrodynamic time scale for the explicit
scheme. Recently, the linearized implicit method
\cite{bilanceri2012linearized} was proposed to enlarge the time step
size.

The multi-scale method we developed in this paper for the sediment
transport problem looks somewhat like the decoupled method.  In the
subcritical case, we first reformulate the sediment transport model to
a non-conservative scalar equation (or the zeroth order equation), which
preserves the hyperbolicity of the coupled system. We assume that the
morphodynamic process is governed by a limiting equation (or the
zeroth order equation) when the time scaling parameter
$\varepsilon$ tends to zero.  For one dimensional case, the resulting
limiting equation is exactly same as the result of De Vries
\cite{devries1973river}.  For two dimensional case, the limiting
equation can be derived subsequently as a reasonable extension of the
one dimensional case, with a difference in the source term.  When
developing the numerical method in the framework of Heterogeneous
Multiscale Method (HMM) \cite{weinan2003heterogeneous}, the steady
state solver for shallow water equations is adopted as the
micro-solver, and the macro-solver for the morphodynamic time scale
dynamics is applied.  Therefore, the first drawback of decoupled
method can be avoided to some extent by taking a stable scheme for
zeroth order model as the macro-solver when updating the riverbed
topography.

For practical problems that the time scaling parameter $\varepsilon$ is
not so extremely close to zero \cite{castro2008sediment}, a first
order correction have to be made to improve the order of accuracy for
$\varepsilon$. In other words, the model error will be
$\mathcal{O}(\varepsilon^2)$ after the first order correction, which
does alleviate the second drawback of decoupled method for the
low-interaction between flow and sediment ($\varepsilon \ll 1$).

The basic idea of the first order correction is as follows. With the
help of the zeroth order model, two correction terms are applied to
the steady state of shallow water equations: the dynamic term used to
balance the flux and source term, and the $\mathcal{O}(\varepsilon)$
term. In light of this correction for flow variables, the
characteristic speed of riverbed is corrected in one dimensional case,
and both the characteristic speed and the source term of riverbed are
corrected in two dimensional case. As for the numerical algorithm, an
interesting observation is that one can update the flow variables
while updating the riverbed topography in the macro time step of
morphodynamic time scale. This {\it fast variables correction}
improves not only the computing accuracy but also the stability of the
numerical multi-scale method.

The rest part of this paper is organized as follows. In section
\ref{sec:1d-linear}, we apply our multi-scale method to one
dimensional linear model to describe its key ingredients.  In section
\ref{sec:1d}, we give the homogenization of the model in one
dimensional case and in section \ref{sec:2d}, we give the two
dimensional model subsequently. In section \ref{sec:scheme}, we
develop the corresponding numerical method. The numerical results are
in section \ref{sec:num} and a short conclusion remarks in section
\ref{sec:conclusion} close the main text.

%%%%%%%%%%%%%%%%%%%%%%%%%%%%%%%%%%%%%%%%%%%%%%%%%%
%% Section 2: Linear model 
%%%%%%%%%%%%%%%%%%%%%%%%%%%%%%%%%%%%%%%%%%%%%%%%%%
\section{One Dimensional Linear Model} \label{sec:1d-linear}
In order to clarify our method for nonlinear hyperbolic systems
\eqref{equ:1d-sediment} and \eqref{equ:2d-sediment}, we first 
present our basic idea by studying
the one dimensional linear hyperbolic system as follows
\begin{equation} \label{equ:linear}
  \left\{
    \begin{aligned}
      &\bU_t + \bA \bU_x = -\bg B_x, \\
      &B_t + \varepsilon \bc^T\bU_x = 0.
    \end{aligned}
  \right.
\end{equation}
where $\bU=(u_1,\cdots,u_l)^T\in\R^l$ are fast variables, $B\in\R$ is
slow variable, $\bg,\bc$ are two constant vectors in $\R^l$. 
The constant matrix $\bA \in\R^{l\times l}$ satisfies
\begin{equation} \label{equ:linear-A}
  \bA = \bX^{-1}\bLambda, \bX \qquad \bLambda = \text{diag}\{
    \lambda_1, \cdots, \lambda_l
  \}.
\end{equation}
With the purpose of scale separation on the characteristic speeds of
\eqref{equ:linear}, it requires $|\lambda_i|\gg\varepsilon$, thus
$\bA$ is invertible.  Moreover, we assume $\bA$ has only discrete
eigenvalues (i.e. the algebraic multiplicity is one).  We emphasize
that $t$ represents the fast time scale, and $\tau=\varepsilon t$
represents the slow time scale in the context of this paper.

Let
\begin{equation}\label{mat:C-Ceps0}
\bC = \begin{pmatrix} \bA & \bg \\ 0 & 0 \\ \end{pmatrix}, \qquad 
\bC^\varepsilon = \begin{pmatrix} \bA & \bg \\ \varepsilon \bc^T & 0 \\ 
\end{pmatrix}, \qquad \bV = 
\begin{pmatrix}
\bU \\ B
\end{pmatrix}.
\end{equation}
Then, the original system \eqref{equ:linear} can be recast as:
\[
\bV_t+\bC^\varepsilon \bV_x=0.
\]
It follows from \eqref{equ:linear-A} that $\bC=\bK^{-1}\bD\bK$, where
\[
  \bK=\begin{pmatrix} \bX & \bLambda^{-1}\bX \bg \\ 0 & 1 
  \end{pmatrix}, \qquad 
  \bD=\begin{pmatrix} \bLambda & 0\\0 & 0\end{pmatrix}.
\]
By the perturbation theory of discrete eigenvalues and eigenvectors in
\cite{wilkins1965eigenvalue}, 
  \begin{equation}\label{decom-0}
    \bC^\varepsilon = (\bK^\varepsilon)^{-1} 
     \bD^\varepsilon \bK^{\varepsilon},
  \end{equation}
where 
\[
 \bK^\varepsilon = 
 \begin{pmatrix} \bX+\varepsilon\hat{\bX} &
 \bLambda^{-1}\bX \bg+\varepsilon\hat{\balpha} \\
 \varepsilon\hat{\bbeta}^T & 1+\varepsilon\hat{\theta} \\
 \end{pmatrix}, \qquad
 \bD^\varepsilon=\begin{pmatrix} 
 \bLambda + \varepsilon \hat{\bLambda} & 0 \\ 0 & \varepsilon\mu\\ 
 \end{pmatrix}.
\]
%%%%%%%%%%%%%%%%%%%%%%%%%%%%%%%%%%%%%%%%%%%%%%%%%% 
%% Zeroth order model 
%%%%%%%%%%%%%%%%%%%%%%%%%%%%%%%%%%%%%%%%%%%%%%%%%%
\subsection{Zeroth order model}
The first step of our method is to predict the slow variable $B$.
Eliminating the spatial derivative terms of fast variables in
\eqref{equ:linear}, we have
\[
  B_t + \varepsilon\bc^T\bA^{-1}(-\bg B_x - \bU_t) = 0, 
\]
that is
\begin{equation} \label{equ:linear-0tau}
  B_\tau - \bc^T\bA^{-1}\bg B_x - \varepsilon \bc^T\bA^{-1} \bU_\tau =
  0,
\end{equation}
The {\it zeroth order model} (or limiting equation) for the linear
hyperbolic system \eqref{equ:linear} is derived by $\varepsilon\to0$,
namely
\begin{equation} \label{equ:linear-0}
  B^{(0)}_\tau + \lambda_B^{(0)} B^{(0)}_x = 0,
\end{equation}
where $\lambda_B^{(0)} = -\bc^T\bA^{-1}\bg$. We note that the zeroth
order model \eqref{equ:linear-0} is established only on the hypothesis
that $\bU_\tau$ is bounded. In particular, this hypothesis holds when
the fast dynamics has a steady state (up to $\cO(\varepsilon)$) with
fixed slow variable in the initial state.  At this point, the zeroth
order model \eqref{equ:linear-0} provides a prediction for $B$ on the
interval $[0,\Delta\tau]$ as
\[
  B^{(0)}(x,\tau) = B(x-\lambda_B^{(0)} \tau,0), \qquad 
  \tau \in[0, \Delta \tau].
\]
In light of the prediction, a modified fast dynamics $\bU^{(1)}$ can
be derived which is slightly different from that with fixed slow
variable. Our next step is to calculate the difference between the
steady variable $\bU^{(0)}(x,\tau)$ and $\bU^{(1)}(x,\tau)$ to get the
correction terms, and by this way the correction model for the
riverbed can be obtained. More precisely, we will show this process step by
step strictly for the linear system. First, we have the following
lemma for zeroth order model \eqref{equ:linear-0}: 

%%% Lemma 2.1: Zeroth-order %%%
\begin{lemma}\label{prop:1-conv}
Suppose that $\bA\in \R^{l\times l}$ has only real and discrete
eigenvalues and all eigenvalues satisfy $|\lambda_i|>\delta \gg
\varepsilon$ for some positive $\delta$. The initial dynamics
satisfy 
\begin{enumerate}
\item $\bA\bU_x(x,0)+\bg B_x(x,0) = \cO(\varepsilon)$. 
\item $B^{(0)}(x,0) = B(x,0) \in W^{1,\infty}(\R)$.
\end{enumerate} 
Then 
\[
\|B(x,t) - B^{(0)}(x,t)\|_{\infty} = \cO(\varepsilon) \qquad
\text{for }~ t\sim \cO(\varepsilon^{-1}).
\]
Moreover, if $B(x,0) \in W^{2,\infty}(\mathbb{R})$ and $\bA\bU_{xx}(x,0)+\bg
B_{xx}(x,0)=\cO(\varepsilon)$, then 
\[
\|B_x(x,t) - B^{(0)}_x(x,t)\|_{\infty} = \cO(\varepsilon) \qquad 
\text{for }~ t\sim \cO(\varepsilon^{-1}).
\]
\end{lemma}
%%% proof of Lemma 2.1 %%%
\begin{proof}
    See Appendix \ref{proof:1-conv}.
\end{proof}
\subsection{First order model}
Now we consider the case in which $\varepsilon$ is small but does not
tend to zero. The basic idea here is to improve the accuracy of the
weakly
coupled term $-\varepsilon \bc^T \bA^{-1}\bU_\tau$ by zeroth order
model \eqref{equ:linear-0}. More specifically, we first substitute $B$
in \eqref{equ:linear} by the solution of the zeroth order model, i.e.
\begin{equation}\label{equ:U1}
\bU^{(1)}_t + \bA \bU^{(1)}_x = -\bg B^{(0)}_x(x,t).
\end{equation}
Then, the corrected equation of $B$ can be derived by
\eqref{equ:linear-0tau} and \eqref{equ:U1} that 
\begin{equation}\label{equ:B1-hat}
  \hat{B}^{(1)}_\tau -\bc^T\bA^{-1}\bg
  \hat{B}^{(1)}_x-\varepsilon\bc^T\bA^{-1}\bU^{(1)}_\tau=0.
\end{equation}
The following lemma indicates that the error between the solution of
\eqref{equ:B1-hat} and original equations \eqref{equ:linear} is of
$\cO(\varepsilon^2)$ order.

%%% Lemma 2.2: first order, version 1 %%% 
\begin{lemma}\label{lemma:2-conv-1}
Under the same assumptions of Lemma \ref{prop:1-conv}, and assume the
initial dynamics satisfy
\begin{enumerate}
\item $(\bA+\varepsilon \bX^{-1}\bLambda\hat{\bX})\bU_x(x,0) +
(\bg+\varepsilon\bX^{-1}\bLambda\hat{\alpha})B_x(x,0) =\cO(\varepsilon^2)$.
\item $\bA\bU_{xx}(x,0)+\bg B_{xx}(x,0) =\cO(\varepsilon)$.

\item $(\bA + \varepsilon \bA)\bU_x^{(1)}(x,0) + (\bg + \varepsilon
    \bg - \varepsilon\bA^{-1}\bg\bc^T\bA^{-1}\bg)B_x^{(0)}(x,0) =
\cO(\varepsilon^2)$. 
\item $B^{(0)}(x,0) = \hat{B}^{(1)}(x,0) = B(x,0) \in
W^{2,\infty}(\R)$.
\end{enumerate}
Then 
\[
\lVert B(x,t)-\hat{B}^{(1)}(x,t)\rVert_\infty=\cO(\varepsilon^2) \qquad
\text{for }~ t\sim \cO(\varepsilon^{-1}).
\]
\end{lemma}
%%% proof of Lemma 2.2 %%
\begin{proof}
    See Appendix \ref{proof:2-conv-1}.
\end{proof}
%%% end of lemma 2.2 %%%

%%% Remark: high order model %%%
\begin{remark}
It is expected that if we apply the above steps repeatedly (i.e. 
substitute $B^{(0)}$ in \eqref{equ:U1} by the solution of
\eqref{equ:B1-hat}, and solve \eqref{equ:B1-hat} where $U^{(1)}$
substituted by the solution of the new equation), the results of
better accuracy will be achieved. In nonlinear cases, however, it may
be of some difficulties to get the correction terms (which will be
introduced later) of more than second order accuracy, so we only
consider the model up to second order accuracy, i.e. the first order
model.
\end{remark}

%%%%%%%%%%%%%%%%%%%%%%%%%%%%%%%%%%%%%%%%%%%%%%%%%% 
%% revised first order model 
%%%%%%%%%%%%%%%%%%%%%%%%%%%%%%%%%%%%%%%%%%%%%%%%%%
The equation \eqref{equ:B1-hat} does not provide a
convenient way to compute $\hat{B}^{(1)}$ due to the appearance of
$\bU^{(1)}_\tau$. Let us keep in mind that our aim is to deduce an
equation which only has the riverbed as the variable (just like
\eqref{equ:linear-0}).  To this end, one needs to represent
$\bU^{(1)}_\tau$ by the riverbed up to the accuracy of order
$\cO(\varepsilon^2)$.

Suppose $\bU^{(0)}$ is the steady state when $B$ is fixed to
$B^{(0)}(x,0)$, namely  $\bU^{(0)}_\tau=0$. This implies that
$(\bU^{(1)}-\bU^{(0)})_\tau=\bU^{(1)}_\tau$, and thereafter we
consider $\bvarphi=\bU^{(1)}-\bU^{(0)}$ other than $\bU^{(1)}$. 
This trick is also used for the nonlinear systems: instead of solving
the coupled equations to get $\bU^{(1)}_\tau$, we use the steady
state $\bU^{(0)}$ and calculate the difference $\bvarphi$ to
obtain $\bU^{(1)}$. Now, we derive the closed form of $\bvarphi$.

First, we have the following equations:
\[
  \begin{aligned}
    \varepsilon \bU^{(0)}_{\tau} + \bA \bU^{(0)}_x &= -\bg
    B^{(0)}_x(x, 0), \\
    \varepsilon \bU^{(1)}_{\tau} + \bA \bU^{(1)}_x &= -
    \bg B^{(0)}_x(x-\lambda_B^{(0)}\tau, 0).
  \end{aligned}
\]
Taking the difference of the above dynamics and applying the
characteristic decomposition of $\bA$, we have 
\begin{equation}\label{equ:linear-varphi0}
\bX\bvarphi_\tau + \frac{\bLambda}{\varepsilon}\bX
\bvarphi_x = -\frac{\bX\bg}{\varepsilon}
\big(B^{(0)}_x(x-\lambda^{(0)}_B\tau,0)- B^{(0)}_x(x,0)\big).
\end{equation}
The initial condition of $\bvarphi$ can be derived from Lemma
\ref{lemma:2-conv-1} and steady state of $\bU^{(0)}$, i.e. 
\begin{equation}\label{equ:varphi-init}
\bX\bvarphi_x(x,0)=
\varepsilon\bLambda^{-2}\bX\bg\bc^T\bA^{-1}\bg
B_x^{(0)}(x,0)+\cO(\varepsilon^2).
\end{equation}
Then, \eqref{equ:linear-varphi0} can be solved analytically by the
method of characteristics as
\[
  \begin{aligned}
  & (\bX\bvarphi)_k(x, \tau) \\ 
  =&(\bX\bvarphi)_k(x-\frac{\lambda_k}{\varepsilon}\tau,0)\\
  &-\frac{(\bX\bg)_k}{\varepsilon} \int_0^{\tau}
  B^{(0)}_x\big( x-\frac{\lambda_k}{\varepsilon}\tau +
  (\frac{\lambda_k}{\varepsilon}-\lambda_B^{(0)})s, 0\big) -
  B^{(0)}_x\big( x-\frac{\lambda_k}{\varepsilon}\tau +
  \frac{\lambda_k}{\varepsilon}s, 0 \big) \rd s \\
  =& -\frac{(\bX\bg)_k}{(\lambda_k-\varepsilon\lambda_B^{(0)})} 
  B^{(0)}(x-\lambda_B^{(0)}\tau,0) +
  \frac{(\bX\bg)_k}{\lambda_k}B^{(0)}(x, 0) \\
 &+(\bX\bvarphi)_k(x-\frac{\lambda_k}{\varepsilon}\tau,0) 
 +\frac{\varepsilon(\bX\bg)_k\lambda_B^{(0)}}{\lambda_k(\lambda_k-\varepsilon \lambda_B^{(0)})}
  B^{(0)}(x-\frac{\lambda_k}{\varepsilon}\tau, 0).
  \end{aligned}
\]
That is, 
%\begin{equation} label{equ:linear-varphi}
  \begin{flalign}
    (\bX\bvarphi)_k(x,\tau) = &-
    \frac{\varepsilon(\bX\bg)_k\lambda_B^{(0)}}{\lambda_k(\lambda_k-\varepsilon\lambda_B^{(0)})}
    B^{(0)}(x, 0) & \cdots\cdots ~ 
    \cO(\varepsilon)~\text{term} \nonumber\\ 
    &-\frac{(\bX\bg)_k}{(\lambda_k-\varepsilon\lambda_B^{(0)})}
    \big(B^{(0)}(x-\lambda_B^{(0)}\tau, 0) - B^{(0)}(x, 0)\big)  &
    \cdots\cdots ~ \cO(\tau)~\text{term} \label{equ:linear-varphi}\\
    &+(\bX\bvarphi)_k(x-\frac{\lambda_k}{\varepsilon}\tau,0)
    &\cdots\cdots ~ \text{high
    order term}\nonumber\\
    &+\frac{\varepsilon(\bX\bg)_k\lambda_B^{(0)}}{\lambda_k(\lambda_k-\varepsilon \lambda_B^{(0)})}
    B^{(0)}(x-\frac{\lambda_k}{\varepsilon}\tau, 0). 
    & \cdots\cdots ~ \text{high order term} \nonumber
  \end{flalign}
%\end{equation}
We will show later that the last two terms provide
$\cO(\varepsilon^2)$ term in the error estimation. This is  
why they are called {\it high order terms}. In the nonlinear system,
the correction term will be shown to have a similar form.

Denote $\be_k (k=1,\cdots, l)$ the unit vectors of $\R^l$ and
$\diag(\bx)$ the diagonal matrix with vector $\bx$ in its diagonal
entries. By the decomposition \eqref{equ:linear-varphi} and initial
condition \eqref{equ:varphi-init}, we have  
\[
\begin{aligned}
  \bvarphi_\tau(x,\tau) = &~ \lambda_B^{(0)}\bX^{-1}(\bLambda-
    \varepsilon\lambda_B^{(0)}\bI)^{-1}\bX\bg
    B_x^{(0)}(x-\lambda_B^{(0)}\tau,0)\\
    &~-\frac{1}{\varepsilon}\bX^{-1}\bLambda
    \left(\sum_{k=1}^l(\bX\bvarphi_x)_k(x-\frac{\lambda_k}{\varepsilon}\tau,0)\be_k\right)\\
  &-\lambda_B^{(0)}\bX^{-1}(\bLambda-\varepsilon
    \lambda^{(0)}_B\bI)^{-1}\diag(\bX\bg)
    \left(\sum_{k=1}^l B^{(0)}_x(x-\frac{\lambda_k}{\varepsilon}\tau,0)
    \be_k \right) \\
  =&~ \lambda_B^{(0)}\bX^{-1}\bLambda^{-1}\bX\bg
     B_x^{(0)}(x-\lambda_B^{(0)}\tau,0)\\
     &~-\bX^{-1}\bLambda^{-1}\diag(\bX\bg\bc^T\bA^{-1}\bg)
     \left(\sum_{k=1}^l B^{(0)}_x(x-\frac{\lambda_k}{\varepsilon}\tau,0)\be_k\right)\\
     &~-\lambda_B^{(0)}\bX^{-1}\bLambda^{-1}
    \diag(\bX\bg)\left(\sum_{k=1}^l B^{(0)}_x(x-\frac{\lambda_k}{\varepsilon}\tau,0)
    \be_k \right) + \cO(\varepsilon).
\end{aligned}
\]
Define 
\[
\tilde{\bc}^T = \bc^T \bA^{-1}\bX^{-1}\bLambda^{-1}
 \diag\left(\lambda_B^{(0)}\bX\bg +\bX\bg\bc^T\bA^{-1}\bg
 \right) \qquad \tilde{c}_k = \bc^T \be_k.
\]
Note that $\bU^{(1)}_\tau=\bvarphi_\tau$ , the model
\eqref{equ:B1-hat} can be rewritten as
\begin{equation}\label{equ:order1-approx1}
\begin{aligned}
\hat{B}^{(1)}_\tau - \bc^T\bA^{-1}\bg \hat{B}^{(1)}_x = &~
\varepsilon\lambda_B^{(0)}\bc^T\bA^{-2}\bg
B_x^{(0)}(x-\lambda_B^{(0)}\tau,0)\\
 &-\varepsilon \sum_{k=1}^l\tilde{c}_k
 B^{(0)}_x(x-\frac{\lambda_k}{\varepsilon}\tau,0)  +
 \cO(\varepsilon^2).
\end{aligned}
\end{equation}

By discarding the last two terms in \eqref{equ:order1-approx1} and
denote its solution by $\tilde{B}^{(1)}$, we have 
\begin{equation}\label{equ:order1-approx2}
    \tilde{B}^{(1)}_\tau-\bc^T\bA^{-1}\bg
    \tilde{B}^{(1)}_x=\varepsilon\lambda_B^{(0)}\bc^T\bA^{-2}\bg
     B_x^{(0)}(x-\lambda_B^{(0)}\tau,0).
\end{equation}
%%% error for tilde{B} %%%
Now, we show that $\lVert \hat{B}^{(1)}-\tilde{B}^{(1)}\rVert_\infty=
\cO(\varepsilon^2)$ for $\tau\sim\cO(1)$. Taking the difference of
\eqref{equ:order1-approx2} and \eqref{equ:B1-hat}. Let $\hat{E} =
\hat{B}^{(1)}-\tilde{B}^{(1)}$, since \eqref{equ:B1-hat} can be
rewritten as \eqref{equ:order1-approx1}, we have
\[
\hat{E}_\tau+\lambda_B^{(0)}\hat{E}_x
= \varepsilon \sum_{k=1}^l \tilde{c}_k
B_x^{(0)}(x-\frac{\lambda_k}{\varepsilon}\tau,0) +
\cO(\varepsilon^2),
\]
with initial condition $\hat{E}(x,0)=0$.  By the method of
characteristics again, 
\[
\begin{aligned}
\hat{E}(x, \tau) &= \varepsilon \sum_{k=1}^l \tilde{c}_k 
\int_0^\tau
B_x^{(0)}(x-\lambda_B^{(0)}\tau+(\lambda_B^{(0)} -
  \frac{\lambda_k}{\varepsilon})s,0) \rd s + \cO(\varepsilon^2)\\
& = \sum_{k=1}^l
\frac{\varepsilon^2}{\varepsilon\lambda_B^{(0)}-\lambda_k} \left(
    B^{(0)}(x-\frac{\lambda_k}{\varepsilon}\tau,0)-B^{(0)}(x-\lambda_B^{(0)}\tau,0)
    \right) + \cO(\varepsilon^2) = \cO(\varepsilon^2).
\end{aligned}
\]
Hence, $\tilde{B}^{(1)}$ is proven to be a $\cO(\varepsilon^2)$ order
approximation to $\hat{B}^{(1)}$, and thus the $\cO(\varepsilon^2)$
order approximation to original solution $B$ by Lemma
\ref{lemma:2-conv-1}.  The final formula of {\it first order model}
can be acquired by modifying \eqref{equ:order1-approx2} as
\[
B^{(1)}_\tau-\bc^T\bA^{-1}\bg
B^{(1)}_x=\varepsilon\lambda^{(0)}_B\bc^T\bA^{-2} \bg
B^{(1)}_x(x,\tau).
\]
Or, 
\begin{equation} \label{equ:linear-1}
  B^{(1)}_\tau + \lambda_B^{(1)} B^{(1)}_x = 0,
\end{equation}
where $\lambda_B^{(1)} = \lambda_B^{(0)} - \varepsilon \lambda_B^{(0)}
\bc^T\bA^{-2}\bg$.  

It is straightforward to prove that $B^{(1)}$ is a
$\cO(\varepsilon^2)$ order approximation to $\tilde{B}^{(1)}$ and thus
a $\cO(\varepsilon^2)$ order approximation to original solution $B$,
which is precisely the following theorem:

\begin{theorem}\label{prop:2-conv}
Assume the initial dynamics satisfies 
\begin{enumerate}
\item $(\bA+\varepsilon \bX^{-1}\bLambda\hat{\bX})\bU_x(x,0) +
(\bg+\varepsilon\bX^{-1}\bLambda\hat{\alpha})B_x(x,0) =\cO(\varepsilon^2)$.
\item $\bA\bU_{xx}(x,0) + \bg B_{xx}(x,0) = \cO(\varepsilon)$.
\item $B^{(1)}(x,0) = B(x,0) \in W^{2,\infty}(\R)$.
\end{enumerate}
Then 
\[
\|B(x,t) - B^{(1)}(x,t)\|_{\infty} = \cO(\varepsilon^2) \qquad
\text{for}~ t\sim \cO(\varepsilon^{-1}).
\]
\end{theorem}

In nonlinear cases, $\lambda_B^{(0)}$ and $\lambda_B^{(1)}$ are
functions of the steady state other than constants. Therefore, the
steady state as well as the correction term $\bvarphi$ should be
computed at every step when solving zeroth order model or the first
order model. 
%%Hence, we need to set some time $\tau_0$
%%other than $0$ as the initial time and computing the correction term
%%during $[\tau_0,\tau_0+\tilde{\tau}]$.
Although we actually do not have to use $\bvarphi$ in linear cases (we
only need to use $\bU^{(1)}_\tau$), the correction term $\bvarphi$
is essential in nonlinear cases due to two reasons: (i) it is needed
to give the high order algorithm for solving first order model; (ii)
it will be used to give the prediction of fast variables at next step,
which could improve the efficiency of computing the steady state.

%%%%%%%%%%%%%%%%%%%%%%%%%%%%%%%%%%%%%%%%%%%%%%%%%%
%% section 3 %%
%%%%%%%%%%%%%%%%%%%%%%%%%%%%%%%%%%%%%%%%%%%%%%%%%%
\section{One Dimensional Sediment Transport Model}\label{sec:1d}

Now we consider the one dimensional sediment transport model
\eqref{equ:1d-sediment}. The process here is quite similar with the
linear case. First, we derive the formulation of riverbed equation
through the original coupled system. Then, the zeroth order model is
obtained by taking $\varepsilon\rightarrow0$. After that, the
correction term is considered and first order model will then be
derived.  

%%%%%%%%%%%%%%%%%%%%%%%%%%%%%%%%%%%%%%%%%%%%%%%%%%
%% zeroth order model
%%%%%%%%%%%%%%%%%%%%%%%%%%%%%%%%%%%%%%%%%%%%%%%%%%
\subsection{Zeroth order model} \label{subsec:1d-0} First,
we reformulate the sediment transport systems
\eqref{equ:1d-sediment} with primitive variables as
\begin{equation} \label{equ:1d-equation} 
  \left\{
    \begin{aligned}
      &\begin{pmatrix}
        h \\ u 
      \end{pmatrix}_t + 
      \begin{pmatrix}
        u & h \\
        g & u  \\
      \end{pmatrix}
      \begin{pmatrix}
        h \\ u 
      \end{pmatrix}_x = \begin{pmatrix} 
        0 \\ -gB_x 
      \end{pmatrix}, \\
      &B_t + \varepsilon \big(\tilde{q}_b(|u|)+|u|\tilde{q}_b'(|u|)\big) 
      u_x = 0. 
    \end{aligned}
  \right.
\end{equation}
When the flow is subcritical, namely $|u|<\sqrt{gh}$ everywhere, the
fast dynamics in \eqref{equ:1d-equation} deduces that 
\[
  \left\{
  \begin{aligned}
    h_x &= \frac{1}{u^2-gh} (-uh_t + hu_t + ghB_x), \\
    u_x &= \frac{1}{u^2-gh} (gh_t - uu_t - guB_x).
  \end{aligned}
  \right.
\]
Eliminating the spatial derivative terms of fast variables in the
sediment transport equation, we have
\begin{equation} \label{equ:1d-0tau}
  B_\tau - \frac{gu\tilde{\lambda}_b(|u|)}{u^2-gh}B_x + \varepsilon
  \frac{\tilde{\lambda}_b(|u|)}{u^2-gh}(gh_\tau - uu_\tau) = 0,
\end{equation}
where $\tilde{\lambda}_b(|u|) = \tilde{q}_b(|u|) +
|u|\tilde{q}_b'(|u|)$. The {\it zeroth order model} (or the
{\it limiting equation}) for the 1D sediment transport systems
\eqref{equ:1d-sediment} is derived by taking $\varepsilon \to 0$ in
\eqref{equ:1d-0tau} as
\begin{equation} \label{equ:1d-0}
  B^{(0)}_\tau + \lambda_B^{(0)}(h^{(0)},u^{(0)}) B^{(0)}_x = 0,
\end{equation}
where $u^{(0)},h^{(0)}$ are the steady states with fixed riverbed, and
$\lambda_B^{(0)}(h,u) = -\dfrac{gu\tilde{\lambda}_b(|u|)}{u^2-gh}$. It
should be noted that $h^{(0)}$ and $u^{(0)}$ are the functions of
riverbed $B^{(0)}$.  Unlike the linear hyperbolic system, the
characteristic speed of \eqref{equ:1d-0} depends on the fast
variables. Similar to the discussion in section \ref{sec:1d-linear},
$h_\tau$ and $u_\tau$ are assumed to be bounded so that the last term
in \eqref{equ:1d-0tau} tends to zero when $\varepsilon\to0$. Usually,
the steady state of flow exists with fixed riverbed when appropriately
applying the boundary condition, which implies that $h_\tau$ and
$u_\tau$ are bounded for the sediment transport.

\begin{remark}
From the characteristic speed of riverbed in \eqref{equ:1d-0tau}, 
we know that $|\lambda_B^{(0)}|\rightarrow\infty$ if $|u|\rightarrow
\sqrt{gh}$, which means that our model can only handle the case
in which $|u|$ stays away from $\sqrt{gh}$ everywhere. This can be
guaranteed for the subcritical case $|u| < \sqrt{gh}$. 
\end{remark}

%%%%%%%%%%%%%%%%%%%%%%%%%%%%%%%%%%%%%%%%%%%%%%%%%% 
%% first order model 
%%%%%%%%%%%%%%%%%%%%%%%%%%%%%%%%%%%%%%%%%%%%%%%%%%
\subsection{First order model} \label{subsec:1d-1}
As with linear case, the most essential step in deriving the
correction model is to compare the fast dynamics with fixed riverbed
to the modified one with riverbed moving according to zeroth order
model. Let $\tau_0$ be the base time with the riverbed $B^{(0)}(x,
\tau_0)$. At time $\tau = \tau_0 + \tilde{\tau}$, 
\begin{equation} \label{equ:1d-diff}
  \begin{aligned}
  \begin{pmatrix}
    h^{(0)} \\ u^{(0)}
  \end{pmatrix}_\tau + \frac{1}{\varepsilon}
  \begin{pmatrix}
    u^{(0)} & h^{(0)} \\ g & u^{(0)}
  \end{pmatrix}
  \begin{pmatrix}
    h^{(0)} \\ u^{(0)}
  \end{pmatrix}_x &= 
  \begin{pmatrix}
    0 \\ -\dfrac{g}{\varepsilon}B^{(0)}_x(x,\tau_0)
  \end{pmatrix}, \\
  \begin{pmatrix}
    h^{(1)} \\ u^{(1)}
  \end{pmatrix}_\tau + \frac{1}{\varepsilon}
  \begin{pmatrix}
    u^{(1)} & h^{(1)} \\ g & u^{(1)}
  \end{pmatrix}
  \begin{pmatrix}
    h^{(1)} \\ u^{(1)}
  \end{pmatrix}_x &= 
  \begin{pmatrix}
    0 \\ -\dfrac{g}{\varepsilon}B^{(0)}_x(x,\tau_0 + \tilde{\tau})
  \end{pmatrix}. 
\end{aligned}
\end{equation}
Thereafter, we often consider the case in which $\tau_0 = 0$ to make
it more concise. Let
$$
\varphi_h(x,\tilde{\tau}) = h^{(1)}(x,\tilde{\tau})-h^{(0)}(x),
\qquad \varphi_u(x,\tilde{\tau}) = u^{(1)}(x,\tilde{\tau})-u^{(0)}(x)
$$ be
the correction of fast variables. Similar to
\eqref{equ:linear-varphi}, we intend to decompose $\varphi_h$ and
$\varphi_u$ into the sum of $\cO(\varepsilon)$ term,
$\cO(\tilde{\tau})$ term and high order term, namely 
\begin{equation} \label{equ:1d-correction}
  \begin{aligned}
    \varphi_h(x, \tilde{\tau}) &= \varepsilon \varphi_h^{(0)}(x) +
    \varphi_h^{(1)}(x,\tilde{\tau}) + \text{high order term}, \\
    \varphi_u(x, \tilde{\tau}) &= \varepsilon \varphi_u^{(0)}(x) +
    \varphi_u^{(1)}(x,\tilde{\tau}) + \text{high order term}.
  \end{aligned}
\end{equation}

%% O(tau) %%
\subsubsection{$\cO(\tilde{\tau})$ Term}
Let $\tilde{B}^{(0)} = B^{(0)}(x, \tilde{\tau})$. Taking the
difference of the two dynamics in \eqref{equ:1d-diff} to obtain
\[
  \left\{
    \begin{aligned}
    &(\varphi_h)_\tau + \frac{1}{\varepsilon}
    (h^{(1)}u^{(1)} - h^{(0)}u^{(0)})_x = 0, \\
    &(\varphi_u)_\tau + \frac{1}{\varepsilon} 
    \left[ gh^{(1)} + \frac{1}{2}(u^{(1)})^2 + g\tilde{B}^{(0)}
    - gh^{(0)}-\frac{1}{2}(u^{(0)})^2-gB^{(0)} \right]_x = 0.
    \end{aligned}
  \right.
\]
Then, eliminating $h^{(1)}, u^{(1)}$ and neglecting the high order term
to obtain the linearized equation as
\begin{equation} \label{equ:1d-linearization}
  \left\{
    \begin{aligned}
    &(\varphi_h)_\tau + \frac{1}{\varepsilon}
    (h^{(0)}\varphi_u + u^{(0)}\varphi_h)_x \approx 0, \\
    &(\varphi_u)_\tau + \frac{1}{\varepsilon} 
    (g\varphi_h + u^{(0)}\varphi_u +g\tilde{B}^{(0)}-gB^{(0)})_x
    \approx 0. 
    \end{aligned}
  \right.
\end{equation}
Collecting the $\cO(1/\varepsilon)$ terms and assuming that  
$\varphi_h^{(1)},\varphi_u^{(1)}$ are zero when $|x|\to \infty$, we
have 
\begin{equation}\label{equ:1d-varphi1-equ}
    \left\{
    \begin{aligned}
        &h^{(0)}\varphi_u^{(1)} + u^{(0)}\varphi_h^{(1)}=0,\\
        &g\varphi_h^{(1)}+u^{(0)}\varphi_u^{(1)}+g\tilde{B}^{(0)}-gB^{(0)}=0.
    \end{aligned}
    \right.
\end{equation}
Namely, 
\begin{equation} \label{equ:1d-varphi1}
  \varphi_h^{(1)} = \frac{gh^{(0)}}{(u^{(0)})^2-gh^{(0)}}(\tilde{B}^{(0)} -
  B^{(0)}), \qquad 
  \varphi_u^{(1)} = -\frac{gu^{(0)}}{(u^{(0)})^2-gh^{(0)}}(\tilde{B}^{(0)} -
  B^{0}).
\end{equation}
\subsubsection{$\cO(\varepsilon)$ Term}
%Let us take the difference of the two dynamics in \eqref{equ:1d-diff}
%to obtain
%\[
%  \left\{
%    \begin{aligned}
%    &(\varphi_h)_\tau + \frac{1}{\varepsilon}
%    (h^{(1)}u^{(1)} - h^{(0)}u^{(0)})_x = 0, \\
%    &(\varphi_u)_\tau + \frac{1}{\varepsilon} 
%    \left[ gh^{(1)} + \frac{1}{2}(h^{(1)})^2+g\tilde{B}^{(0)}
%    - gh^{(0)}-\frac{1}{2}(u^{(0)})^2-gB^{(0)} \right]_x = 0.
%    \end{aligned}
%  \right.
%\]
%Then we eliminate $h^{(1)}$ and $u^{(1)}$ by taking the formations
%of $\varphi_h$ and $\varphi_u$, and neglect the high order term to
%obtain the linearized equation as
%\begin{equation} \label{equ:1d-linearization}
%  \left\{
%    \begin{aligned}
%    &(\varphi_h)_\tau + \frac{1}{\varepsilon}
%    (h^{(0)}\varphi_u + u^{(0)}\varphi_h)_x = 0, \\
%    &(\varphi_u)_\tau + \frac{1}{\varepsilon} 
%    (g\varphi_h + u^{(0)}\varphi_u +g\tilde{B}^{(0)}-gB^{(0)})_x = 0. 
%    \end{aligned}
%  \right.
%\end{equation}
Consider the remaining terms in $\varphi_h,\varphi_u$ besides the
$\cO(\tilde{\tau})$ terms.
%%Based on
%%\eqref{equ:1d-correction}, the linearized equation
%%\eqref{equ:1d-linearization} leads to the error of order
%%$\cO(\varepsilon^2+\varepsilon\tilde{\tau}+\tilde{\tau}^2)$. 
Taking $\varphi_h = \varphi_h^{(1)} + \varepsilon \varphi_h^{(0)}$ and
$\varphi_u = \varphi_u^{(1)} + \varepsilon \varphi_u^{(0)}$ into
\eqref{equ:1d-linearization}, we have
\[
  \left\{
  \begin{aligned}
    (u^{(0)}\varphi_h^{(0)} + h^{(0)}\varphi_u^{(0)})_x &= 
    \frac{gh^{(0)}\lambda_B^{(0)}(h^{(0)},u^{(0)})}{(u^{(0)})^2-gh^{(0)}}B^{(0)}_x -
    \varepsilon (\varphi_h^{(0)})_\tau, \\
    (g\varphi_h^{(0)} + u^{(0)}\varphi_u^{(0)})_x &= - 
    \frac{gu^{(0)}\lambda_B^{(0)}(h^{(0)},u^{(0)})}{(u^{(0)})^2-gh^{(0)}}B^{(0)}_x
    -\varepsilon (\varphi_u^{(0)})_\tau.
  \end{aligned}
  \right.
\]
Collecting the $\cO(1)$ term, we obtain 
\begin{equation} \label{equ:1d-varphi0}
  \left\{
  \begin{aligned}
    (u^{(0)}\varphi_h^{(0)} + h^{(0)}\varphi_u^{(0)})_x &= 
    \frac{gh^{(0)}\lambda_B^{(0)}(h^{(0)},u^{(0)})}{(u^{(0)})^2-gh^{(0)}}B^{(0)}_x, \\ 
    (g\varphi_h^{(0)} + u^{(0)}\varphi_u^{(0)})_x &= - 
    \frac{gu^{(0)}\lambda_B^{(0)}(h^{(0)},u^{(0)})}{(u^{(0)})^2-gh^{(0)}}B^{(0)}_x. 
  \end{aligned}
  \right.
\end{equation}
Notice that \eqref{equ:1d-varphi0} is a linear system for
$\varphi_h^{(0)}$ and $\varphi_u^{(0)}$, which is convenient to be
numerically solved, see Subsection \ref{sec:scheme-correction}.

%% slow variable correction
\subsubsection{Slow variable correction}
Having the $\cO(\tilde{\tau})$ and $\cO(\varepsilon)$ terms,
\eqref{equ:1d-0tau} can be reformulated as
\[
  \hat{B}^{(1)}_\tau -
  \frac{gu^{(1)}\tilde{\lambda}_b(|u^{(1)}|)}{(u^{(1)})^2-gu^{(1)}}
  \hat{B}^{(1)}_x + \varepsilon
  \frac{g((u^{(1)})^2+gh^{(1)})\tilde{\lambda}_b(|u^{(1)}|)} 
  {((u^{(1)})^2-gh^{(1)})^2} \tilde{B}^{(0)}_\tau = 0,
\]
where $h^{(1)} = h^{(0)} + \varepsilon \varphi_h^{(0)} +
\varphi_h^{(1)}, u^{(1)} = u^{(0)} + \varepsilon \varphi_u^{(0)} +
\varphi_u^{(1)}$.
Similarly, the last term is regarded as a correction term on the
characteristic speed of slow variable. Let $\tilde{\tau} \rightarrow
0$, we derive the model with {\it first order correction} for the 1D
sediment transport systems \eqref{equ:1d-equation} as
\begin{equation} \label{equ:1d-1}
  B^{(1)}_\tau + \lambda_B^{(1)}(h^{(0)}+\varepsilon \varphi_h^{(0)},
      u^{(0)}+\varepsilon \varphi_u^{(0)})B^{(1)}_x = 0,
\end{equation}
where
\[ 
  \lambda_B^{(1)}(h,u) = \lambda_B^{(0)}(h,u)-\varepsilon\lambda_B^{(0)}(h,u) 
  \frac{g(u^2+gh)\tilde{\lambda}_b(|u|)}{(u^2-gh)^2}.
\]

\section{Two Dimensional Sediment Transport Model} \label{sec:2d} 

We move to the modelling of the two dimensional hyperbolic system
governing sediment transport. We will see later that there is an
essential difference between one dimensional model and two dimensional
model. In two dimensional case, the model cannot be obtained by
eliminating the spatial derivative terms of fast variables. Instead,
the two dimensional model is a convection equation with source term,
while is consistent with the one dimensional model.

\subsection{Zeroth order model}
%The 2D hyperbolic system governing sediment transport is formulated as
%\begin{equation} \label{equ:2d-sediment}
%  \dfrac{\partial}{\partial t}\begin{pmatrix}
%    h \\ hu \\ hv \\ B
%  \end{pmatrix} +
%  \dfrac{\partial}{\partial x}\begin{pmatrix}
%    hu \\ hu^2+\frac{1}{2}gh^2 \\ huv \\ \varepsilon u\tilde{q}_b(|\bu|)
%  \end{pmatrix} + 
%  \dfrac{\partial}{\partial y}\begin{pmatrix}
%    hv \\ huv \\ hv^2+\frac{1}{2}gh^2 \\ \varepsilon v\tilde{q}_b(|\bu|)
%  \end{pmatrix} = 
%  \begin{pmatrix}
%    0 \\ -ghB_x \\ -ghB_y \\ 0
%  \end{pmatrix},
%\end{equation}
%where $v$ is the vertically averaged flow velocity along the $y$
%direction, $\bu=(u,v)^T$, and the sediment transport flux is
%introduced in section \ref{sec:1d}. 
Let
\[
\nabla\cdot\bu = u_x+v_y, \quad  
\bu\cdot \nabla\bu = (uu_x+vu_y, uv_x+vv_y)^T, 
\]
then  
\[
\begin{aligned}
  (u\tilde{q}_b(|\bu|))_x + (v\tilde{q}_b(|\bu|))_y &= \tilde{q}_b u_x
  + \frac{\tilde{q}_b'}{|\bu|}(u^2u_x+uvv_x)\\
  &\quad + \tilde{q}_b v_y +\frac{\tilde{q}_b'}{|\bu|}(vuu_y+v^2v_y) \\
  &= \tilde{q}_b(\nabla\cdot\bu) + \frac{\tilde{q}_b'}{|\bu|}\bu^T
  (\bu\cdot\nabla\bu).
\end{aligned}
\]
We reformulate the 2D sediment transport system
\eqref{equ:2d-sediment} with primitive variables $\bW=(h,\bu,B)^T$ as%

\begin{equation} \label{equ:2d-equation}
  \left\{
    \begin{aligned}
      &h_t + h\nabla\cdot\bu + \bu^T\nabla h = 0, \\
      &\bu_t + \bu\cdot\nabla\bu + g\nabla h = -g\nabla B, \\ 
      &B_t + \varepsilon \big[ \tilde{q}_b(|\bu|)(\nabla\cdot\bu)
      + \frac{\tilde{q}_b'(|\bu|)}{|\bu|}\bu^T (\bu\cdot\nabla\bu) 
    \big] = 0. 
    \end{aligned}
  \right.
\end{equation}
By the equations of mass and momentum conservation in
\eqref{equ:2d-equation}, we have
\[
\begin{aligned}
  gh_t + gh\nabla\cdot\bu + g\bu^T\nabla h &=0, \\
  \frac{1}{2}(|\bu|^2)_t + \bu^T(\bu\cdot\nabla\bu) + g\bu^T\nabla h
  &= -g\bu^T\nabla B.
\end{aligned}
\]
Eliminating the spatial derivative of $h$ to obtain
\begin{equation} \label{equ:2d-du}
  \bu^T(\bu\cdot\nabla\bu) - gh\nabla\cdot\bu = -\bu^T\bu_t + gh_t -
  g\bu^T\nabla B.
\end{equation}
Notice that the spatial derivative of $\bu$ can not be solved from
\eqref{equ:2d-du}. Therefore, we introduce a rotational invariant
operator $\cL^{\bS}$ as
\begin{equation} \label{equ:2d-L}
  \cL^{\bS} \bu \triangleq \bu^T(\bu\cdot\nabla\bu) - |\bu|^2
  (\nabla\cdot\bu),
\end{equation}
which degenerates to null operator for 1D case $(v=0)$.
%Actually, the
%fact that the operator $\cL^\bS$ vanishes in 1D is the essential
%difference between 1D and 2D case. 
Then, the spatial derivative of $\bu$ from \eqref{equ:2d-du} and
\eqref{equ:2d-L} are represented as
\begin{equation} \label{equ:2d-duL}
  \left\{
    \begin{aligned}
      &\nabla\cdot\bu = \frac{-\bu^T\bu_t + gh_t}{|\bu|^2-gh} 
      - \frac{g\bu^T\nabla B}{|\bu|^2-gh} 
      - \frac{\cL^\bS \bu}{|\bu|^2-gh}, \\
      &\bu^T(\bu\cdot\nabla\bu) = 
      \frac{-\bu^T\bu_t + gh_t}{|\bu|^2-gh}|\bu|^2 
      - \frac{g\bu^T\nabla B}{|\bu|^2-gh}|\bu|^2 
      - \frac{gh\cL^\bS \bu}{|\bu|^2-gh}.
    \end{aligned}
  \right.
\end{equation}

Substituting \eqref{equ:2d-duL} into \eqref{equ:2d-equation}, we have
\begin{equation} \label{equ:2d-0tau}
  B_\tau - \frac{g\tilde{\lambda}_b(|\bu|)\bu^T}{|\bu|^2-gh} \nabla B
   + \varepsilon
  \frac{\tilde{\lambda}_b(|\bu|)}{|\bu|^2-gh}  (gh_\tau
  -\bu^T\bu_\tau) =
  \frac{\tilde{q}_b+gh\frac{\tilde{q}_b'}{|\bu|}}{|\bu|^2-gh}\cL^\bS \bu,
\end{equation}
where $\tilde{\lambda}_b(|\bu|) = \tilde{q}_b(|\bu|) +
|\bu|\tilde{q}_b'(|\bu|)$. Similar to 1D case, the {\it zeroth order
model} for the 2D sediment transport system \eqref{equ:2d-equation}
can be derived by taking $\varepsilon \to 0$, 
\begin{equation} \label{equ:2d-0}
  B^{(0)}_\tau + \blambda_B^{(0)}(h^{(0)},\bu^{(0)}) \nabla B^{(0)} =
  S_B^{(0)}(h^{(0)},\bu^{(0)}),
\end{equation}
where 
\begin{equation}\label{equ:2d-lambda-S}
  \blambda_B^{(0)}(h^{(0)},\bu^{(0)}) =
  -\frac{g\tilde{\lambda}_b(|\bu^{(0)}|)\bu^{(0)T}}{|\bu^{(0)}|^2-gh^{(0)}},
  \qquad S_B^{(0)}(h^{(0)},\bu^{(0)}) =
  \frac{\tilde{q}_b+gh^{(0)}\frac{\tilde{q}_b'}{|\bu^{(0)}|}}{|\bu^{(0)}|^2-gh^{(0)}}\cL^\bS
  \bu^{(0)},
\end{equation}
and $h^{(0)},\bu^{(0)}$ are steady states with respect to $B^{(0)}$.
Roughly speaking, \eqref{equ:2d-0} is not a convective equation with
source term, since $S_B^{(0)}$ in \eqref{equ:2d-0} involves the spatial
derivative of fast variables. From the comparison of \eqref{equ:1d-0}
and \eqref{equ:2d-0}, we find that the zeroth order model of 2D case
is consistent with the model of 1D case, namely, \eqref{equ:2d-0}
degenerates to \eqref{equ:1d-0} if $v=0$ and $B_y=0$. Therefore, one
may expect that \eqref{equ:2d-0} has captured the leading order part
of the characteristic speed for sediment transport.

%% first order correction
\subsection{First order correction}

Similar to the 1D case, we first calculate the difference between the
fast dynamics with fixed slow variable and the modified one with
predicted slow variable. Using the similar notation with 1D case, we
have

\begin{equation} \label{equ:chap7-2d-diff}
  \begin{aligned}
    & \left\{ 
    \begin{aligned}
      h^{(0)}_\tau + \frac{1}{\varepsilon} [h^{(0)}\nabla\cdot\bu^{(0)} +
      \bu^{(0),T}\nabla h^{(0)}]
      &= 0, \\
      \bu^{(0)}_\tau + \frac{1}{\varepsilon}
      [\bu^{(0)}\cdot\nabla\bu^{(0)} + g\nabla h^{(0)}] 
      &= -\frac{1}{\varepsilon}g\nabla B^{(0)},
    \end{aligned}
    \right.\\
    & \left\{ 
    \begin{aligned}
      h^{(1)}_\tau + \frac{1}{\varepsilon} [h^{(1)}\nabla 
        \cdot\bu^{(1)} + \bu^{(1),T} \nabla h^{(1)}] &= 0, \\
        \bu^{(1)}_\tau + \frac{1}{\varepsilon} [\bu^{(1)}\cdot 
        \nabla\bu^{(1)} + g\nabla h^{(1)}] 
        &= -\frac{1}{\varepsilon}g\nabla \tilde{B}^{(0)}.
    \end{aligned}
    \right.
  \end{aligned}
\end{equation}
Let  
$$
\varphi_h = h^{(1)}(\bx,\tilde{\tau}) - h^{(0)}(\bx) \qquad
\bvarphi_\bu = \bu^{(1)}(\bx,\tilde{\tau}) - \bu^{(0)}(\bx). 
$$ 
The following linearized equation for $\varphi_h$ and $\varphi_\bu$
can be derived by dropping off the high order term
\begin{equation} \label{equ:2d-linearization}
  \left\{
    \begin{aligned}
      &(\varphi_h)_\tau + \frac{1}{\varepsilon} \nabla\cdot
      (h^{(0)}\bvarphi_\bu + \bu^{(0)}\varphi_h) \approx 0, \\
      &(\bvarphi_\bu)_\tau + \frac{1}{\varepsilon} [ 
      \bvarphi_\bu\cdot\nabla\bu^{(0)} + \bu^{(0)}\cdot\nabla\bvarphi_\bu 
    + \nabla(g\varphi_h + g(\tilde{B}^{(0)}-B^{(0)}))] \approx \bzero.
    \end{aligned}
  \right.
\end{equation}
Similar to 1D case, we intend to decompose $\varphi_h$ and
$\bvarphi_\bu$ into the sum of $\cO(\varepsilon)$ term,
$\cO(\tilde{\tau})$ term and high order term, i.e.
\begin{equation} \label{equ:2d-fast}
  \begin{aligned}
    \varphi_h &= \varepsilon \varphi_h^{(0)}(\bx) + 
    \varphi_h^{(1)}(\bx, \tilde{\tau}) + \text{high order term}, \\
    \bvarphi_\bu &= \varepsilon \bvarphi_\bu^{(0)}(\bx) + 
    \bvarphi_\bu^{(1)}(\bx, \tilde{\tau}) + \text{high order term}.
  \end{aligned}
\end{equation}

%% 2D correction: O(\tau) %%
\subsubsection{$\cO(\tilde{\tau})$ Term}
We first try to find the $\cO(\tilde{\tau})$ term analytically, which
is used to depict the change of the fast variables with the evolving
of slow variable. In comparison to the 1D case, the mass conservation
in \eqref{equ:2d-linearization} simply tells us that
$h^{(0)}\bvarphi_\bu+\bu^{(0)}\varphi_h$ is divergence free when
neglecting the time derivative. However, the term
$\bvarphi_\bu\cdot\nabla\bu^{(0)} + \bu^{(0)}\cdot\nabla\bvarphi_\bu$
in the momentum conservative equation can not be formulated to a total
derivative . Therefore, we first define
\begin{equation} \label{equ:2d-Lu}
  \cL^\bu(\bvarphi_\bu) = \bvarphi_\bu\cdot\nabla\bu + 
  \bu\cdot\nabla\bvarphi_\bu - \nabla(\bu^T\bvarphi_\bu).
\end{equation}
It is obvious that $\cL^\bu$ is a linear operator and degenerates to
null operator for 1D case ($\varphi_v=0$ or $v=0$). Hereafter, we
devote to analytically matching the flux and source term in
\eqref{equ:2d-linearization} with $\bar{\varphi}_h^{(1)}$ and
$\bar{\bvarphi}_\bu^{(1)}$, namely
\[
  h^{(0)}\bar{\bvarphi}_\bu^{(1)} + \bu^{(0)} \bar{\varphi}_h^{(1)} = \bzero,
  \qquad \bu^{(0),T}\bar{\bvarphi}_\bu^{(1)} + g\bar{\varphi}_h^{(1)} +
  g(\tilde{B}^{(0)} - B^{(0)}) = 0,
\]
which yields
\begin{equation} \label{equ:2d-varphi1-1}
  \bar{\varphi}_h^{(1)} = \frac{gh^{(0)}}{|\bu^{(0)}|^2-gh^{(0)}}(\tilde{B}^{(0)}-B^{(0)}), \quad 
  \bar{\bvarphi}_\bu^{(1)} =
  -\frac{g\bu^{(0)}}{|\bu^{(0)}|^2-gh^{(0)}}(\tilde{B}^{(0)}-B^{(0)}).
\end{equation}
Let $\hat{\varphi}_h^{(1)} = \varphi_h^{(1)} - \bar{\varphi}_h^{(1)}$
and $\hat{\bvarphi}_\bu^{(1)} = \bvarphi_\bu^{(1)} -
\bar{\bvarphi}_\bu^{(1)}$ be the rest parts of $\cO(\tilde{\tau})$
term. Notice that the $\cO(\tilde{\tau})$ term is used to balance
\eqref{equ:2d-linearization} without time evolving terms, thus the
equation of $\hat{\varphi}_h^{(1)}$ and $\hat{\bvarphi}_\bu^{(1)}$
can be proposed as
\begin{equation} \label{equ:2d-varphi1-2}
  \left\{
    \begin{aligned}
      &\nabla \cdot (h^{(0)}\hat{\bvarphi}_\bu^{(1)} +
      \bu^{(0)}\hat{\varphi}_h^{(1)}) = 0, \\
      &\cL^{\bu^{(0)}}(\hat{\bvarphi}_\bu^{(1)}) + \nabla \cdot
      (g\hat{\varphi}_h^{(1)} + \bu^{(0),T}\hat{\bvarphi}_\bu^{(1)}) =
      -\cL^{\bu^{(0)}}(\bar{\bvarphi}_\bu^{(1)}).
    \end{aligned}
  \right.
\end{equation}

Therefore, the $\cO(\tilde{\tau})$ term is composed of the
analytical terms from \eqref{equ:2d-varphi1-1} and the other terms
from \eqref{equ:2d-varphi1-2}. We also note that the
$\cO(\tilde{\tau})$ term of 2D case is consistent with that of 1D case
by the degeneration of $\cL^\bu$ in 1D case. 

%% First order: O(eps) %%
\subsubsection{$\cO(\varepsilon)$ Term}

The equations of $\varphi_h^{(0)}$ and $\bvarphi_\bu^{(0)}$ are
deduced by taking the $\cO(\tilde{\tau})$ term into linearized
equation \eqref{equ:2d-linearization} as well as neglecting the high
order term, namely 
\begin{equation} \label{equ:2d-varphi0-equ}
  \left\{
  \begin{aligned}
    &\nabla\cdot(h^{(0)}\bvarphi_\bu^{(0)} + \bu^{(0)}\varphi_h^{(0)}) = 
     - (\bar{\varphi}_h^{(1)})_\tau - (\hat{\varphi}_h^{(1)})_\tau -
     \varepsilon(\varphi_h^{(0)})_\tau, \\
    &\cL^{\bu^{(0)}}(\bvarphi_\bu^{(0)}) + \nabla
    (\bu^{(0),T}\bvarphi_\bu^{(0)} + g\varphi_h^{(0)}) =
    -(\bar{\bvarphi}_\bu^{(1)})_\tau - (\hat{\bvarphi}_\bu^{(1)})_\tau
    -\varepsilon(\bvarphi_\bu^{(0)})_\tau.
  \end{aligned}
  \right.
\end{equation}

The time derivatives of $\bar{\varphi}_h^{(1)}$ and
$\bar{\bvarphi}_\bu^{(1)}$ can be calculated by
\eqref{equ:2d-varphi1-1} as
\begin{equation} \label{equ:2d-varphi0-bar}
  (\bar{\varphi}_h^{(1)})_\tau =
  \frac{gh^{(0)}}{|\bu^{(0)}|^2-gh^{(0)}}\tilde{B}^{(0)}_\tau,
  \qquad (\bar{\bvarphi}_\bu^{(1)})_\tau =
  -\frac{g\bu^{(0)}}{|\bu^{(0)}|^2-gh^{(0)}}
  \tilde{B}^{(0)}_\tau.
\end{equation}
For the time derivatives of $\hat{\varphi}_{h}^{(1)}$ and
$\hat{\bvarphi}_{\bu}^{(1)}$, we have 
\begin{equation} \label{equ:2d-varphi0-hat}
  \left\{
    \begin{aligned}
      &\nabla \cdot \big(h^{(0)} (\hat{\bvarphi}_\bu^{(1)})_\tau +
      \bu^{(0)} (\hat{\varphi}_h^{(1)})_\tau \big) = 0, \\
      &\cL^{\bu^{(0)}} \left( (\hat{\bvarphi}_\bu^{(1)})_\tau \right)
      + \nabla \big( g(\hat{\varphi}_h^{(1)})_\tau + \bu^{(0),T}
      (\hat{\bvarphi}_\bu^{(1)})_\tau \big) =
      \cL^{\bu^{(0)}}\left(
       \frac{g\bu^{(0)}}{|\bu^{(0)}|^2-gh^{(0)}}\tilde{B}^{(0)}_\tau 
       \right),
    \end{aligned}
  \right.
\end{equation}
by taking a time derivative on \eqref{equ:2d-varphi1-2}. Denote
$\hat{\varphi}_h^{(0)} = \left.  (\hat{\varphi}_h^{(1)})_\tau
\right|_{\tilde{\tau}=0}$ and $\hat{\bvarphi}_\bu^{(0)} = \left.
(\hat{\bvarphi}_\bu)^{(1)})_\tau \right|_{\tilde{\tau}=0}$, and it is
easy to check that  
\[
  \left. \tilde{B}^{(0)}_\tau \right|_{\tilde{\tau}=0} = -\blambda_B^{(0)}
  \nabla B^{(0)} + S_B^{(0)}.
\]
Taking $\tilde{\tau}\rightarrow 0$ and collecting the $\cO(1)$ term in
\eqref{equ:2d-varphi0-equ}, the $\cO(\varepsilon)$ term satisfies 
\begin{equation} \label{equ:2d-varphi0}
  \left\{
  \begin{aligned}
    &\nabla\cdot(h^{(0)}\bvarphi_\bu^{(0)} + \bu^{(0)}\varphi_h^{(0)}) = 
    \frac{gh^{(0)}}{|\bu^{(0)}|^2-gh^{(0)}}(\blambda_B^{(0)} \nabla B^{(0)} - S_B^{(0)}) -
    \hat{\varphi}_h^{(0)}, \\
    &\cL^{\bu^{(0)}}(\bvarphi_\bu^{(0)}) + \nabla
    (\bu^{(0),T}\bvarphi_\bu^{(0)} + g\varphi_h^{(0)}) =
    \frac{g\bu^{(0)}}{|\bu^{(0)}|^2-gh^{(0)}}(-\blambda_B^{(0)} \nabla B^{(0)} +
    S_B^{(0)}) -
    \hat{\bvarphi}_\bu^{(0)},
  \end{aligned}
  \right.
\end{equation}
where $\hat{\varphi}_h^{(0)}$ and $\hat{\bvarphi}_\bu^{(0)}$ satisfy
\begin{equation} \label{equ:2d-varphi0-hat0}
  \left\{
    \begin{aligned}
      &\nabla \cdot (h^{(0)} \hat{\bvarphi}_\bu^{(0)} + \bu^{(0)}
      \hat{\varphi}_h^{(0)}) = 0, \\
      &\cL^{\bu^{(0)}}(\hat{\bvarphi}_\bu^{(0)}) + \nabla (g
      \hat{\varphi}_h^{(0)} + \bu^{(0),T} \hat{\bvarphi}_\bu^{(0)}) =
      \cL^{\bu^{(0)}} \left(
      \frac{g\bu^{(0)}}{|\bu^{(0)}|^2-gh^{(0)}}(-\blambda_B^{(0)} \nabla B^{(0)} +
      S_B^{(0)}) \right).
    \end{aligned}
  \right.
\end{equation}

We also note that the equations \eqref{equ:2d-varphi1-2},
\eqref{equ:2d-varphi0} and \eqref{equ:2d-varphi0-hat0} share the
similar form, which is able to be numerically solved, see Subsection
\ref{sec:scheme-correction}.

\subsubsection{Slow variable correction}
From \eqref{equ:2d-varphi1-1}, \eqref{equ:2d-varphi1-2} and
\eqref{equ:2d-varphi0}, the corrected fast variables in
\eqref{equ:2d-fast} can be obtained. Then, we substitute the fast
variables correction into \eqref{equ:2d-0tau} and omit the high order
term to have 
\[
  \begin{aligned}
    \hat{B}^{(1)}_\tau & + \blambda_B^{(0)}(h^{(1)}, \bu^{(1)})\nabla \hat{B}^{(1)} \\ 
    & + \varepsilon \frac{g\tilde{\lambda}_b(|\bu^{(1)}|)
      (|\bu^{(1)}|^2+gh^{(1)})} {(|\bu^{(1)}|^2-gh^{(1)})^2} \left[
          -\blambda_B^{(0)}(h^{(1)},\bu^{(1)}) \nabla \hat{B}^{(1)}
          +S_B^{(0)}(h^{(1)},\bu^{(1)}) \right] \\ 
    & + \varepsilon
      \frac{\tilde{\lambda}_b(|\bu^{(1)}|)} {|\bu^{(1)}|^2-gh^{(1)}}
      (h^{(1)}\hat{\varphi}_h^{(0)} - \bu^{(1),T}
       \hat{\bvarphi}_\bu^{(0)}) = S_B^{(0)}(h^{(1)}, \bu^{(1)}).
  \end{aligned}
\]

Let $\tilde{\tau} \rightarrow0$, the {\it first order
correction model} for the 2D sediment transport system
\eqref{equ:2d-sediment} can be derived as
\begin{equation} \label{equ:2d-1}
  \begin{aligned}
    B^{(1)}_\tau &+ [\blambda_B^{(1)}(h^{(0)}+\varepsilon
        \varphi_h^{(0)}, \bu^{(0)}+\varepsilon \bvarphi_\bu^{(0)})
    \nabla B^{(1)} \\
    &= S_B^{(1)} (h^{(0)}+\varepsilon \varphi_h^{(0)},
    \bu^{(0)}+\varepsilon \bvarphi_\bu^{(0)}),
  \end{aligned}
\end{equation}
where
\begin{equation}\label{equ:2d-corr-lambda-S}
  \begin{aligned}
    \blambda_B^{(1)}(h,\bu) &=\blambda_B^{(0)}(h,\bu) -\varepsilon\blambda_B^{(0)}(h,\bu)\frac{g(|\bu|^2+gh)
    \tilde{\lambda}_b(|\bu|)}{(|\bu|^2-gh)^2}, \\
    S_B^{(1)}(h,\bu) &= S_B^{(0)}(h,\bu)-\varepsilon\bigg[S_B^{(0)}(h,\bu)\frac{g(|\bu|^2+gh)
  \tilde{\lambda}_b(|\bu|)}{(|\bu|^2-gh)^2} +
  \frac{(h\hat{\varphi}_h^{(0)} - \bu^T\hat{\bvarphi}_\bu^{(0)})
  \tilde{\lambda}_b(|\bu|)}{|\bu|^2-gh}\bigg]. 
  \end{aligned}
\end{equation}

%%%%%%%%%%%%%%%%%%%%%%%%%%%%%%%%%%%%%%%%%%%%%%%%%% 
%% section 5 
%%%%%%%%%%%%%%%%%%%%%%%%%%%%%%%%%%%%%%%%%%%%%%%%%%
\section{Numerical Scheme} \label{sec:scheme}

In this section, we develop the numerical scheme to solve the sediment
transport using the models introduced in the previous sections. Our
numerical scheme basically falls into the framework of HMM method
\cite{weinan2003heterogeneous}, which contains a micro-scale solver,
namely the steady state solver, and a macro-scale solver, namely the
riverbed solver. Meanwhile, our scheme also contains a fast variable
correction which differs from the traditional HMM method. 
%%Firstly we
%%solve the steady state of flow with fixed riverbed at the beginning of
%%every macro step, next we apply the correction and get the prediction
%%of fast variables, at last we use all the information in hand to
%%calculate the riverbed movement in next several macro steps. 
Briefly, the scheme contains three parts: solving the steady state of
flow, calculating the correction term, and solving the equation of the
riverbed. We only focus on last two parts in this section. After
introducing the algorithms to calculate the correction term and solve the
equation of the riverbed, we will give an implementation framework.

\subsection{Calculating the correction term}\label{sec:scheme-correction}
Correction term contains two parts: the $\cO(\tilde{\tau})$ term and
$\cO(\varepsilon)$ term. The former part needs $\tilde{B}^{(0)}$ which
is solved later in section \ref{subsec:advequ}. Here, we assume 
$\tilde{B}^{(0)}$ can be acquired somehow.

%% solve correction 1D %% 
\subsubsection{1D case}
The time correction term could be calculated according to
\eqref{equ:1d-varphi1} analytically after having $\tilde{B}^{(0)}$.
Hence, we only focus on the $\cO(\varepsilon)$ term. For
\eqref{equ:1d-varphi0}, we will implement a method which provides an
inspiration on solving the correction term in 2D case. 

We rewrite \eqref{equ:1d-varphi0} in the following formulation:
\begin{equation}\label{equ:1d-steadystate_varphi}
    \bF(\bvarphi^{(0)})_x = \bS(x),
\end{equation}
where $\bvarphi^{(0)}=(\varphi_h^{(0)}, \varphi_u^{(0)})^T$,
\[
\bF(\bvarphi^{(0)}) = \begin{pmatrix} u^{(0)}\varphi_h^{(0)} +
h^{(0)}\varphi_u^{(0)}\\
g\varphi_h^{(0)} + u^{(0)}\varphi_u^{(0)}
\end{pmatrix} \qquad
\bS(x) = \begin{pmatrix}
\dfrac{gh^{(0)}\lambda_B^{(0)}(h^{(0)},u^{(0)})}{(u^{(0)})^2-gh^{(0)}}B^{(0)}_x\\
\dfrac{gu^{(0)}\lambda_B^{(0)}(h^{(0)},u^{(0)})}{(u^{(0)})^2-gh^{(0)}}B^{(0)}_x
\end{pmatrix}.
\]
Here, $\bS$ is written as the function of $x$ since the steady states
can be computed with fixed $B^{(0)}$.  Further, notice that
\[
\frac{\partial\bF}{\partial \bvarphi^{(0)}} = \begin{pmatrix}
	u^{(0)} & h^{(0)}\\
	g^{(0)} & u^{(0)}\\
\end{pmatrix}
\]
shares the same eigenvalues with 1D shallow water equations.
Therefore, the {\it flux-based wave decomposition method}
\cite{bale2002fluxdecom} can be applied to solve $\bvarphi^{(0)}$. 
More precisely, let $s_{i-1/2}^p, \br_{i-1/2}^p (p=1,2)$ be the
eigenvalues and eigenvectors of Jacobi matrix $\left(\partial
\bF/\partial \bvarphi^{(0)}\right)_{i-1/2}$ respectively. Here,
$\left(\partial \bF/\partial \bvarphi^{(0)}\right)_{i-1/2}$ is
acquired by using the Roe averages $h_{i-1/2}^{(0)}, u_{i-1/2}^{(0)}$.
To make it more concise, we omit the superscript $^{(0)}$ below in
this subsection. The algorithm is described as follows:

\begin{enumerate}[leftmargin=*]
\item Decompose the fluxes as
\[
  \bF_{i}-\bF_{i-1} =\sum_{p=1}^2 \alpha_{i-1/2}^p s_{i-1/2}^p
  \br_{i-1/2}^p.
\]
where 
\begin{gather*}
s_{i-1/2}^1=u_{i-1/2}+\sqrt{gh_{i-1/2}} \qquad
s_{i-1/2}^2=u_{i-1/2}-\sqrt{gh_{i-1/2}},\\
\br_{i-1/2}^1 = \begin{pmatrix} 
\sqrt{\dfrac{h_{i-1/2}}{g}} \\ 1 
\end{pmatrix} \qquad
\br_{i-1/2}^2 = \begin{pmatrix} 
-\sqrt{\dfrac{h_{i-1/2}}{g}} \\ 1 
\end{pmatrix},
\end{gather*}
and 
\[
\begin{pmatrix}
\alpha_{i-1/2}^1s_{i-1/2}^1 \\ \alpha_{i-1/2}^2 s_{i-1/2}^2
\end{pmatrix}
=\begin{pmatrix}
\dfrac{1}{2}\sqrt{\dfrac{g}{h_{i-1/2}}} & \dfrac{1}{2} \\
-\dfrac{1}{2}\sqrt{\dfrac{g}{h_{i-1/2}}} & \dfrac{1}{2}
\end{pmatrix}
\begin{pmatrix}
u_i\varphi_{h,i}+h_i\varphi_{u,i}-u_{i-1}\varphi_{h,i-1}-h_{i-1}\varphi_{u,i-1}\\
g\varphi_{h,i}+u_{i}\varphi_{u,i}-g\varphi_{h,i-1}-u_{i-1}\varphi_{u,i-1}
\end{pmatrix}.
\]
\item Calculate wave fluctuations by
\[
	\bF_{i-1/2}^{\pm} :=
	\sum_{p=1}^2(s_{i-1/2}^p)^{\pm}\alpha_{i-1/2}^pr_{i-1/2}^p, 
\]
where $(s)^+=\max(s,0), (s)^-=\min(s,0)$. By the subcritical
assumption, 
\[
\bF_{i-1/2}^+ = \alpha_{i-1/2}^1
\begin{pmatrix}\sqrt{\dfrac{h_{i-1/2}}{g}}\\1 \end{pmatrix} \qquad 
\bF_{i+1/2}^- = \alpha_{i+1/2}^2
\begin{pmatrix}-\sqrt{\dfrac{h_{i+1/2}}{g}}\\ 1\end{pmatrix}.
\]

\item Solve the algebraic linear system
\begin{equation} \label{equ:1d-solve}
\bF_{i-1/2}^+ + \bF_{i+1/2}^- = \Delta x \bS_{i},
\end{equation}
%%After tedious calculation, we obtain
%%\[
%%\bF_{i-1/2}^+
%%=
%%\begin{pmatrix}
%%A_i^+\varphi_{h,i}+B_i^+\varphi_{u,i}+C_i^+\varphi_{h,i-1}+D_i^+
%%\varphi_{u,i-1}\\
%%E_i^+\varphi_{h,i}+F_i^+\varphi_{u,i}+G_i^+\varphi_{h,i-1}+H_i^+\varphi_{u,i-1}
%%\end{pmatrix}
%%\]
%%where
%%\begin{gather*}
%%A_i^+=\frac{1}{2}(u_i+\sqrt{g h_{i-1/2}}) \qquad B_i^+ = \frac{1}{2}\left(h_i+
%%\sqrt{\frac{h_{i-1/2}}{g}}u_i\right),\\
%%C_i^+=-\frac{1}{2}(u_{i-1}+\sqrt{gh_{i-1/2}})\qquad
%%D_i^+=-\frac{1}{2}\left(h_{i-1}+\sqrt{\frac{h_{i-1/2}}{g}}u_{i-1}\right),\\
%%E_i^+=\frac{1}{2}\left(\sqrt{\frac{g}{h_{i-1/2}}}+g\right) \qquad F_i^+=
%%\frac{1}{2}\left(\sqrt{\frac{g}{h_{i-1/2}}}h_i+u_i\right),\\
%%G_i^+=\frac{1}{2}\left(\sqrt{\frac{g}{h_{i-1/2}}}u_{i-1}+g\right),\quad
%%H_i^+=-\frac{1}{2}\left(\sqrt{\frac{g}{h_{i-1/2}}}h_{i-1}+u_{i-1}\right).
%%\end{gather*}
%%\[
%%(\bF_{i}-\bF_{i-1})^-
%%=
%%\begin{pmatrix}
%%A_i^-\varphi_{h,i}+B_i^-\varphi_{u,i}+C_i^-\varphi_{h,i-1}+D_i^-
%%\varphi_{u,i-1}\\
%%E_i^-\varphi_{h,i}+F_i^-\varphi_{u,i}+G_i^-\varphi_{h,i-1}+H_i^-\varphi_{u,i-1},
%%\end{pmatrix}
%%\]
%%where
%%\begin{gather*}
%%A_i^-=\frac{1}{2}(u_i-\sqrt{g h_{i-1,i}}),\quad B_i^- = \frac{1}{2}
%%\left(h_i-\sqrt{\frac{h_{i-1,i}}{g}}u_i\right),\\
%%C_i^-=-\frac{1}{2}(u_{i-1}-\sqrt{gh_{i-1,i}}),\quad D_i^-=-\frac{1}{2}\left(h_{i-1}-\sqrt{\frac{h_{i-1,i}}{g}}u_{i-1}\right),\\
%%E_i^-=-E_i^+,\quad F_i^-=F_i^-,\quad G_i^-=-G_i^+,\quad H_i^-=-H_i^+.
%%\end{gather*}
where the central difference is used to discretize $\bS_i$: 
\[
    \bS_i=\begin{pmatrix}
        S_{1,i}\\
        S_{2,i}
    \end{pmatrix} \quad  
    S_{1,i} = \frac{g h_i\lambda_B(h_i,u_i)}{u_i^2-gh_i}\cdot\frac{B_{i+1}-
    B_{i-1}}{2\Delta x}\quad S_{2,i}=-\frac{u_i}{h_i}S_{1,i}.
\]
%%Then the \eqref{equ:1d-solve} can be written as
%%\[
%%    \begin{split}
%%        &\begin{pmatrix}
%%        A_{i+1}^-\varphi_{h,i+1}+B_{i+1}^-\varphi_{u,i+1}+
%%        (A_i^++C_{i+1}^-)\varphi_{h,i}+(B_i^++D_{i+1}^-)\varphi_{u,i}
%%        +C_i^+\varphi_{h,i-1}+D_i^+\varphi_{u,i-1}\\
%%        E_{i+1}^-\varphi_{h,i+1}+F_{i+1}^-\varphi_{u,i+1}+
%%        (E_i^++G_{i+1}^-)\varphi_{h,i}+(F_i^++H_{i+1}^-)\varphi_{u,i}+
%%        G_i^+\varphi_{h,i-1}+H_i^+\varphi_{u,i-1}
%%    \end{pmatrix}\\
%%    &=\begin{pmatrix}
%%    \Delta xS_{1,i}\\
%%    \Delta xS_{2,i}
%%\end{pmatrix}.
%%    \end{split}
%%\]
\end{enumerate}
Note that $\lambda_B(h,u) = 0$ if $u=0$ from \eqref{equ:1d-0}, then
the scheme is naturally well-balanced. Further, zero boundary
condition is enforced in a computational domain $[a,b]$, i.e.
$\varphi_{h}(a)=\varphi_{h}(b)=0$ and $\varphi_{u}(a) =
\varphi_{u}(b)=0$. For the linear system, the BiCGSTAB solver is
used whose parameters will be specified in the numerical test.

%% solve collection 2D %% 
\subsubsection{2D case}\label{sec:2d-case-correction}
For 2D case, We need to combine \eqref{equ:2d-varphi1-2} with
\eqref{equ:2d-varphi1-1} to obtain the $\mathcal{O}(\tilde{\tau})$
term, and to solve \eqref{equ:2d-varphi0-hat} and
\eqref{equ:2d-varphi0-bar} to obtain the $\cO(\varepsilon)$ term. 
Notice that \eqref{equ:2d-varphi1-1}, \eqref{equ:2d-varphi0-bar}, and
\eqref{equ:2d-varphi0-hat} are of the same form as follow,
\begin{equation}\label{equ:varphi-old-form}
\left\{
\begin{aligned}
	\nabla\cdot(h^{(0)}\bphi_\bu+\bu^{(0)}\phi_h) &= S_h, \\
  \cL^{\bu^{(0)}}(\bphi_\bu)+\nabla(g\phi_h+\bu^{(0)T}\bphi_\bu)
  &= \bS_\bu.
\end{aligned}
\right.
\end{equation}
Thus, we only present the numerical scheme to solve
\eqref{equ:varphi-old-form}, which can be written as
\[
\begin{pmatrix}
	h^{(0)}\phi_u + u^{(0)}\phi_h\\
	u^{(0)}\phi_u + v^{(0)}\phi_v + g\phi_h\\
	0
\end{pmatrix}_x
+
\begin{pmatrix}
	h^{(0)}\phi_v + v^{(0)}\phi_h\\
	0\\
	u^{(0)}\phi_u + v^{(0)}\phi_v + g\phi_h
\end{pmatrix}_y
=\begin{pmatrix}
	S_h\\
	\bS_\bu-\cL^{\bu^{(0)}}(\bphi_\bu)
\end{pmatrix}.
\]
This form is inappropriate to solve due to the degeneration of the
fluxes. To fix it, we add an additional term to both sides:
\[
	\begin{pmatrix}
		0\\
		(v^{(0)}\phi_u)_y-(v^{(0)}\phi_v)_x\\
		(u^{(0)}\phi_v)_x-(u^{(0)}\phi_u)_y
	\end{pmatrix}. 
\]
It is interesting to note that this fixing term degenerates to zero
for 1D case. Therefore, \eqref{equ:varphi-old-form} can be recast as  
\begin{equation}\label{equ:varphi-new-form}
	\bF(\bphi,x,y)_x+\bG(\bphi,x,y)_y = \tilde{\bS},
\end{equation}
where $\bphi = (h, \bphi_\bu)^T$, 
\[
\bF(\bphi,x,y)=\begin{pmatrix}
	h^{(0)}\phi_u + u^{(0)}\phi_h\\
	u^{(0)}\phi_u + g\phi_h\\
	u^{(0)}\phi_v
\end{pmatrix}, \qquad
\bG(\bphi,x,y)=\begin{pmatrix}
	h^{(0)}\phi_v + v^{(0)}\phi_h\\
	v^{(0)}\phi_u\\
	v^{(0)}\phi_v+g\phi_h
\end{pmatrix},
\]
and 
\[
\tilde{\bS}=\begin{pmatrix}
	S_h\\
  \bS_\bu - \cL^{\bu^{(0)}}(\bphi_\bu) + \cL_f^{\bu^{(0)}}(\bphi_u)
\end{pmatrix}, \qquad
\cL_f^{\bu^{(0)}}(\bphi_u)=\begin{pmatrix}
	v^{(0)}_y\phi_u-u^{(0)}_y\phi_v\\
	u^{(0)}_x\phi_v-v^{(0)}_x\phi_u
\end{pmatrix}.
\]
We note again that $\partial\bF/\partial\bphi,
\partial\bG/\partial\bphi$ share the same eigenvalues with the 2D
shallow water equations. Denote $s_{i-1/2,j}^p, \br_{i-1/2,j}^p
(p=1,2,3)$ the eigenvalues and eigenvectors of Roe-averaged
$\left(\partial\bF/\partial \bphi\right)_{i-1/2,j}$, and
$s_{i,j-1/2}^p, \br_{i,j-1/2}^p (p=1,2,3)$ the eigenvalues and
eigenvectors of Roe-averaged $\left(\partial \bG/\partial \bphi
\right)_{i,j-1/2}$. To be concise, we omit the
superscript $^{(0)}$ below in this subsection. The detailed algorithm
for \eqref{equ:varphi-old-form} is detailed as follows:

\begin{enumerate}[leftmargin=*]
\item Decompose fluxes as
\[
\begin{aligned}
\bF_{i,j}-\bF_{i-1,j} &=
\sum_{p=1}^3\alpha_{i-1/2,j}^ps_{i-1/2,j}^p \br_{i-1/2,j}^p,\\
\bG_{i,j}-\bG_{i,j-1} &=
\sum_{p=1}^3\alpha_{i,j-1/2}^ps_{i,j-1/2}^p \br_{i,j-1/2}^p,
\end{aligned}
\]
where 
$$ 
\begin{aligned}
& s_{i-1/2,j}^1 = u_{i-1/2,j} + \sqrt{gh_{i-1/2,j}}, \quad s_{i-1/2,j}^2 =
u_{i-1/2,j}, \quad s_{i-1/2,j}^3 = u_{i-1/2,j} -
\sqrt{gh_{i-1/2,j}},\\
& \br_{i-1/2,j}^1 = \begin{pmatrix}
\sqrt{\frac{h_{i-1/2,j}}{g}} \\ 1 \\ 0
\end{pmatrix},
\qquad \quad \quad  
\br_{i-1/2,j}^2 = \begin{pmatrix}
0 \\ 0 \\ 1
\end{pmatrix},
\qquad \quad
\br_{i-1/2,j}^3 = \begin{pmatrix}
-\sqrt{\frac{h_{i-1/2,j}}{g}} \\ 1 \\ 0
\end{pmatrix},
\end{aligned}
$$ 
and
$$ 
\begin{aligned}
& s_{i,j-1/2}^1 = v_{i,j-1/2} + \sqrt{gh_{i,j-1/2}}, \quad s_{i,j-1/2}^2 =
v_{i,j-1/2}, \quad s_{i,j-1/2}^3 = v_{i,j-1/2} -
\sqrt{gh_{i,j-1/2}},\\
& \br_{i,j-1/2}^1 = \begin{pmatrix}
\sqrt{\frac{h_{i,j-1/2}}{g}} \\ 0 \\ 1
\end{pmatrix},
\qquad \quad \quad  
\br_{i,j-1/2}^2 = \begin{pmatrix}
0 \\ 1 \\ 0
\end{pmatrix},
\qquad \quad
\br_{i,j-1/2}^3 = \begin{pmatrix}
-\sqrt{\frac{h_{i,j-1/2}}{g}} \\ 0 \\ 1
\end{pmatrix}.
\end{aligned}
$$ 
\item Calculate wave fluctuations by
\[
  \bF_{i-1/2,j}^{\pm} = \sum_{p=1}^3(s_{i-1/2,j})^{\pm}
  \alpha_{i-1/2,j}^p \br_{i-1/2,j}^p, \qquad 
	\bG_{i,j-1/2}^{\pm} = \sum_{p=1}^3(s_{i,j-1/2})^{\pm}
	\alpha_{i,j-1/2}^p \br_{i,j-1/2}^p.
\]
\item Solve the algebraic linear system
\[
	(\bF_{i-1/2,j}^+ + \bF_{i+1/2,j}^- )\Delta y
  +(\bG_{i,j-1/2}^+ + \bG_{i,j+1/2})^- \Delta x = \Delta x\Delta y
  \tilde{\bS}_{i,j},
\]
where the central difference is used to discretize
$\tilde{\bS}_{i,j}$.  Similar to the 1D case, we need to write
$\alpha^p$ in terms of  $\bphi_{i\pm1,j\pm1}$ to obtain a  linear
system. We also note that $\tilde{\bS}_{i,j}$ may contain $\bphi$,
the corresponding terms of which should be moved to left
hand side when building the linear system. Further, zero boundary
condition is enforced as the 1D case.
\end{enumerate}

The eigenvalues of $\partial \bF/\partial\bphi$ and $\partial
\bG/\partial\bphi$ cannot guaranteed to be away from zero under the
subcritical assumption. Therefore,  we make usage of the {\it Harten's
entropy fix} \cite{harten1983highresolution} to stablize algorithm.
The wave fluctuations after the entropy fix are as follows:
\[
\bF_{i-1/2,j}^{\pm}=\sum_{p=1}^3(s_{i-1/2,j})^{\pm}
\alpha_{i-1/2,j}^pr_{i-1/2,j}^p
\pm\frac{1}{2}\hat{\bM}_{i-1/2,j}
(\bphi_{i,j}-\bphi_{i-1,j}),
\]
where
\[
\begin{aligned}
& \hat{\bM}_{i-1/2,j} = \bR_{i-1/2,j}\mathrm{diag}
	\{\rho_{i-1/2,j}^p\} \bR_{i-1/2,j}^{-1},\\
	\rho_{i-1/2,j}^p &= \begin{cases}
		0, & \text{if }|s_{i-1/2,j}^p|>\delta, \\
		[(s_{i-1/2,j}^p)^2+\delta^2]/(2\delta) -|s_{i-1/2,j}^p|, & \text{otherwise,}
	\end{cases} \qquad p = 1,2,3.
\end{aligned}
\]
Here, $\delta$ is a small positive constant, $\bR_{i-1/2,j} =
[\br_{i-1/2,j}^1, \br_{i-1/2,j}^2, \br_{i-1/2,j}^3]$. We apply the entropy fix
in the $y$ direction in the same way.

\begin{remark}
In the subcritical assumption, the eigenvalues in 1D case are
guaranteed to be away from zero. Thus, entropy fix is not applied in
1D case.
\end{remark}

\begin{remark}
Here, we only use the first order scheme to solve the correction
terms in consideration of the accuracy. Specifically, the
$\mathcal{O}(\varepsilon)$ correction term $\varphi_h^{(0)}$ always
has the contribution as $\varepsilon\varphi_h^{(0)}$, whose error is
$\cO(\varepsilon \Delta x)$ that consistent with the overall error.   
\end{remark}

\begin{remark}
If $\bu^{(0)} = \bzero$, then $\bar{\bvarphi}_{\bu}^{(1)} = \bzero$
from \eqref{equ:2d-varphi1-1}, which implies that $\hat{\varphi}_h^{(1)} =
0$ and $\hat{\bvarphi}_h^{(1)} = 0$ in \eqref{equ:2d-varphi1-2}.
Further, $\bu^{(0)} = \bzero$ implies that $\blambda_B^{(0)} = \bzero$
and $S_B^{(0)} = 0$ from \eqref{equ:2d-lambda-S}. Consequently, we
have $\hat{\varphi}_h^{(0)} = 0$ and $\hat{\bvarphi}_\bu^{(0)} =
\bzero$ in \eqref{equ:2d-varphi0-hat0}, thus $\varphi_h^{(0)} = 0$ and
$\bvarphi_\bu^{(0)} = \bzero$ in \eqref{equ:2d-varphi0}. Therefore,
the 2D scheme is well-balanced.
\end{remark}

\subsection{Solving the riverbed equation}\label{subsec:advequ}
%Notice that in \eqref{equ:1d-correction} and \eqref{equ:2d-fast},
%expansions are carried out for every given time $\tau_0$ and we
%omitted $\tau_0$ there for conciseness. Because $\tau_0$ represents
%every given time, we write $\tau_0$ as $\tau$
%and this tells us that the
%characteristic speed of the riverbed is the function of both $x$ and
%$\tau$. 
%To show it more clearly,  we consider the first order model of 2D case as 
%an example. By \eqref{equ:2d-1},
%\eqref{equ:2d-corr-lambda-S}, we have
%$\blambda=\blambda^{(1)}(h^{(0)}+\varepsilon\varphi_h^{(0)},\bu^{(0)}
%+\varepsilon\bvarphi_\bu^{(0)}),S=S^{(1)}(h^{(0)}+\varepsilon\varphi_h^{(0)},\bu^{(0)}+\varepsilon\bvarphi_\bu^{(0)})$.
%$h^{(0)}(x,\tau),\bu^{(0)}(x,\tau)$ represent the steady state of shallow water
%equation when $B$ is fixed to $B(x,\tau)$, $\varphi_h^{(0)},\bvarphi_\bu^{(0)}$ is 
%the function of $h^{(0)},\bu^{(0)}$ according to \eqref{equ:2d-varphi0} and 
%\eqref{equ:2d-varphi0-hat}. Thus we $\blambda,S$ are all function of 
%$h^{(0)},\bu$
The homogenized models of both zeroth order and first order can be
written in the following common form
\begin{equation} \label{equ:1d-adv}
B_{\tau}+\lambda B_x=0
\end{equation}
for 1D case and 
\begin{equation}
B_{\tau}+\blambda \cdot\nabla B=S
\label{equ:2d-adv}
\end{equation}
for 2D case, respectively.  It suffices to describe the scheme for
\eqref{equ:2d-adv} since the numerical scheme for 1D case is a
simplification of that for 2D case.  In light of \eqref{equ:2d-1} and
\eqref{equ:2d-corr-lambda-S}, we have
$\blambda=\blambda^{(1)}(h^{(0)}+\varepsilon\varphi_h^{(0)},\bu^{(0)}
+\varepsilon\bvarphi_\bu^{(0)}),S=S^{(1)}(h^{(0)}+\varepsilon\varphi_h^{(0)},\bu^{(0)}+\varepsilon\bvarphi_\bu^{(0)})$.
Here, $h^{(0)}(\bx,\tau),\bu^{(0)}(\bx,\tau)$ represent the steady
state of shallow water equations when $B$ is fixed to $B(x,\tau)$,
and $\varphi_h^{(0)},\bvarphi_\bu^{(0)}$ are the functions of
$h^{(0)},\bu^{(0)}$ according to \eqref{equ:2d-varphi0} and
\eqref{equ:2d-varphi0-hat}. 

First, assume $\blambda$ and $S$ are known. We modify the second order
{\it TVD Runge-Kutta scheme} \cite{gottlieb1998tvdrungekutta} to solve
\eqref{equ:2d-adv} as \begin{equation}\label{equ:riverbed-1}
	\begin{aligned}
		\tilde{B}_{i,j}^{n+1} = B_{i,j}^n&-\frac{\Delta \tau}{\Delta x}
		\lambda_{i,j}^{x,n}(B_{i,j}^{n,R}-B_{i,j}^{n,L}) \\
		&-\frac{\Delta\tau}{\Delta y}
		\lambda_{i,j}^{y,n}(B_{i,j}^{n,U}-B_{i,j}^{n,D})+
		\Delta\tau S_{i,j}^n,
	\end{aligned}
\end{equation}
\begin{equation}\label{equ:riverbed-2}
	\begin{aligned}
		B_{i,j}^{n+1} = \frac{1}{2}(B_{i,j}^n+\tilde{B}_{i,j}^{n+1})&-
	\frac{\Delta \tau}{2 \Delta x}
	\lambda_{i,j}^{x,n+1}
	(\tilde{B}_{i,j}^{n+1,R}-\tilde{B}_{i,j}^{n+1,L}) \\
	&-\frac{\Delta\tau}{2\Delta y}
	\lambda_{i,j}^{y,n+1}(\tilde{B}_{i,j}^{n+1,U}-\tilde{B}_{i,j}^{n+1,D})+
	\frac{\Delta\tau}{2} S_{i,j}^{n+1},
\end{aligned}
\end{equation}
where 
$$
\begin{aligned}
  B_{i,j}^{n,L} &=f^{\text{upwind}}(B_{i-1/2,j}^{n,L}, B_{i-1/2,j}^{n,R},
  \lambda_{i,j}^{x,n}), \quad B_{i,j}^{n,R}
  =f^{\text{upwind}}(B_{i+1/2,j}^{n,L},B_{i+1/2,j}^{n,R},
      \lambda_{i,j}^{x,n}),\\
  B_{i,j}^{n,D} &=f^{\text{upwind}}(B_{i,j-1/2}^{n,D},
      B_{i,j-1/2}^{n,U}, \lambda_{i,j}^{y,n}), \quad 
  B_{i,j}^{n,U} =f^{\text{upwind}}(B_{i,j+1/2}^{n,D},
      B_{i,j+1/2}^{n,U}, \lambda_{i,j}^{y,n}).
\end{aligned}
$$ 
Here, $f^{\text{upwind}}$ is an upwind flux function that 
\[
	f^{\text{upwind}}(a,b,\lambda)=\begin{cases}
		a,\quad\text{if }\lambda>0,\\
		b,\quad\text{if }\lambda<0.
	\end{cases}
\]

To achieve the second order spatial discretization, we apply the {\it
MUSCL-type slope limiter} \cite{vanleer1979muscl} to obtain  
\[
\begin{aligned}
	B_{i-1/2,j}^{n,L}&=B_{i-1,j}^n+\frac{1}{2}\phi(r_{i-1,j}^{x,n})
	(B_{i,j}^n-B_{i-1,j}^n),\\
	B_{i-1/2,j}^{n,R}&=B_{i,j}^n-\frac{1}{2}\phi(r_{i,j}^{x,n})
	(B_{i+1,j}^n-B_{i,j}^n),
\end{aligned}
\]
where $r_{i,j}^{x,n}=(B_{i,j}^n-B_{i-1,j})/(B_{i+1,j}^n-B_{i,j}^n)$
and $\phi(r)=\max(0,\min(1,r))$ is the minmod limiter. Discretization
on $y$ direction takes the same form. In multidimensional cases, this
slope limiter scheme may bring spurious oscillations in regions with
large gradients in conservation laws. When having a riverbed with
sharp shape, we can require the limiter to satisfy a limiting
condition in \cite{kim2005mlp} by setting the {\it MLP-type limiter}
as an upper bound.

It remains to show how to obtain $\blambda^{n},\blambda^{n+1}$ and
$S^n,S^{n+1}$. Based on $B^{n}$, the steady states  $h^{(0),n}$ and
$\bu^{(0),n}$ can be computed, as well as $\varphi_h^{(0),n},
\bvarphi_\bu^{(0),n}$ due to \eqref{equ:2d-1} and
\eqref{equ:2d-corr-lambda-S}. We note that slope limiters of the
$h^{(0)}$ and $\bu^{(0)}$ are applied to calculate the source term
$S^n$.  For $\blambda^{n+1}, S^{n+1}$, one option is to repeat the
above procedure when fixing $B$ to $\tilde{B}^{n+1}$. Another option
is to apply the $\mathcal{O}(\tilde{\tau})$ correction to approximate
the desired
terms 
\begin{equation} \label{equ:runge-kutta-approximation}
	\begin{aligned}
  h^{(0),n+1}+\varepsilon\varphi_h^{(0),n+1} &\approx h^{(0),n}+
  \varphi_h^{(1),n}(\Delta\tau) + \varepsilon\varphi_h^{(0),n},\\
  \bu^{(0),n+1}+\varepsilon\varphi_\bu^{(0),n+1} &\approx
  \bu^{(0),n}+ \varphi_\bu^{(1),n}(\Delta\tau) +
  \varepsilon\varphi_\bu^{(0),n},\\
	\end{aligned}
\end{equation}
where $\varphi_h^{(1),n},\bvarphi_\bu^{(1),n}$ can be acquired by
$B^{n},\tilde{B}^{n+1},h^{(0),n},\bu^{(0),n}$  by
\eqref{equ:2d-varphi1-1} and \eqref{equ:2d-varphi1-2}.
Here, we use $\varepsilon\bvarphi^{(0),n}$ to approximate
$\varepsilon\bvarphi^{(0),n+1}$ with error
$\cO(\varepsilon\Delta\tau)$, and use $h^{(0),n}+
\varphi_h^{(1),n}(\Delta\tau)$ to approximate $h^{(0),n+1}$ with error
$\cO(\Delta \tau^2)$ (the same with $\bu$).

%%We point out that once we calculate the correction term using
%%$\tilde{B}^{n+1}$, one has already got the second order accuracy
%%correction and there is no need to calculate them again after we get
%%$B^{n+1}$.
The error can be roughly estimated as below. First, the error of
$\blambda^{n+1}$ and $S^{n+1}$ are of order
$\mathcal{O}(\varepsilon^2 + \varepsilon\Delta\tau + \Delta\tau^2)$.
Since the TVD Runge-Kutta scheme is applied up to time
$\mathcal{O}(1)$ in $\tau$ scale, the error in computing the riverbed
is as $\mathcal{O}(\Delta x^2+\Delta \tau ^2)$. Hence, the total error
is approximately of order $\mathcal{O}(\varepsilon^2 + \varepsilon
\Delta \tau + \Delta x^2 + \Delta \tau^2)$. Then, by the CFL condition
we have that $\Delta \tau \times (\text{speed of the riverbed
evolving}) \sim \Delta x$, the total error is of order
\begin{equation}\label{equ:error-nondim}
\mathcal{O}(\varepsilon^2+\varepsilon\Delta x+\Delta x^2).
\end{equation}
As shown above, we solve the steady state $h^{(0),n}$, and then use
$h^{(0),n}+ \varphi_h^{(1),n}(\Delta\tau)$ to approximate
$h^{(0),n+1}$.  Actually, such approximation can be repeated for
several successive steps, i.e. using $h^{(0),n+1}+
\varphi_h^{(1),n+1}(\Delta\tau)$ to approximate $h^{(0),n+2}$.  In a
practical simulation, in order to make the computation more efficient,
we will apply this approximation for fixed steps (denote as $K$
later on, $K$ is not big, say 2 or 3), i.e. we only solve steady state
for $h^{(0),n}$ and approximate $h^{(0),n+1},\cdots,h^{(0),n+K}$ through
time correction term. During these steps, we use the same
$\mathcal{O}(\varepsilon)$ correction $\varepsilon \varphi_h^{(0),n}$,
which does not affect the overall error.

%% algorithm %% 
\subsection{Sediment transport algorithm}\label{subsec:alg}
After all the preparations above, we are ready to give the  
second order algorithm for sediment transport. We will only give the
algorithm in 2D case below for conciseness.
\begin{description}
\item[Step 1] {\bf Initialization}: Let $t=0, n=0$, and set the
  initial data $B^0$. Give a positive integer $K$ and we will take $K$
  macro steps forward for every sample. 
\item[Step 2] {\bf  Sampling and calculating the $\cO(\varepsilon)$ term
  correction.}
  \begin{itemize}[leftmargin=*]
    \item Sampling: Fix $B=B^n$, apply the steady state solver to
    obtain $h^{(0),n}$, $\bu^{(0),n}$. 
  \item Solve \eqref{equ:2d-varphi0-hat0} to obtain
  $\hat{\varphi}_h^{(0),n},\hat{\bvarphi}_\bu^{(0),n}$:
  \[ 
  \left\{
    \begin{aligned}
      &\nabla \cdot (h^{(0),n} \hat{\bvarphi}_\bu^{(0),n} + \bu^{(0),n}
      \hat{\varphi}_h^{(0),n}) = 0, \\
      &\cL^{\bu^{(0)},n}(\hat{\bvarphi}_\bu^{(0),n}) + \nabla (g
    \hat{\varphi}_h^{(0),n} + (\bu^{(0),n})^T \hat{\bvarphi}_\bu^{(0),n}) =
      \cL^{\bu^{(0)},n} \left(
      \frac{g\bu^{(0),n}}{|\bu^{(0),n}|^2-gh^{(0),n}}(-\blambda_B^{(0),n} \nabla B^{n} +
      S_B^{(0),n}) \right).
    \end{aligned}
  \right.
\]
where $\blambda_B^{(0),n}=\blambda_B^{(0)}(h^{(0),n},\bu^{(0),n}), 
S_B^{(0),n}=S_B^{(0)}(h^{(0),n},\bu^{(0),n})$ by \eqref{equ:2d-lambda-S}.
\item Solve \eqref{equ:2d-varphi0} to obtain  $\varphi_h^{(0),n},\bvarphi_\bu^{(0),n}$
    \[
  \left\{
  \begin{aligned}
    &\nabla\cdot(h^{(0),n}\bvarphi_\bu^{(0),n} + \bu^{(0)}\varphi_h^{(0),n}) = 
    \frac{gh^{(0),n}}{|\bu^{(0),n}|^2-gh^{(0),n}}(\blambda_B^{(0),n} \nabla B^n - S_B^{(0),n}) -
    \hat{\varphi}_h^{(0),n}, \\
    &\cL^{\bu^{(0),n}}(\bvarphi_\bu^{(0),n}) + \nabla
    (\bu^{(0),T}\bvarphi_\bu^{(0),n} + g\varphi_h^{(0),n}) =
    \frac{g\bu^{(0),n}}{|\bu^{(0),n}|^2-gh^{(0),n}}(-\blambda_B^{(0),n} \nabla B^n +
    S_B^{(0),n}) -
    \hat{\bvarphi}_\bu^{(0),n}.
  \end{aligned}
  \right.
\] 
\item Apply the $\cO(\varepsilon)$ correction: Let $m=0$, $t = t^n$
and 
  $$ 
    B^{n,0}=B^{n}, \qquad h^{n,0}=h^{(0),n}+\varepsilon\varphi_h^{(0),n},
    \qquad \bu^{n,0} = \bu^{(0),n}+\varepsilon\bvarphi_\bu^{(0),n}.
  $$
  \end{itemize}
\item[Step 3] {\bf Riverbed prediction}
    \begin{itemize}[leftmargin=*]
     \item Use $h^{n,m},\bu^{n,m}$ to calculate characteristic speed
     $\blambda^{n,m}$ and source term $S^{n,m}$ according to
     \eqref{equ:2d-corr-lambda-S}:
    \[
    \blambda^{n,m} =\blambda_B^{(1)}(h^{n,m},\bu^{n,m}),\quad
     S^{n,m}=S_B^{(1)}(h^{n,m},\bu^{n,m}).
    \]
\item Calculate $\tilde{B}^{n,m+1}$ using \eqref{equ:riverbed-1}:
    \[
    \begin{aligned}
        \tilde{B}_{i,j}^{n,m+1} = B_{i,j}^{n,m}&-\frac{\Delta \tau}{\Delta x}
		\lambda_{i,j}^{x,n,m}(B_{i,j}^{n,m,R}-B_{i,j}^{n,m,L}) \\
		&-\frac{\Delta\tau}{\Delta y}
		\lambda_{i,j}^{y,n,m}(B_{i,j}^{n,m,U}-B_{i,j}^{n,m,D})+
        \Delta\tau S_{i,j}^{n,m}.
	\end{aligned}
\]
Here, the time step $\Delta\tau^{n,m}$ is determined by the CFL
condition, namely
  \[
  \Delta\tau^{n,m} = 
  C_{\mathrm{cfl}}\cdot\frac{1}{\max\limits_{i,j}\{
  |\lambda_{i,j}^{x,n,m}|/\Delta x+|\lambda_{i,j}^{y,n,m}|/\Delta y\}}.
  \]
  where $0<C_{\mathrm{cfl}}<1$.
  \end{itemize}
\item[Step 4] {\bf Approximate the steady state by time correction}
    \begin{itemize}[leftmargin=*]
    \item Let $\bar{h}^{n,m}=h^{n,m}-\varepsilon\varphi_h^{(0),n},
          \bar{\bu}^{n,m}=\bu^{n,m}-\varepsilon\bvarphi_\bu^{(0),n}$.
          In this step, we use $\bar{h}^{n,m},\bar{\bu}^{n,m}$ other
          than $h^{n,m},\bu^{n,m}$ to approximate the steady states. 
        \item Solve $\bar{\varphi}_h^{(1),n,m},\bar{\bvarphi}_\bu^{(1),n,m}$ by 
            \eqref{equ:2d-varphi1-1}
            \begin{gather*}
            \bar{\varphi}_h^{(1),n,m} = \frac{g\bar{h}^{n,m}}{|\bar{\bu}^{n,m}|^2-
            g\bar{h}^{n,m}}(\tilde{B}^{n,m+1}-B^{n,m}), \\
              \bar{\bvarphi}_\bu^{(1),n,m} =
              -\frac{g\bar{\bu}^{n,m}}{|\bar{\bu}^{n,m}|^2-g\bar{h}^{n,m}}(\tilde{B}^{n,m+1}-
                  B^{n,m}).
  \end{gather*}
  \item Solve $\hat{\varphi}_h^{(1),n,m},\hat{\bvarphi}_\bu^{(1),n,m}$ by 
      \eqref{equ:2d-varphi1-2}:
      \[
    \left\{
    \begin{aligned}
    &\nabla \cdot (\bar{h}^{n,m}\hat{\bvarphi}_\bu^{(1),n,m} +
        \bar{\bu}^{n,m}\hat{\varphi}_h^{(1),n,m}) = 0,\\
        &\cL^{\bar{\bu}^{n,m}}(\hat{\bvarphi}_\bu^{(1),n,m}) + \nabla \cdot
        (g\hat{\varphi}_h^{(1),n,m} + (\bu^{n,m})^T\hat{\bvarphi}_\bu^{(1),n,m}) =
        -\cL^{\bar{\bu}^{n,m}}(\bar{\bvarphi}_\bu^{(1),n,m}).
    \end{aligned}
    \right.
  \]
  \item Update the steady state:
    \[
    h^{n,m+1}= h^{n,m}+\bar{\varphi}_h^{(1),n,m}+
    \hat{\varphi}_h^{(1),n,m} + \varepsilon \varphi_h^{(0),n},
    \qquad 
    \bu^{n,m+1} = \bar{\bvarphi}_\bu^{(1),n,m}+
    \hat{\bvarphi}_{\bu}^{(1),n,m} + \varepsilon \varphi_\bu^{(0),n}.
    \]
  \end{itemize}
\item[Step 5] {\bf Riverbed correction}
    \begin{itemize}[leftmargin=*]
        \item Calculate $\blambda^{n,m+1},S^{n,m+1}$ using $h^{n,m+1},\bu^{n,m+1}$
            according to \eqref{equ:2d-corr-lambda-S}.
        \item Update riverbed $B^{n,m+1}$ by \eqref{equ:riverbed-2}:
  \[
      \begin{aligned}
          B_{i,j}^{n,m+1} = \frac{1}{2}(B_{i,j}^{n,m}+\tilde{B}_{i,j}^{n,m+1})&-
          \frac{\Delta \tau^{n,m}}{2 \Delta x}
          \lambda_{i,j}^{x,n,m+1}
          (\tilde{B}_{i,j}^{n,m+1,R}-\tilde{B}_{i,j}^{n,m+1,L}) \\
          &-\frac{\Delta\tau^{n,m}}{2\Delta y}
          \lambda_{i,j}^{y,n,m+1}(\tilde{B}_{i,j}^{n,m+1,U}-
          \tilde{B}_{i,j}^{n,m+1,D})+
          \frac{\Delta\tau^{n,m}}{2} S_{i,j}^{n,m+1}.
      \end{aligned}
  \]
  Update current time $t\rightarrow t+\Delta\tau^{n,m}/\varepsilon$
  and set $m\rightarrow m+1$.
  \end{itemize}
\item[Step 6] If $m\ge K$, set $B^{n+1}=B^{n,m}$, $n\rightarrow n+1$,
  go to {\bf Step 2}, otherwise go to {\bf Step 3}.
\end{description}

If the $\cO(\varepsilon)$ correction and {\bf Step 5} are omitted,
then the resulting scheme becomes a first order discretization, whose
overall error becomes $\cO(\varepsilon + \Delta x)$ due to the CFL
condition.  We call the scheme in such a simplified version the {\it
first order scheme}, and the scheme contains all the steps above will
be referred as the {\it second order scheme} later on.

%% nondimensionalize %% 
\subsection{Nondimensionalization}
We often nondimensionalize the parameters in practice\cite{hudson2001numerical}.
Suppose the order of length, height, velocity, time and gravity
constant are $L, H, U, T, G$. We set
\[
\begin{aligned}
	&x=Lx^*, \quad y=Ly^*,\quad h=Hh^*, \quad B=HB^*,\\
	&u=Uu^*, \quad v=Uv^*,\quad t=Tt^*,\quad g=Gg^*,
\end{aligned}
\]
and the system can be reformulated in terms of $x^*,y^*,B^*,h^*,t^*$.
We may set $T=L/U,G=U^2/H$ thus the parameter $A_g$ (see
\eqref{equ:sed-Grass} or \eqref{equ:sed-MPM}) is set to be
$A_g\tilde{Q}_B/H$, where $\tilde{Q}_B=
\tilde{q}_b(U)$.  Based on the error estimate in section
\ref{subsec:advequ}, the error after nondimensionalization becomes 
\begin{equation}\label{error}
  \mathrm{Error}\sim 
  \begin{cases}
  \cO(\dfrac{\tilde{Q}_B}{H}\varepsilon+
  \dfrac{\Delta x}{L}), & \text{first order scheme,}\\
  \cO((\dfrac{\tilde{Q}_B}{H}\varepsilon)^2+
  (\dfrac{\Delta x}{L})^2+\dfrac{\tilde{Q}_B}{LH}\varepsilon\Delta x),
  &\text{second order scheme}.\\
  \end{cases}
\end{equation}

\section{Numerical Results} \label{sec:num}

In this section, we present several numerical results to validate the
effectiveness of the second order time homogenized model for sediment
transport. For all sediment transport problems considered, the initial
setup of the flow are obtained by solving the steady state on the
initial riverbed. All the computations are carried out on a laptop
computer with core speed of 2.3 GHz and the algorithm is implemented
using C++ programming language.

\subsection{One dimensional case} 
We consider the examples studied in \cite{hudson2001numerical,
  benkhaldoun2009solution}. The channel is of length with $1000 \rmm$ and
the initial riverbed is given as
\begin{equation} \label{chap7-dune-initB}
  B(x,0)=
  \left\{
    \begin{array}{ll}
     \sin^2\left(\dfrac{(x-300)\pi}{200}\right), & 300\leq x\leq 500, \\
      0, & \text{else where}. \\
    \end{array}
  \right.
\end{equation}
The initial water level is set to be $10\rmm$, and $Q$ is a constant
discharge taken case by case. Therefore, the nondimensionalized
parameters are
\[
L = 1000, \quad H = 10, \quad U = Q/10, \quad \tilde{Q}_B =
(Q/10)^{m-1}.
\]
The porosity constant $\gamma=0.4$, and the constant $A_g$
is set to $0.001$ representing the slow interaction of the riverbed
with water flow. Thus, the time scaling parameter in this case turns
to be $\varepsilon = 0.001/0.6$.  The CFL number is set to 0.65 when
updating the riverbed.  

In the Step 2 in the algorithm described in Section \ref{subsec:alg}, we 
need to get the steady state, which can be acquired by the 
standard flux-limited Roe scheme (see \cite{hudson2001numerical,
hudson2005numerical,deng2013robust}) for a long time so that 
$\lVert h^{n+1}-h^{n}\rVert_1+\lVert h^{n+1}u^{n+1}-
h^{n}u^n\rVert_1<10^{-6}$ or iteration number is bigger than 20000.
For the initial condition of the flow, we use this iteration until the
steady state is reached.  For the boundary condition used in the
steady solver, we fix the upstream discharge with $10m^2/s$ and use
the transmissive boundary condition for downstream.

Also in the Step 2 and Step 4 in the algorithm, correction terms are
need to compute.  We use the zero boundary condition and use the
BiCGSTAB solver with SSOR preconditioner in deal-II\footnote{see the
webpage at \tt{http://www.dealii.com}.} to solve the linear system. 
The tolerance of the BiCGSTAB solver is $10^{-6}$ and the relaxation
parameter of SSOR preconditioner is $0.955$. 

\subsubsection{Basic results}\label{basic-results} 
First, we present the results obtained when $Q=10 \rmm^2/\rs$ with
ending time $T=238079\rs$. The Grass model with $m=3$ (see
\eqref{equ:sed-Grass}) for the sediment transport flux is
considered at first, then numerical results of other models are given.
Later on, the convergence order of the first order and the second
order multi-scale algorithms will be computed. To make a comparison,
we have included a reference solution computed by the Roe's scheme
with the second order flux-limited method \cite{hudson2001numerical,
hudson2005numerical} using a fine mesh with 4096 grid points.

Figure \ref{fig:water} displays the sampling results (i.e. 
the depth and velocity of water) at initial time and end time.

Figure \ref{fig:dune-schemes} displays the riverbed when applying the
first order scheme and the second order scheme on mesh with
$N=256$. We set $K = 2$ to accelerate the computing. We also plot the
solution of Roe's scheme for comparison. It is clear that the first
order scheme produces the diffusive riverbed. However, this numerical
diffusive has been reduced remarkably by the second order scheme. 
%%When
%%comparing the numerical results of multi-scale method with direct
%%method, we find out that the multi-scale method preforms better. 
%%Moreover, the overall
%%computing time for multi-scale method when applying the first order
%%scheme and the second order scheme are $0.10\rs$ and $0.20\rs$,
%%which is, far less than that for direct method. The overall computing
%%time for direct method with flux-limited Roe's scheme is about
%%$39.15\rs$.
\begin{figure}[!htb]
\centering
\subfigure[Depth of water at initial time]{
    \includegraphics[width=0.45\textwidth]{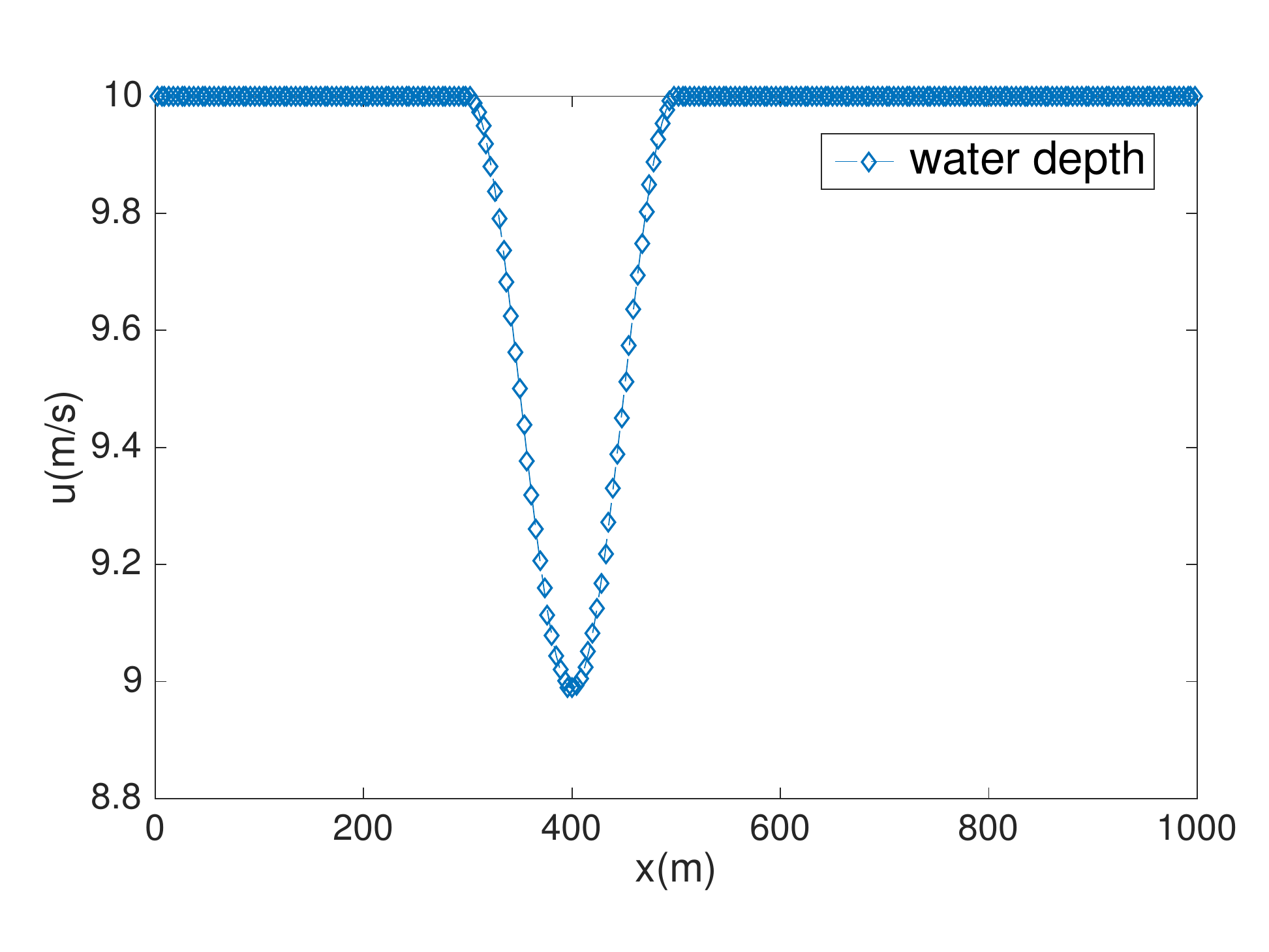}
}
\subfigure[Velocity of water at initial time]{
    \includegraphics[width=0.45\textwidth]{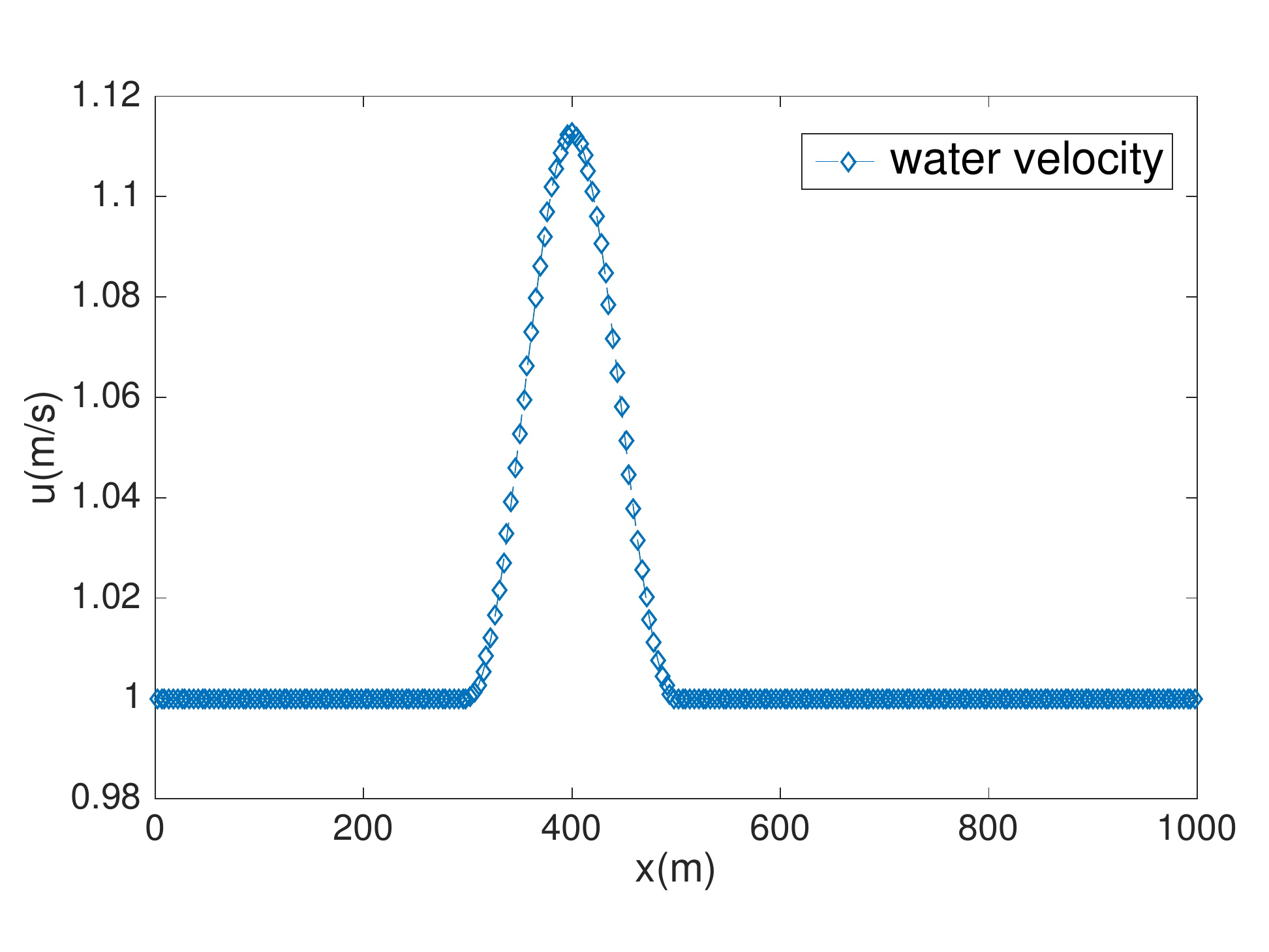}
}
\subfigure[Depth of water at end time]{
    \includegraphics[width=0.45\textwidth]{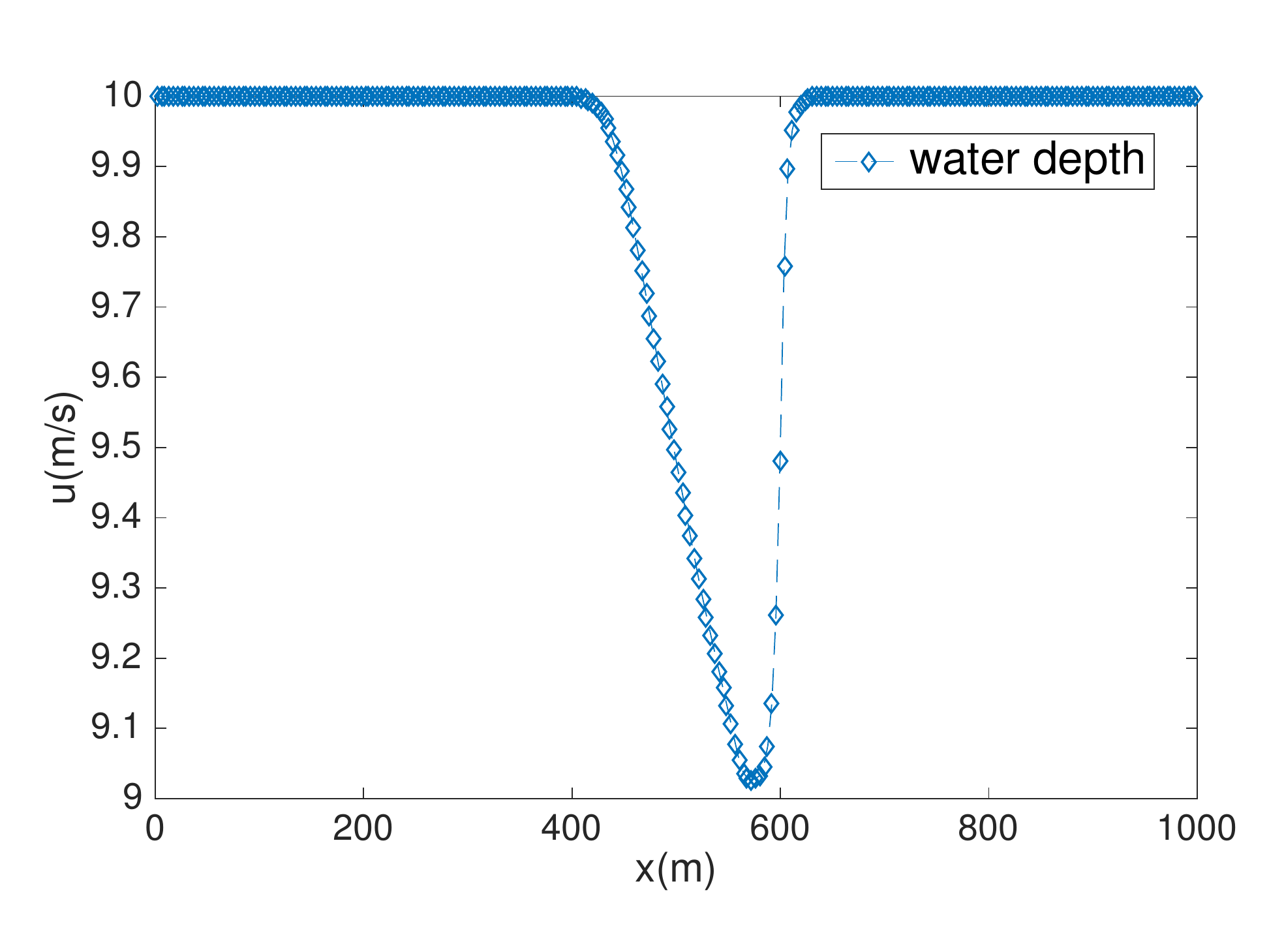}
}
\subfigure[Velocity of water at end time]{
    \includegraphics[width=0.45\textwidth]{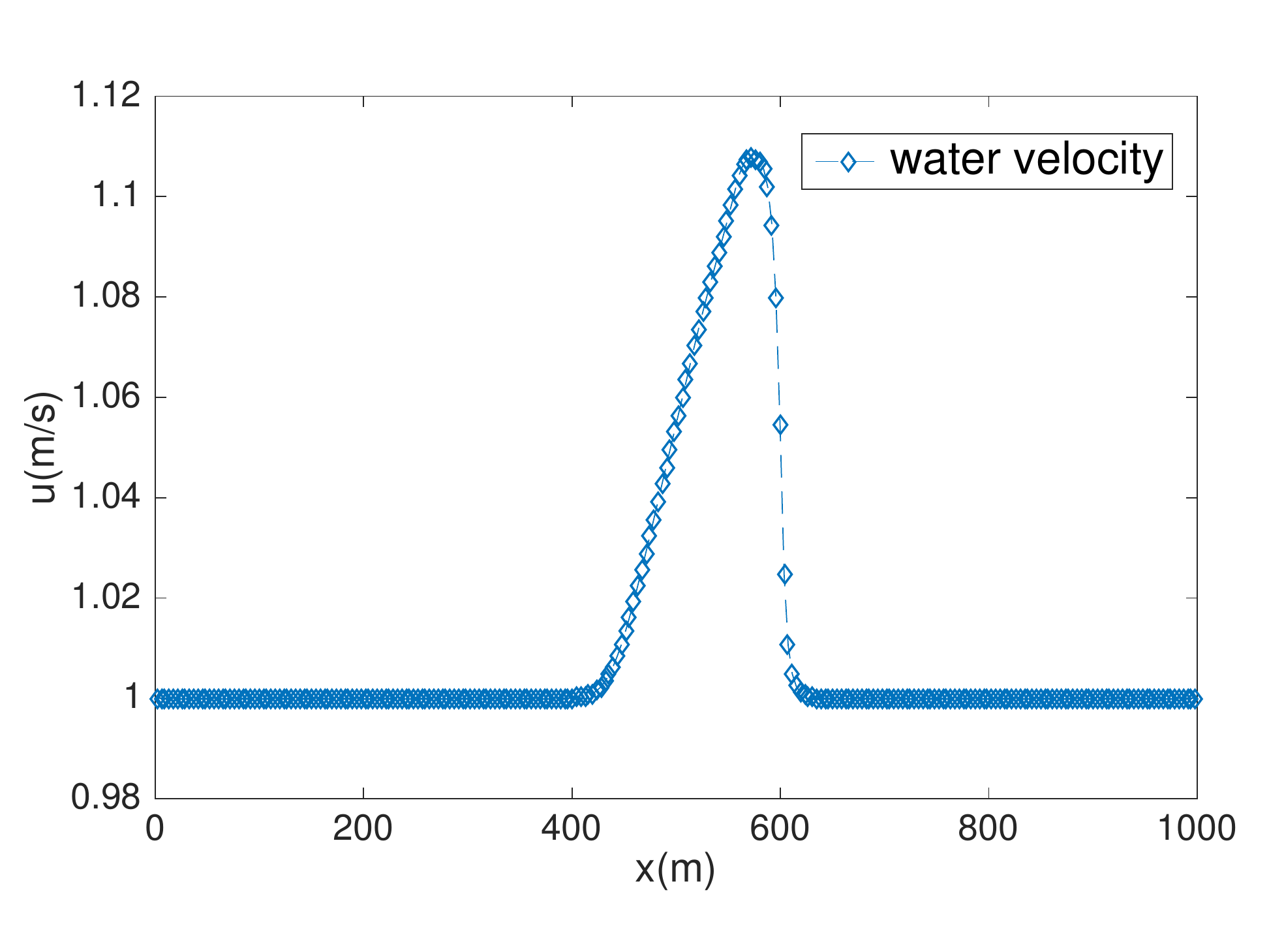}
}
\caption{Sampling results.}
\label{fig:water}
\end{figure}

\begin{figure}[!htb]
\centering
\includegraphics[width=1.0\textwidth]{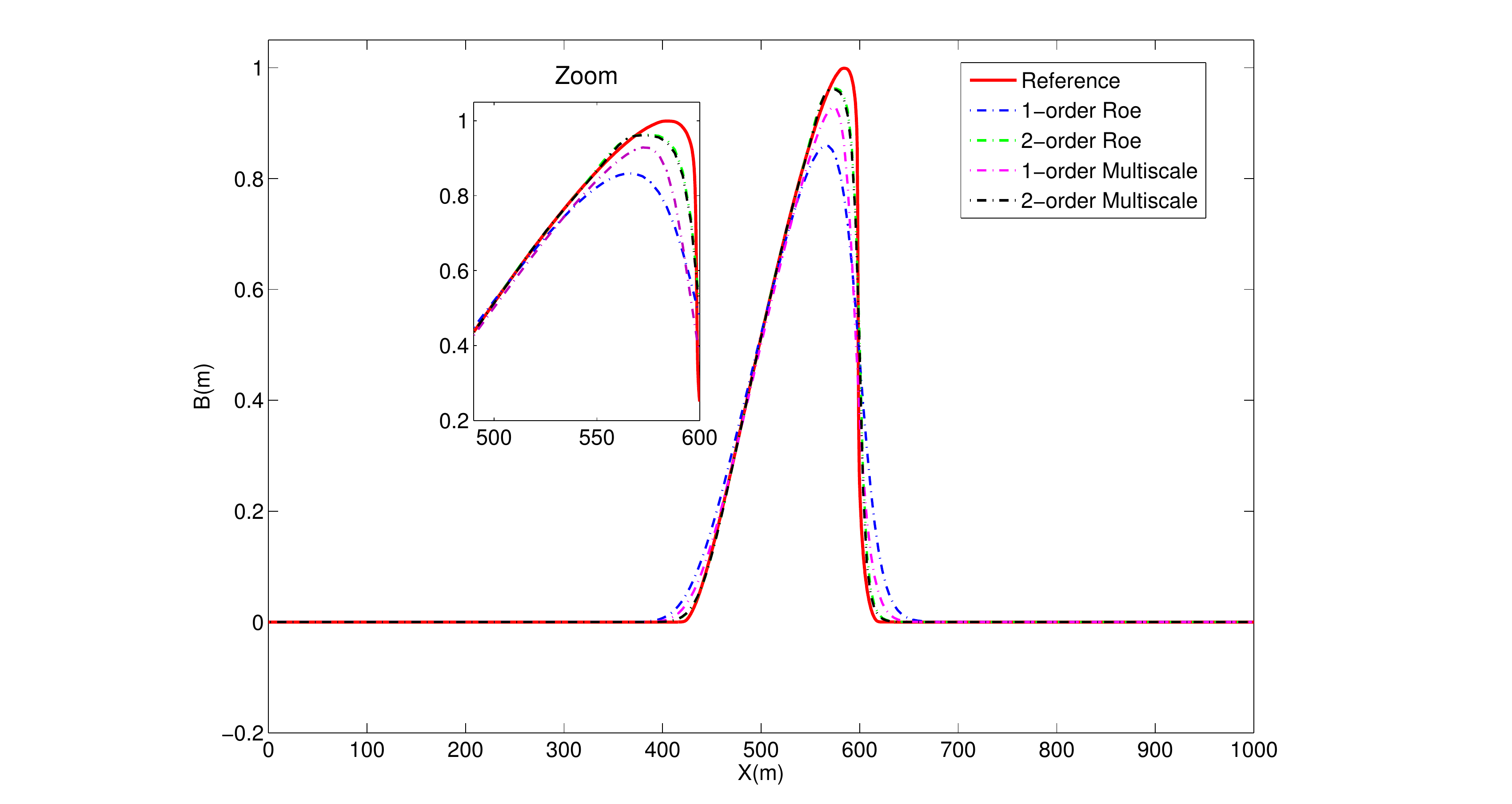}
\caption{Comparison of different methods when $N=256$, $T=238079\rs$. }
\label{fig:dune-schemes}
\end{figure}

\subsubsection{Meyer-Peter-M\"{u}ler Model} 
Figure \ref{fig:other-models} shows the
comparison between Grass model and Meyer-Peter-M\"{u}ler model when
$u_{cr}=0.5$, $1.0$, and $1.04$.  Here, all computations are carried
out using the second order multi-scale algorithms, with parameters the
same as above.
\begin{figure}[!htb]
\centering
\includegraphics[width=1.0\textwidth]{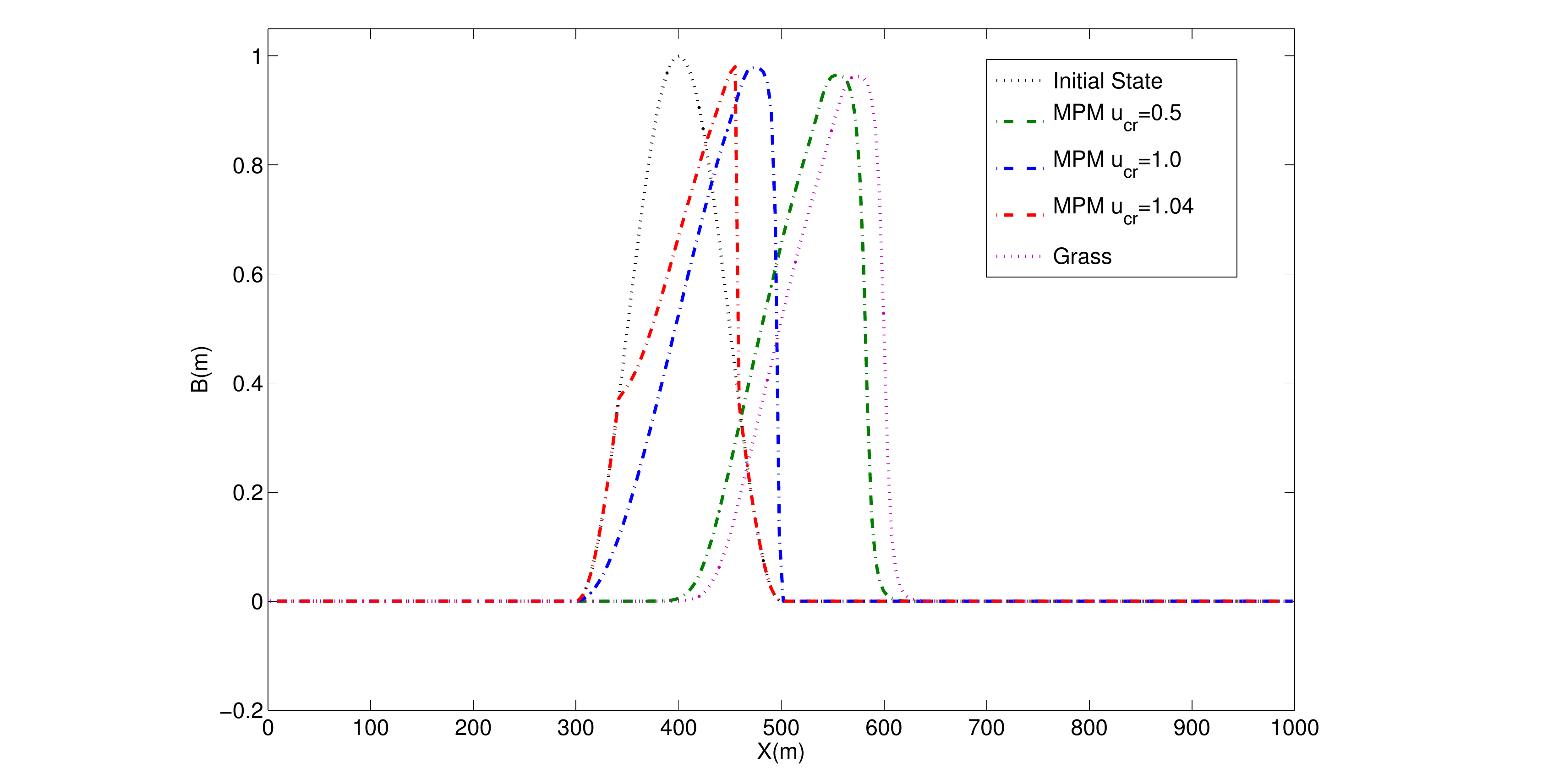}
\caption{Comparison between Grass model and Meyer-Peter-M\"{u}ler}
\label{fig:other-models}
\end{figure}

\subsubsection{Convergence results} 
Let us examine the convergence order of the multi-scale schemes. The
test will be based on the Grass model. The Roe's scheme
\cite{hudson2001numerical, hudson2005numerical} on an extremely fine
mesh with $16384$ grid points to is applied to produce the reference
solution. Due to the limitation of our computing capacity, the
computing time is comparatively short, says $T=90000\rs$. Actually,
the time $T = 150/\varepsilon$ is enough for convergence order
study. Here, we set $K = 1$ and compute using both the first order and
the second order algorithms. Besides, to study the effect of the
$\cO(\varepsilon)$ correction, we use the second order solver while
discarding the $\cO(\varepsilon)$ correction in the third test. 
The Table \ref{table:convergece-order} shows the convergence
order for each algorithm, $\hat{B}$ is the approximate solution and
$B^*$ is the reference solution. One can see our algorithm has
satisfactory convergence order, and the $\cO(\varepsilon)$ correction
is essential to improve the accuracy.

\begin{table}[!htb]
\centering
\begin{tabular}{c|c|c|c|c|c|c}
  \hline\hline
  &\multicolumn{2}{|c|}{first order} & \multicolumn{2}{|c|}{second order}
  &\multicolumn{2}{|c}{without $\varepsilon$ correction}\\
  \hline
  $N$ &$\lVert\hat{B}-B^*\rVert_1$ & order & $\lVert\hat{B}-B^*\rVert_1$ & order & $\lVert\hat{B}-B^*\rVert_1$ & order\\
  \hline
  128 & 7.05& &3.22 & &3.24 & \\  
  \hline
  256 & 3.68& 0.94&1.03 & 1.65&1.05 &1.63 \\  
  \hline
  512 & 1.88& 0.97&3.25e-1 & 1.66&3.57e-1& 1.55\\  
  \hline
  1024& 9.60e-1& 0.97&9.01e-2& 1.85&1.47e-1 &1.27 \\  
  \hline
  2048& 4.88e-1& 0.97&2.39e-2& 1.91&9.67e-2& 0.61\\  
  \hline\hline
\end{tabular}
\caption{Convergence order of different algorithms.}
\label{table:convergece-order}
\end{table}

\subsubsection{Computing time comparison} 
We will show the computing times with different $A_g$'s and mesh sizes
in this subsection. The ending time $T = 150/\varepsilon$, the
porosity constant is 0.4, and $K=2$ in the computations.  
For different cases that $A_g=0.01,0.005,0.001$ and $N=256,512$, Roe
scheme, the  first order multi-scale scheme and the second order
multi-scale scheme are tested. From the computing times shown in the
Table \ref{Time}, we can see that for different $A_g$'s, the computing
times of first order and second order scheme do not change a lot.
It's because the main computational cost attributes to solving the
steady states, which does not change a lot for different $A_g$'s.
These results demonstrate the efficiency of our multi-scale schemes,
especially when $A_g$ is small enough.
\begin{table}[!htb]
    \centering
    \begin{tabular}{c|c|c|c|c}
        \hline\hline
        $A_g$ & $N$ & Roe scheme & first order& second order\\
        \hline
        0.01 &  256 & 4.05 & 0.10 & 0.22\\
        \hline
        0.01 &  512 & 15.41 & 0.65 & 0.97\\
        \hline
        0.005 &  256 & 8.15 & 0.10 & 0.20\\
        \hline
        0.005 &  512 & 30.29 & 0.66 & 0.98\\
        \hline 
        0.001 & 256 & 39.15 & 0.10 & 0.20\\
        \hline
        0.001 & 512 & 152.01 & 0.65 &  0.99\\
        \hline\hline
    \end{tabular}
    \caption{Computing times (seconds) for different cases.}
    \label{Time}
\end{table}

\subsection{Two dimensional example}
This example has been studied in \cite{hudson2001numerical,
  hudson2005numerical, delis2008relaxation}. We adopt the 2D case
where the sediment transport takes place in a $1000m \times 1000m$
channel, with the initial dune profile as
\begin{equation} \label{equ:dune2d-B}
  B(x,y,0) = 
  \left\{ 
    \begin{array}{ll} 
      \sin^2(\frac{(x-300)\pi}{200}) \sin^2(\frac{(y-400)\pi}{200}), 
      & \text{if}~ 300\leq x\leq 500, 400\leq y \leq 600, \\
      0,  & \text{else}.
    \end{array} 
  \right.  
\end{equation}
The initial water surface level is $10$m everywhere with the uniformly
horizontal discharge $Q=10m^2/s$, namely 
\[
h(x,y,0) = 10 - B(x,y,0), \quad u(x,y,0) = \frac{Q}{h(x,y,0)}, \quad
v(x,y,0) = 0.
\]
In this test, the Grass model with $m=3$ is used. The porosity is 
$0.4$ and time scaling parameter $\varepsilon = 0.001/(1-0.4)$ to
coincide with the model in \cite{hudson2001numerical}.

When solving the steady state, we fix the discharge of $x$-direction
to be $Q=10m^2/s$ at the upstream boundary, and the transmissive
boundary condition is applied to the downstream boundary.  The
reflective boundary condition is adopted on the both sides of the
channel. We also use the flux-limited Roe scheme
\cite{hudson2001numerical,hudson2005numerical} to solve the steady
state. As with 1D case, we solve the shallow water 
equations until the residual is less than $10^{-6}$ or the 
iteration number is bigger than 20000, and then the result is 
approximated to be the steady state.
 We compute this
channel test problem using the second order multi-scale method until
$T=3.6\times10^5$s on a $128\times128$ mesh. The CFL number is set to 
be $0.5$ and $K$ is set
to be $2$. When solving the correction terms,  we use the reflective 
boundary condition on the $y=0,1000m$, and use the zero boundaries
condition on $x=0,1000m$. As with 1D case, the
BiCGSTAB solver with SSOR preconditioner is used to solve the correction 
terms.  The tolerance of the BiCGSTAB solver is $10^{-6}$ and the
relaxation parameter of SSOR preconditioner is $0.955$. 
\begin{figure}[!htb]
\centering
\subfigure[Riverbed at initial time]{
    \includegraphics[width=0.45\textwidth]{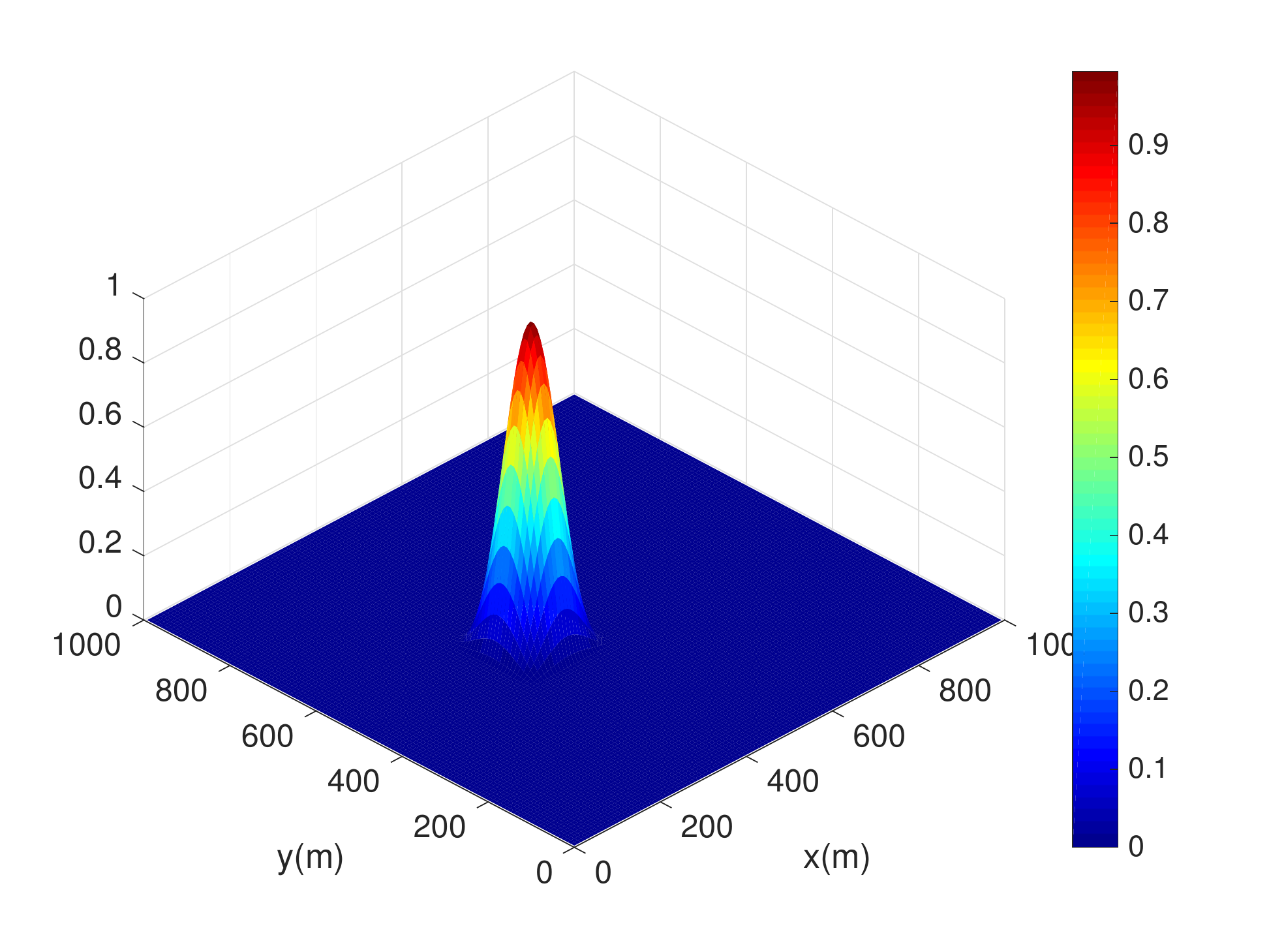}
}
\subfigure[Riverbed at end time]{
    \includegraphics[width=0.45\textwidth]{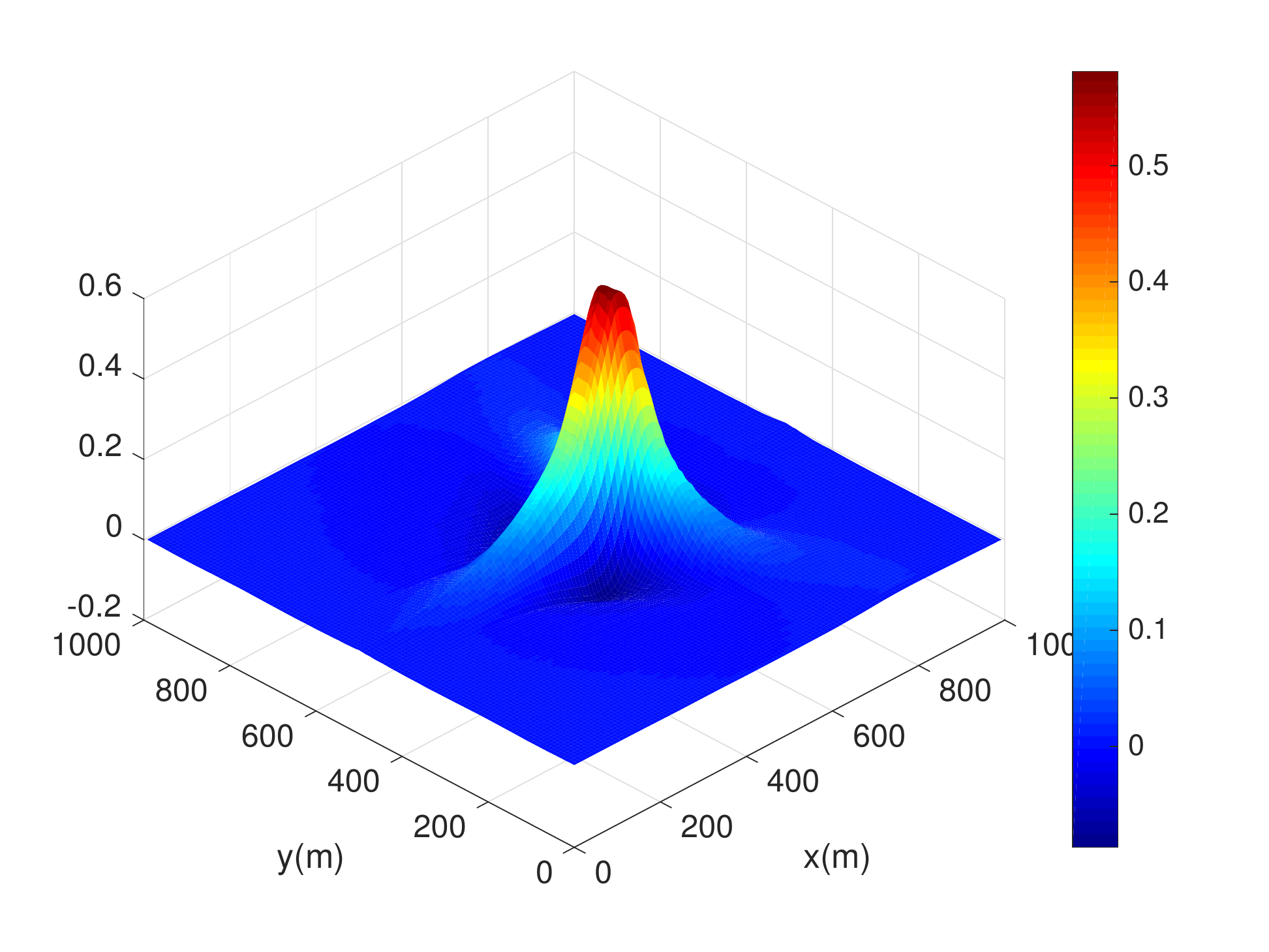}
}
\subfigure[Top view of riverbed at initial time]{
    \includegraphics[width=0.45\textwidth]{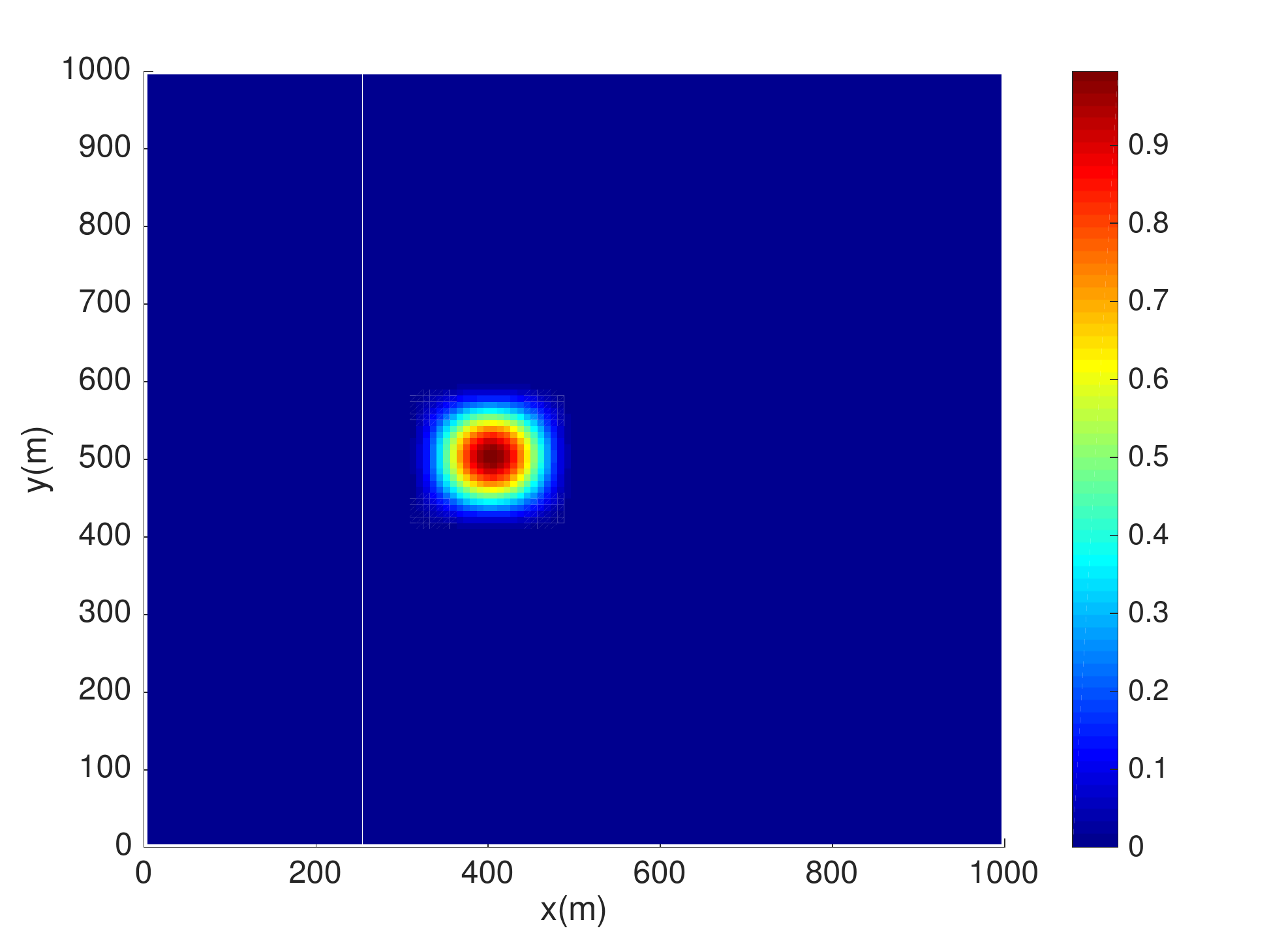}
}
\subfigure[Top view of riverbed at end time]{
    \includegraphics[width=0.45\textwidth]{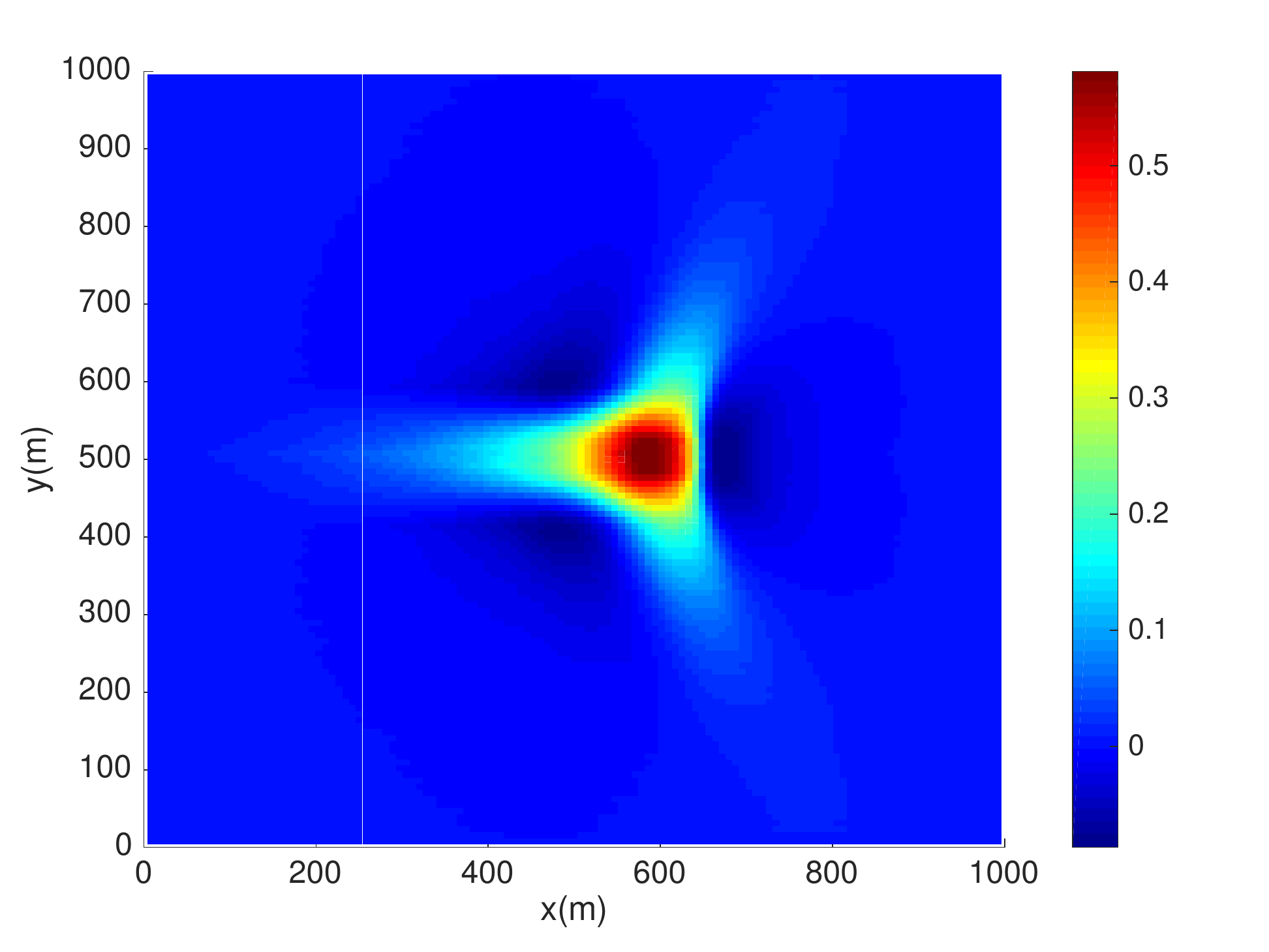}
}
\caption{Numerical results of riverbed.}
\label{fig:2dbed}
\end{figure}

\begin{figure}[!htb]
\centering
\subfigure[Top view of $u$ at initial time]{
    \includegraphics[width=0.45\textwidth]{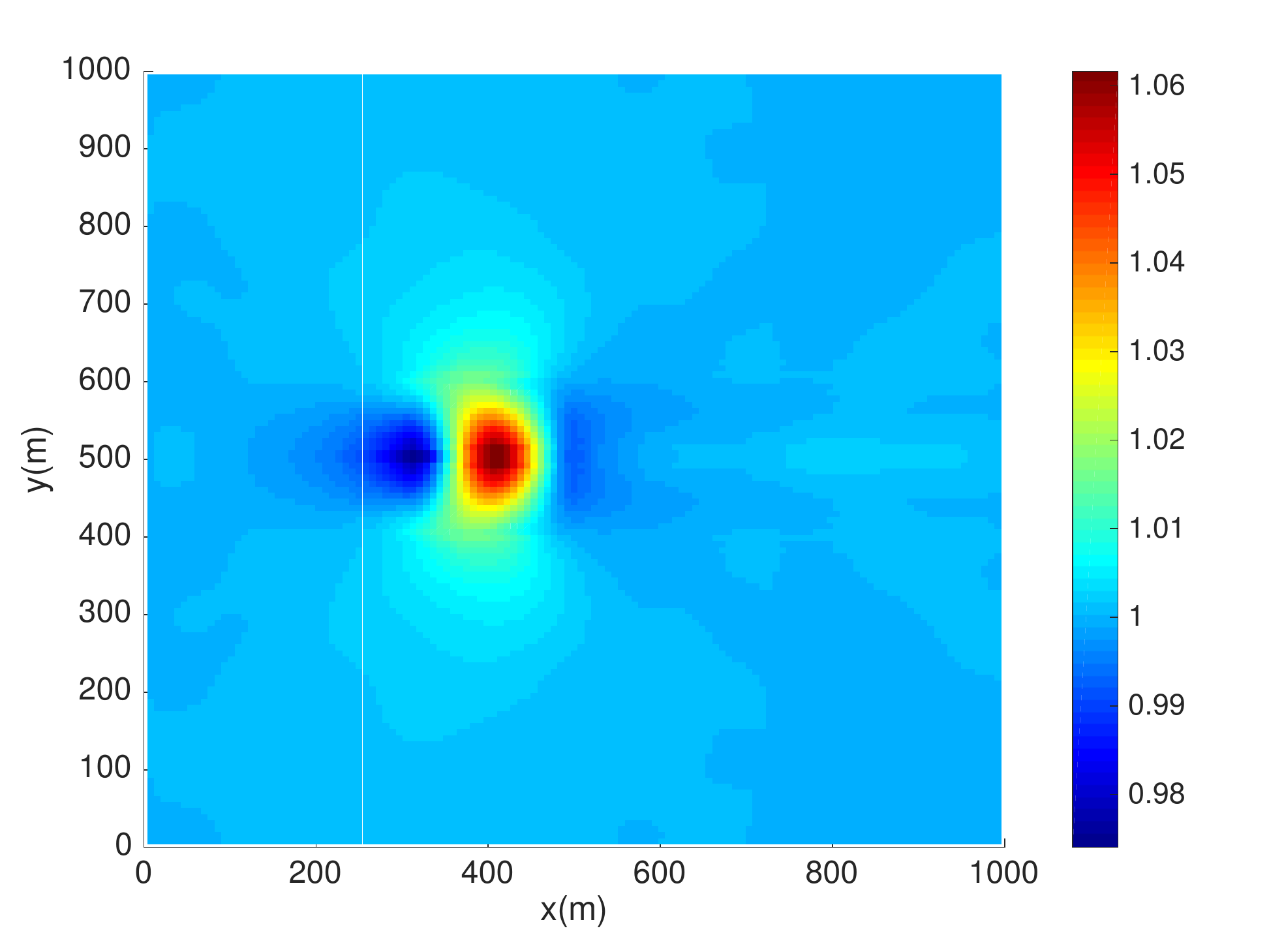}
}
\subfigure[Top view of $u$ at end time]{
    \includegraphics[width=0.45\textwidth]{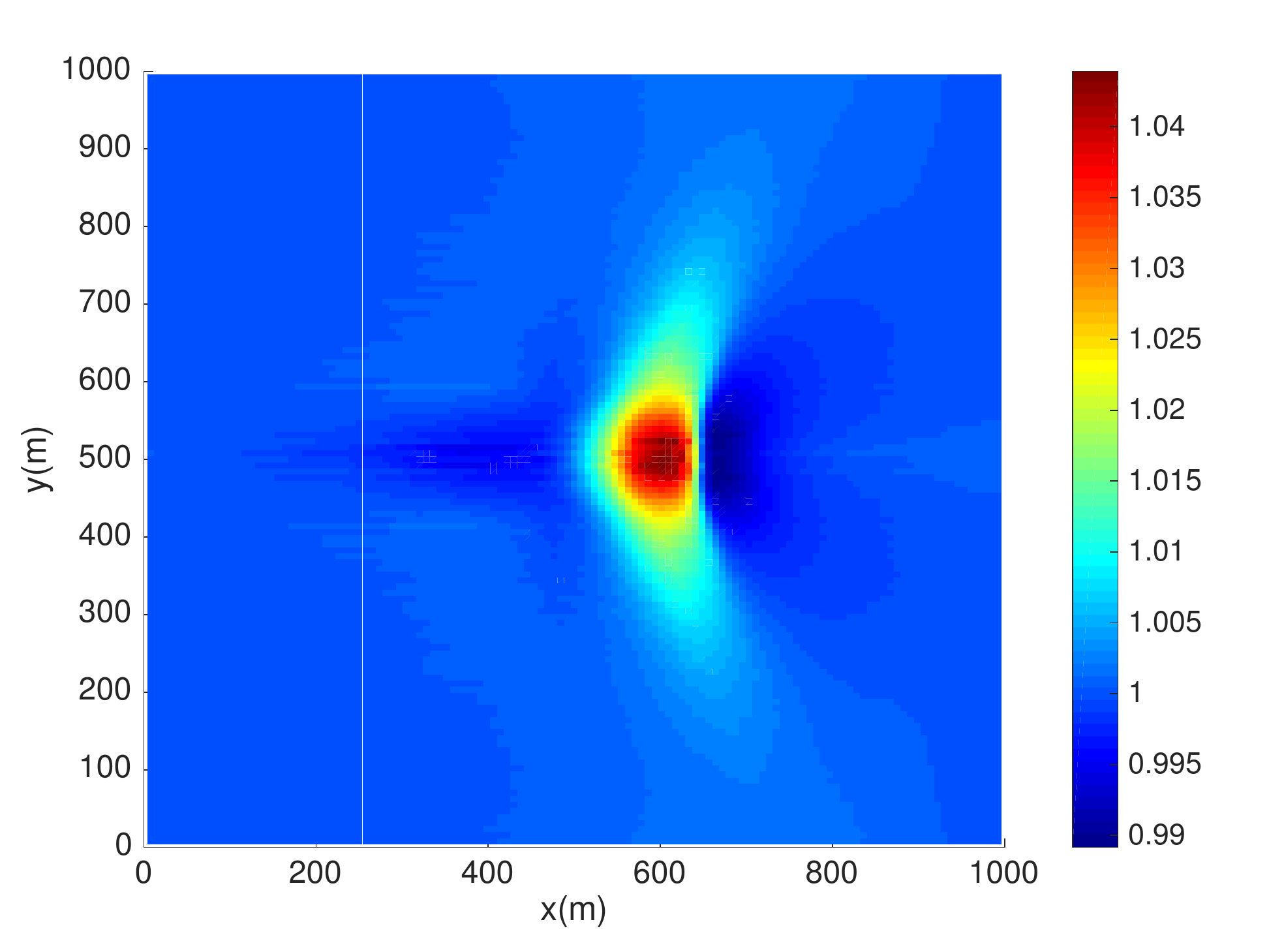}
}

\subfigure[Top view of $v$ at initial time]{
\includegraphics[width=0.45\textwidth]{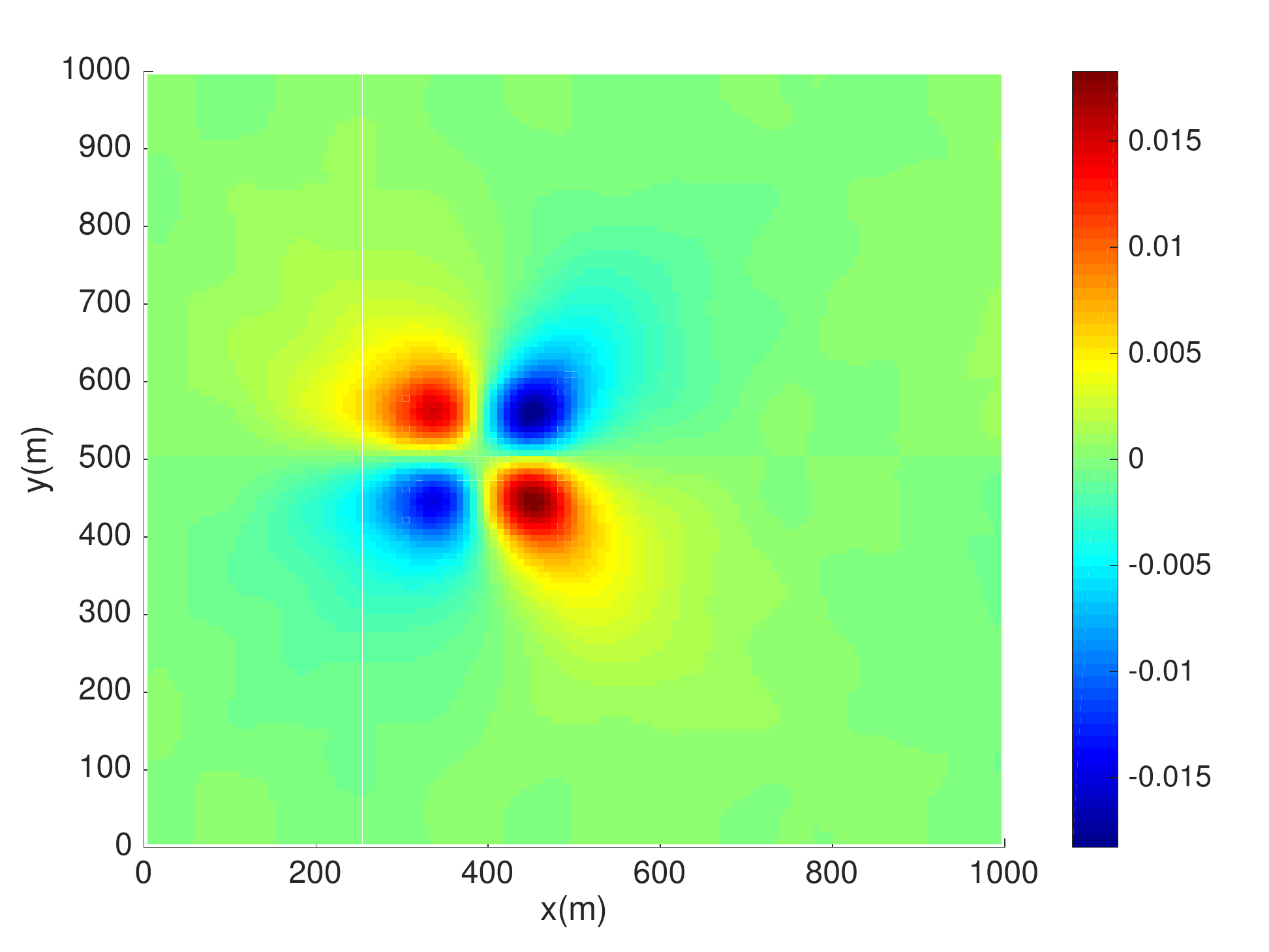}
}
\subfigure[Top view of $v$ at end time]{
    \includegraphics[width=0.45\textwidth]{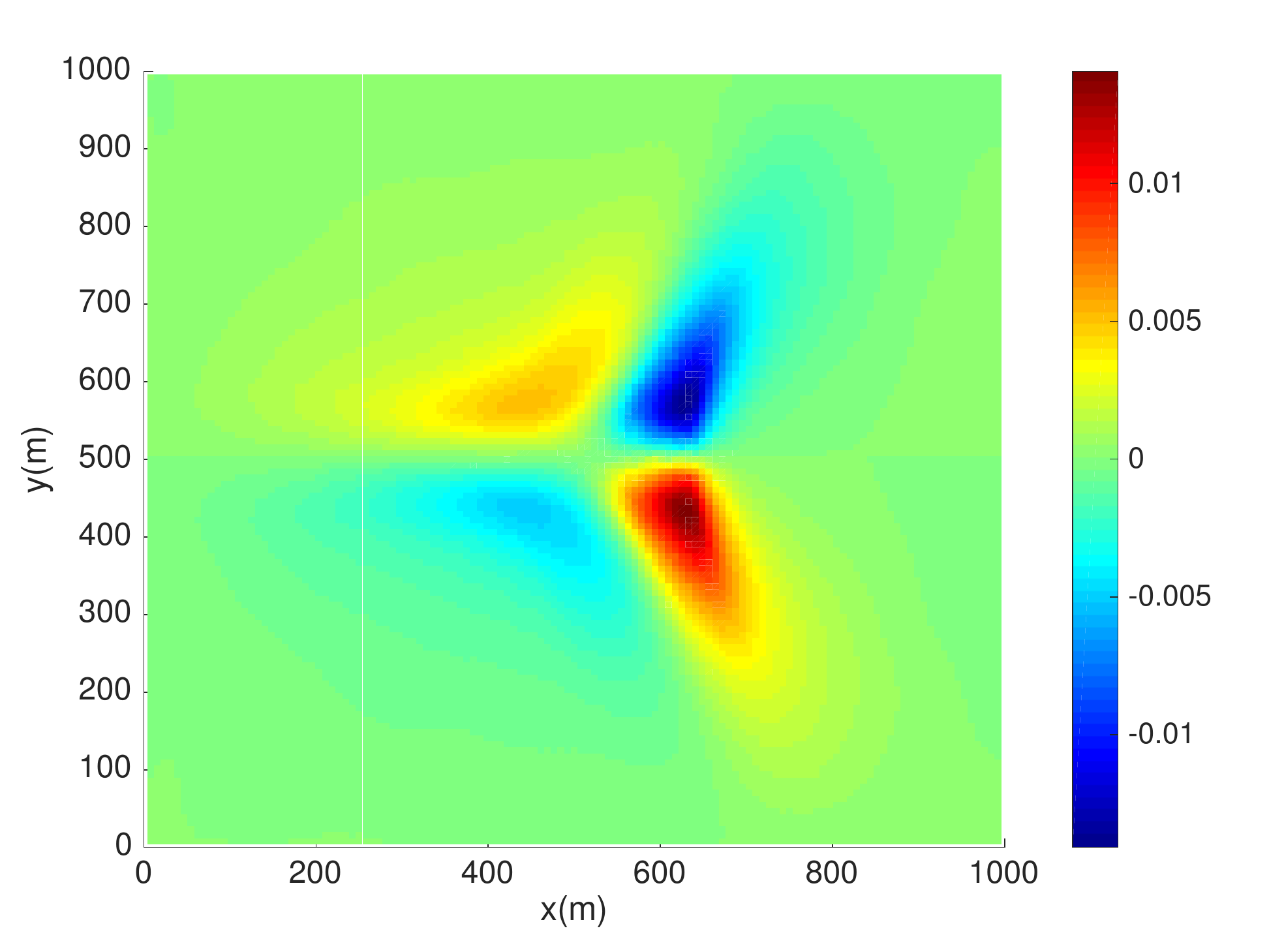}
}
\caption{Sampling results of water velocity.}
\label{fig:2dvelocity}
\end{figure}

\begin{figure}[!htb]
\centering
\includegraphics[width=0.7\textwidth]{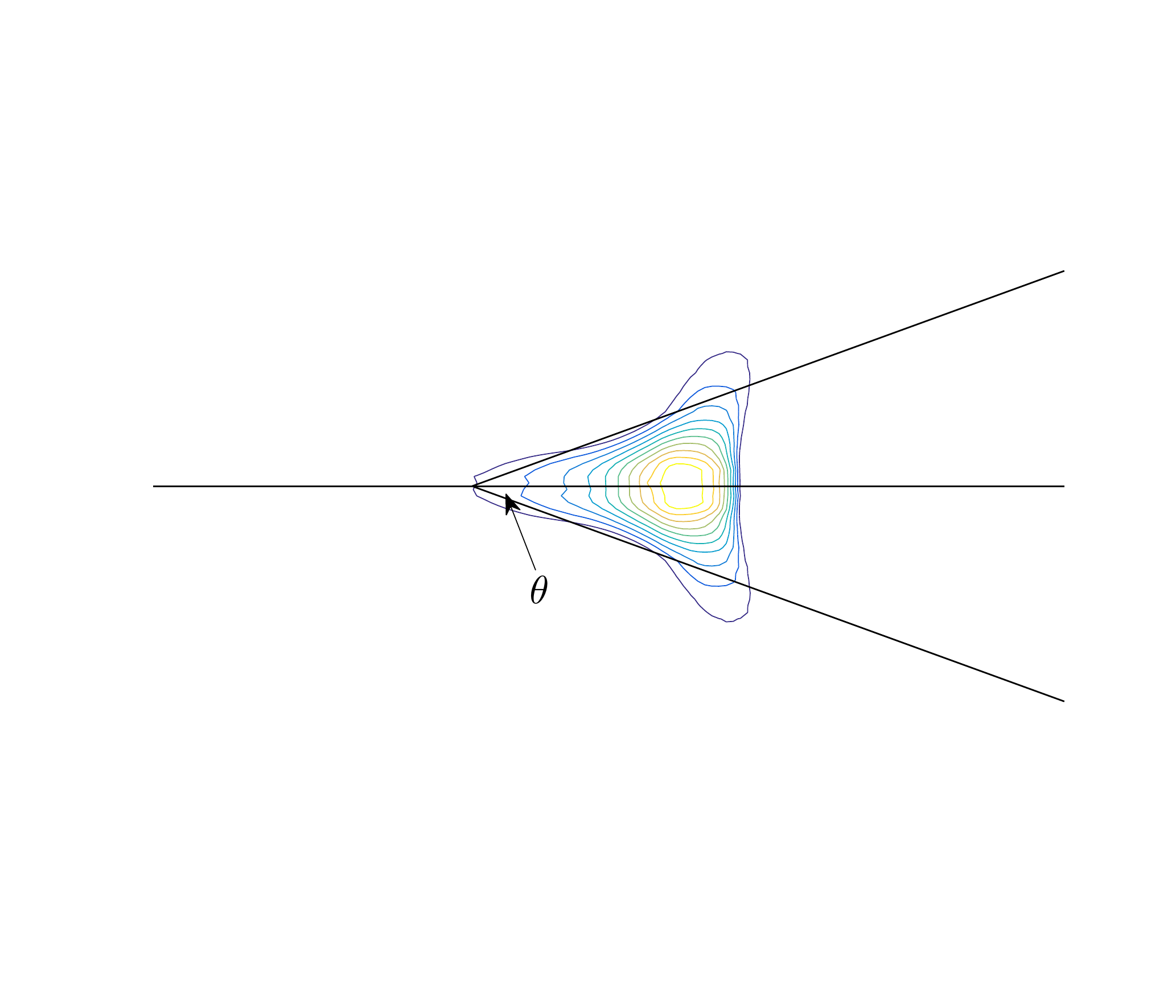}
\caption{Spread angle.}
\label{fig:spreadangle}
\end{figure}

Figure \ref{fig:2dbed} shows the riverbed at initial time and end
time. The steady state of velocities are shown in Figure
\ref{fig:2dvelocity}.  As shown in \cite{hudson2001numerical,
hudson2005numerical, delis2008relaxation}, the initial dune will
gradually deform to a star-shaped pattern. From the Figure
\ref{fig:2dbed} one sees clearly that our method captures correctly
such behaviors. 

More precisely, the spread angle of the riverbed is important to show whether 
our model and scheme work \cite{de2009energetically,de19872dh}.
Assume that the interaction between the sediment layer and flow is
low, the following approximation of the spread angle is proposed by De
Vriend \cite{de19872dh}
\[
    \tan \theta=\frac{3\sqrt{3}(m-1)}{9m-1}.
\]
For the case in which $m=3$, the angle is approximately
$21.78^{\circ}$. Figure \ref{fig:spreadangle} shows the contour of
the riverbed at the ending time and also the angle
$\theta=21.78^\circ$. From the figure, we observe that the spread
angle of our scheme is very approximately to the one derived by De
Vriend.

At last, the computing time of our multi-scale method is $492$s, which
is much less than that of the flux-limited Roe's method
\cite{hudson2001numerical,hudson2005numerical} ($16806$s).

\section{Conclusions} \label{sec:conclusion}

In this paper, a second order time homogenized model and the
corresponding numerical methods for the sediment transport are
proposed. Through the numerical experiments, the multi-scale method
shows significant effectiveness, especially for the long time
simulation of sediment transport while provides a considerably good
approximation to the coupled system.

\section*{Acknowledgment}
This is a succeeding research of a project supported by
ExxonMobile. The authors appreciate the financial supports provided by
the \emph{National Natural Science Foundation of China (NSFC)} (Grant
91330205 and 11325102).

\appendix
\section{Proofs of Lemma \ref{prop:1-conv}}
\begin{proof}[Proof of Lemma \ref{prop:1-conv}]\label{proof:1-conv}
%Let 
%  \begin{equation}\label{mat:C-Ceps0}
%  \bC = \begin{pmatrix} \bA & \bg \\ 0 & 0 \\ \end{pmatrix}, \quad 
%  \bC^\varepsilon = \begin{pmatrix} \bA & \bg \\ \varepsilon \bc^T & 0 \\ 
%  \end{pmatrix},
%  \end{equation}
%  and $\bV=[\bU^T, B]^T$, then the original system can be represented as:
%  \[
%    \bV_t+\bC^\varepsilon \bV_x=0.
%  \]
% 
%  It follows from \eqref{equ:linear-A} that $\bC=\bK^{-1}\bD\bK$, where
%  \[
%    \bK=\begin{pmatrix} \bX & \bLambda^{-1}\bX \bg \\ 0 & 1 
%    \end{pmatrix}, \quad 
%    \bD=\begin{pmatrix} \bLambda & 0\\0 & 0\end{pmatrix}
%  \]
%  By the perturbation theory of discrete eigenvalues and eigenvectors in
%  \cite{wilkins1965eigenvalue}, we know that $\bC^\varepsilon$ commits
%  the decomposition
%  \begin{equation}\label{decom-0}
%    \bC^\varepsilon = (\bK^\varepsilon)^{-1} 
%     \bD^\varepsilon \bK^{\varepsilon},
%  \end{equation}
%  where 
%  \[
%    \bK^\varepsilon = 
%    \begin{pmatrix} \bX+\varepsilon\hat{\bX} &
%    \bLambda^{-1}\bX \bg+\varepsilon\hat{\balpha} \\
%    \varepsilon\hat{\bbeta}^T & 1+\varepsilon\hat{\theta} \\
%    \end{pmatrix},\quad
%  \bD^\varepsilon=\begin{pmatrix} 
%    \bLambda + \varepsilon \hat{\bLambda} & 0 \\ 0 & \varepsilon\mu\\ 
%    \end{pmatrix}.
%  \]
%
  Let $\lambda_k^\varepsilon=\lambda_k+ \varepsilon\hat{\lambda}_k$
  be the $k$-th eigenvalue of $\bC^\varepsilon$. Then, we have  
  \[
    \begin{aligned}
      &\big[(\bX+\varepsilon\hat{\bX})\bU+(\bLambda^{-1}\bX
          \bg+\varepsilon \hat{\balpha})B\big]_k(x,t) \\ =~ &
      \big[(\bX+\varepsilon\hat{\bX})\bU + (\bLambda^{-1}\bX
          \bg+\varepsilon\hat{\balpha})B\big]_k(x-
            \lambda_k^\varepsilon t,0).
    \end{aligned}
  \]
  Taking the spatial derivative to obtain
  \begin{equation}\label{equ:Ux-Bx}
    \begin{aligned}
      & \big(\bX\bU_x+\bLambda^{-1}\bX \bg B_x\big)_k(x,t)
      = \big(\bX\bU_x+\bLambda^{-1}\bX \bg B_x\big)_k(x-
      \lambda_k^\varepsilon t,0) \\ & + \varepsilon\big(\hat{\bX}\bU_x +
      \hat{\balpha}B_x\big)_k(x-\lambda_k^\varepsilon t,0)
      - \varepsilon\big(\hat{\bX}\bU_x+\hat{\balpha}B_x\big)_k(x,t).\\
    \end{aligned}
  \end{equation}
  Since 
  \[
    \bX\bU_x+\bLambda^{-1}\bX \bg B_x = \bLambda^{-1}\bX(\bA\bU_x+\bg B_x),
  \]
  and the initial state of fast variables, we get
  \[
    \big(\bX\bU_x+\bLambda^{-1}\bX \bg B_x\big)_k(x-
        \lambda_k^\varepsilon t,0) = \cO(\varepsilon).
  \]
  It follows from the boundedness of $\bU_x,B_x$ (one can check directly by
  the boundedness of $(\bK^\varepsilon)^{-1}$) that
  \[
  \lVert\bX \bU_x+\bLambda^{-1}\bX\bg B_x\rVert_\infty=\cO(\varepsilon),
  \]
  which implies 
  \begin{equation} \label{equ:derivative-error}
  \lVert \bA \bU_x + \bg B_x\rVert_\infty=\cO(\varepsilon)
  \quad \text{or} \quad 
  \lVert \bU_\tau \rVert_\infty = \cO(1).
  \end{equation}

  Let $E=B-B^{(0)}$. By subtracting \eqref{equ:linear-0} from
  \eqref{equ:linear-0tau}, we have
  \[
  E_\tau +\lambda^{(0)}_BE_x = \varepsilon \bc^T\bA^{-1}\bU_\tau \qquad
  E(x,0)=0.
  \]
  Thus, 
  \[
  E(x,\tau) = \varepsilon\int_{0}^{\tau} \bc^T\bA^{-1}\bU_\tau
  (x-\lambda^{(0)}_B(\tau-s),s)\mathrm{d} s.
  \]
  Finally, we have
  \[
    \lVert E\rVert_\infty \le C_1 \varepsilon \int_0^{\tau}\lVert
    \bU_\tau \rVert_\infty\mathrm{d}s \le C_2 \varepsilon\tau.
  \]
  If $B(x,0) \in W^{2,\infty}(\R)$, then we can take the spatial
  derivative again to \eqref{equ:Ux-Bx} to obtain 
  \begin{equation}\label{equ:Uxx-Bxx}
    \begin{aligned}
      & \big(\bX\bU_{xx}+\bLambda^{-1}\bX \bg B_{xx}\big)_k(x,t)
      = \big(\bX\bU_{xx}+\bLambda^{-1}\bX \bg B_{xx}\big)_k(x -
      \lambda_k^\varepsilon t,0) \\ & +
      \varepsilon\big(\hat{\bX}\bU_{xx} +
          \hat{\balpha}B_{xx}\big)_k(x-\lambda_k^\varepsilon t,0) -
      \varepsilon\big(\hat{\bX}\bU_{xx}+\hat{\balpha}B_{xx}\big)_k(x,t),\\
    \end{aligned}
  \end{equation}
  which yields 
  \[
  \|\bA \bU_{xx} + \bg B_{xx}\|_{\infty} = \cO(\varepsilon), \quad 
  \text{or} \quad \|\bU_{\tau x}\|_{\infty} = \cO(\varepsilon),
  \]
  by the similar argument. Then 
  \[
  (E_x)_{\tau} + \lambda_B^{(0)}(E_x)_x = \varepsilon \bc^T \bA^{-1}
  \bU_{\tau x} \qquad E_x(x,0) = 0,
  \]
  which means that  
  \[
  \|E_x\|_{\infty} \le C_1 \varepsilon \int_0^{\tau} \|\bU_{\tau
    x}\|_{\infty} \rd s \le C_2 \varepsilon \tau.
  \]
  This ends the proof.
\end{proof}

%% appendix 2: prove of Lemma 2.2 %% 
\section{Proof of Lemma \ref{lemma:2-conv-1}}
\begin{proof}[Proof of Lemma \ref{lemma:2-conv-1}]\label{proof:2-conv-1}
From \eqref{decom-0}, we have
\[
    \bK^\varepsilon \bC^\varepsilon=\bD^\varepsilon\bK^\varepsilon,
\]
which can be written as  
\[
    \begin{pmatrix} \bX+\varepsilon\hat{\bX} &
    \bLambda^{-1}\bX \bg+\varepsilon\hat{\balpha} \\
    \varepsilon\hat{\bbeta}^T & 1+\varepsilon\hat{\theta} 
    \end{pmatrix}
    \begin{pmatrix}
        \bA & \bg\\
        \varepsilon\bc^T & 0
    \end{pmatrix}
    =\begin{pmatrix}
        \bLambda + \varepsilon\hat{\bLambda} & 0\\
        0 & \varepsilon\mu
    \end{pmatrix}
    \begin{pmatrix} \bX+\varepsilon\hat{\bX} &
    \bLambda^{-1}\bX \bg+\varepsilon\hat{\balpha} \\
    \varepsilon\hat{\bbeta}^T & 1+\varepsilon\hat{\theta} 
    \end{pmatrix}.
\]
By expanding this equation, we have
\[
    \begin{aligned}
        (\bX+\varepsilon\hat{\bX})\bA+\varepsilon(\bLambda^{-1}\bX\bg+
        \varepsilon\hat{\balpha})\bc^T&=(\bLambda+
        \varepsilon\hat{\bLambda})(\bX+\varepsilon\hat{\bX}),\\
        (\bX+\varepsilon\hat{\bX})\bg&=(\bLambda+
        \varepsilon\hat{\bLambda})(\bLambda^{-1}\bX\bg+\varepsilon\hat{\balpha}).
    \end{aligned}
\]
By collecting the $\cO(\varepsilon)$ term we have 
\[
\left\{
\begin{aligned}
\hat{\bX}\bg &= \bLambda \hat{\balpha} + \hat{\bLambda}\bLambda^{-1}
\bX \bg + \cO(\varepsilon), \\
\hat{\bX}\bA + \bLambda^{-1}\bX\bg\bc^T &= \bLambda \hat{\bX} +
\hat{\bLambda} \bX+\cO(\varepsilon).
\end{aligned}
\right.
\]
Or 
\begin{equation} \label{equ:perb-0}
\left\{
\begin{aligned}
\hat{\balpha} &= \bLambda^{-1}\hat{\bX}\bg -
\hat{\bLambda}\bLambda^{-2}\bX \bg+\cO(\varepsilon), \\
\bLambda^{-1}\bX\bg\bc^T &= \bLambda \hat{\bX} +
\hat{\bLambda} \bX - \hat{\bX}\bX^{-1}\bLambda \bX+\cO(\varepsilon).
\end{aligned}
\right.
\end{equation}

In light of the proof of Lemma \ref{prop:1-conv}, it suffices to show
that 
\begin{equation}\label{cond:lemma:2-1}
\lVert \bU^{(1)}_\tau - \bU_\tau\rVert_\infty = \cO(\varepsilon).
\end{equation}
Note that $\bU^{(1)}$ and $B^{(0)}$ satisfy 
\begin{equation}\label{equ:eps1-form1}
\begin{pmatrix}\bU^{(1)} \\ B^{(0)} \\ 
\end{pmatrix}_t + 
\begin{pmatrix} \bA & \bg\\ 0 &
-\varepsilon\bc^T\bA^{-1}\bg \\ 
\end{pmatrix}
\begin{pmatrix}\bU^{(1)} \\B^{(0)} \\
\end{pmatrix}_x=0.
\end{equation}
Let 
\[
\bC_1^\varepsilon =\begin{pmatrix} \bA & \bg\\ 0 & -\varepsilon\bc^T\bA^{-1}\bg 
\\ \end{pmatrix}.
\]
Again by the perturbation theory \cite{wilkins1965eigenvalue},
$\bC^\varepsilon_1$ has the following decomposition
\begin{equation} \label{decom-1}
\bC^\varepsilon_1 = (\bK^\varepsilon_1)^{-1} \bD_1^\varepsilon
 \bK^\varepsilon_1,
\end{equation}
where 
\[
\bK^\varepsilon_1 = 
\begin{pmatrix}\bX+\varepsilon\hat{\bX}_1 & \bLambda^{-1}
\bX\bg + \varepsilon\hat{\balpha}_1\\ \varepsilon\hat{\bbeta}_1^T &
1+\varepsilon\hat{\theta}_1 
\end{pmatrix} \qquad
\bD_1^\varepsilon=\begin{pmatrix} \bLambda & 0\\
0 & -\varepsilon\bc^T\bA^{-1}\bg\\ 
\end{pmatrix}.
\]
%%Namely, 
%%\[
%%\begin{pmatrix}
%%\bLambda & 0 \\ 0 & -\varepsilon \bc^T\bA^{-1}\bg 
%%\end{pmatrix}
%%\begin{pmatrix}\bX+\varepsilon\hat{\bX}_1 & \bLambda^{-1}
%%\bX\bg + \varepsilon\hat{\balpha}_1\\ \varepsilon\hat{\bbeta}_1^T &
%%1+\varepsilon\hat{t}_1 
%%\end{pmatrix} = 
%%\begin{pmatrix}\bX+\varepsilon\hat{\bX}_1 & \bLambda^{-1}
%%\bX\bg + \varepsilon\hat{\balpha}_1\\ \varepsilon\hat{\bbeta}_1^T &
%%1+\varepsilon\hat{t}_1 
%%\end{pmatrix}  
%%\bC_1^{\varepsilon}.
%%\]
Collecting the $\cO(\varepsilon)$ terms in \eqref{decom-1}, we have
\begin{equation}\label{equ:X-rela-1}
\hat{\bX}_1=\bX \qquad 
\bLambda \hat{\balpha}_1 =
\bX\bg-\bLambda^{-1}\bX\bg\bc^T\bX^{-1}\bLambda^{-1}\bX\bg \qquad 
\hat{\bbeta}_1 = \bzero.
\end{equation}

Comparing with \eqref{equ:Ux-Bx} in Lemma \ref{prop:1-conv}, and from
\eqref{equ:X-rela-1} and the initial condition, we have 
\[
\begin{aligned}
& \left[ \bX(\bU - \bU^{(1)})_x + \bLambda^{-1}\bX\bg (B-B^{(0)})_x
\right]_k (x,t) \\
=~ & \left[(\bX+\varepsilon\hat{\bX})\bU_x + (\bLambda^{-1}\bX\bg +
    \varepsilon \hat{\balpha})B_x \right]_k (x-\lambda_k^\varepsilon t,
      0) \\
 & -\left[(\bX+\varepsilon\hat{\bX}_1) \bU^{(1)}_x +
  (\bLambda^{-1}\bX\bg + \varepsilon \hat{\balpha}_1)
  B^{(0)}_x\right]_k (x - \lambda_k t, 0) \\
& -\varepsilon (\hat{\bX}\bU_x + \hat{\balpha}B_x)_k(x,t) +
\varepsilon (\hat{\bX}_1\bU^{(1)}_x + \hat{\balpha}_1B_x^{(0)})_k(x,t)
+ \cO(\varepsilon^2) \\
=~& -\varepsilon (\hat{\bX}\bU_x + \hat{\balpha}B_x)_k(x,t) +
\varepsilon \hat{\bX}_1\bU^{(1)}_x + \hat{\balpha}_1B_x^{(0)})_k(x,t)
+ \cO(\varepsilon^2) \\
=~& \varepsilon(\bX \bU_x^{(1)} + \bLambda^{-1}\bX \bg
    B_x^{(0)})_k(x,t) + \varepsilon
(\bLambda^{-2}\bX\bg\bc^T\bX^{-1}\bLambda^{-1}\bX \bg)_k(B_x -
    B_x^{(0)})(x,t) \\
 &-\varepsilon \left[ 
 \hat{\bX}\bU_x + (\hat{\balpha} +
     \bLambda^{-2}\bX\bg\bc^T\bX^{-1}\bLambda^{-1} \bX\bg)B_x\right]_k(x,t)
  + \cO(\varepsilon^2)
\end{aligned}
\]
By the similar technique as Lemma \ref{prop:1-conv}, we can prove 
\[
\|\bX\bU_x^{(1)} + \bLambda^{-1}\bX\bg B_x^{(0)}\|_{\infty} =
\cO(\varepsilon).
\]
And Lemma \ref{prop:1-conv} proves that $\|B_x - B_x^{(0)}\|_{\infty}
= \cO(\varepsilon)$. In light of the \eqref{equ:perb-0}, we have the
following estimate for the last part:
\[
\begin{aligned}
& \hat{\bX}\bU_x + (\hat{\balpha} +
    \bLambda^{-2}\bX\bg\bc^T\bX^{-1}\bLambda^{-1} \bX\bg)B_x \\
=~ & \hat{\bX}\bA^{-1}(\bA\bU_x + \bg B_x) + (\hat{\balpha} +
    \bLambda^{-2} \bX \bg\bc^T\bX^{-1}\bLambda^{-1} \bX\bg -
    \hat{\bX}\bX^{-1}\bLambda^{-1}\bX\bg)B_x \\
=~ & \bLambda^{-1}(\hat{\bX} - \bLambda^{-1}\hat{\bLambda}\bX +
    \bLambda^{-1}\bX\bg\bc^T\bX^{-1}\bLambda^{-1} \bX -
    \bLambda\hat{\bX}\bX^{-1}\bLambda^{-1}\bX)\bg B_x +
\cO(\varepsilon) \\
=~ & \bLambda^{-1}\left[ 
\hat{\bX} - \bLambda^{-1}\hat{\bLambda}\bX + (\bLambda\hat{\bX} +
    \hat{\bLambda}\bX - \hat{\bX}\bA)\bX^{-1}\bLambda^{-1} \bX -
\bLambda\hat{\bX}\bX^{-1}\bLambda^{-1}\bX
\right]\bg B_x + \cO(\varepsilon) \\
=~& \cO(\varepsilon).
\end{aligned}
\]
Here, we have used the fact that $\hat{\bLambda}$ is diagonal.
From the above, we finally obtain 
\[
\lVert \bX (\bU-\bU^{(1)})_x+\bLambda^{-1}\bg
(B-B^{(0)})_x\rVert_\infty=\cO(\varepsilon^2),
\]
which implies that 
\[
\lVert \bA (\bU-\bU^{(1)})_x+\bg
(B-B^{(0)})_x\rVert_\infty=\cO(\varepsilon^2).
\]
Thus, we have
\[
\lVert \bU_\tau-\bU^{(1)}_\tau\rVert_\infty=\cO(\varepsilon).
\]
This completes the proof.
\end{proof}

\bibliographystyle{unsrt}
\bibliography{sediment.bib}
\end{document}